\documentclass{aic}

\usepackage{bbm}
\usepackage[numbers,sort&compress]{natbib}

\newtheorem{theorem}{Theorem}

\newtheorem{lemma}[theorem]{Lemma}
\newtheorem{corollary}[theorem]{Corollary}
\newtheorem{conjecture}[theorem]{Conjecture}

\let\leq\leqslant

\let\geq\geqslant

\DeclareMathOperator{\lb}{{\mu}}

\usepackage{subfig}

\graphicspath{{figures/}}

\makeatletter
\def\old@comma{,}
\catcode`\,=13
\def,{%
  \ifmmode%
    \old@comma\discretionary{}{}{}%
  \else%
    \old@comma%
  \fi%
}
\makeatother

\aicAUTHORdetails{%
  title = {Large Independent Sets in Triangle-free Cubic Graphs: Beyond Planarity}, 
  author = {Wouter Cames van Batenburg, Jan Goedgebeur, and Gwena\"el Joret},
  plaintextauthor = {Wouter Cames van Batenburg, Jan Goedgebeur, Gwenael Joret},
}   

\aicEDITORdetails{%
   year={2020},
   number={7},
   received={5 December 2019},   
   published={10 July 2020},  
   doi={10.19086/aic.13667},      
}   

\begin{document}

\begin{frontmatter}[classification=text]

\title{Large Independent Sets in Triangle-free Cubic Graphs: Beyond Planarity} 

\author[wouter]{Wouter Cames van Batenburg\thanks{Supported by an ARC grant from the Wallonia-Brussels Federation of Belgium.}}
\author[jan]{Jan Goedgebeur\thanks{Supported by a Postdoctoral Fellowship of the Research Foundation Flanders (FWO).}}
\author[gwen]{Gwena\"el Joret\thanks{Supported by an ARC grant from the Wallonia-Brussels Federation of Belgium.}}

\begin{abstract}
Every $n$-vertex planar triangle-free graph with maximum degree at most $3$ has an independent set of size at least $\frac{3}{8}n$. 
This was first conjectured by Albertson, Bollob\'as and Tucker, and was later proved by Heckman and Thomas. 
Fraughnaugh and Locke conjectured that the planarity requirement could be relaxed into just forbidding a few specific nonplanar subgraphs: 
They described a family $\mathcal{F}$ of six nonplanar graphs (each of order at most $22$) and conjectured that every $n$-vertex triangle-free graph with maximum degree at most $3$ having no subgraph isomorphic to a member of $\mathcal{F}$ has an independent set of size at least $\frac{3}{8}n$. 
In this paper, we prove this conjecture. 

As a corollary, we obtain that every $2$-connected $n$-vertex triangle-free graph with maximum degree at most $3$ has an independent set of size at least $\frac{3}{8}n$, with the exception of the six graphs in $\mathcal{F}$. 
This confirms a conjecture made independently by Bajnok and Brinkmann, and by Fraughnaugh and Locke.
\end{abstract}
\end{frontmatter}


\section{Introduction}\label{sec:intro}

All graphs in this paper are undirected, finite, and simple. 
A subset $X$ of vertices of a graph $G$ is {\em independent} if no two vertices in $X$ are adjacent in $G$. 
The {\em independence number} of $G$ is the maximum size of an independent set in $G$, denoted $\alpha(G)$. 
The graph $G$ is said to be {\em subcubic} if $G$ has maximum degree at most $3$, and {\em cubic} if $G$ is $3$-regular.  

One of the first results about independent sets in triangle-free subcubic graphs is the following theorem of Staton~\cite{S79} from 1979. 

\begin{theorem}[Staton~\cite{S79}]\label{th:Staton}
Let $G$ be a triangle-free subcubic $n$-vertex graph. 
Then, $\alpha(G) \geq \frac{5}{14}n$.
\end{theorem}

Different proofs of this result have appeared in the literature, see in particular Heckman and Thomas~\cite{HT01} for a short proof. 
The bound is best possible, as witnessed by the two cubic graphs on $14$ vertices in Figure~\ref{fig:forbidden_graphs} (top left and top center). 
In fact, these are the {\em only} tight examples among connected graphs~\cite{BB98,H08}. 
This suggests that a better bound might hold for connected triangle-free subcubic graphs when $n$ is not too small, and indeed Fraughnaugh and Locke~\cite{FL95} proved the following result in 1995.

\begin{theorem}[Fraughnaugh and Locke~\cite{FL95}]\label{th:FL}
Let $G$ be a connected triangle-free subcubic $n$-vertex graph. 
Then, $\alpha(G) \geq \frac{11}{30}n - \frac{2}{15}$.
\end{theorem}

\begin{figure}
\captionsetup[subfigure]{labelformat=empty}
\centering
   \subfloat[$F_{14}^{(1)}$]{\label{fig:F14_1}\includegraphics[width=0.179\textwidth]{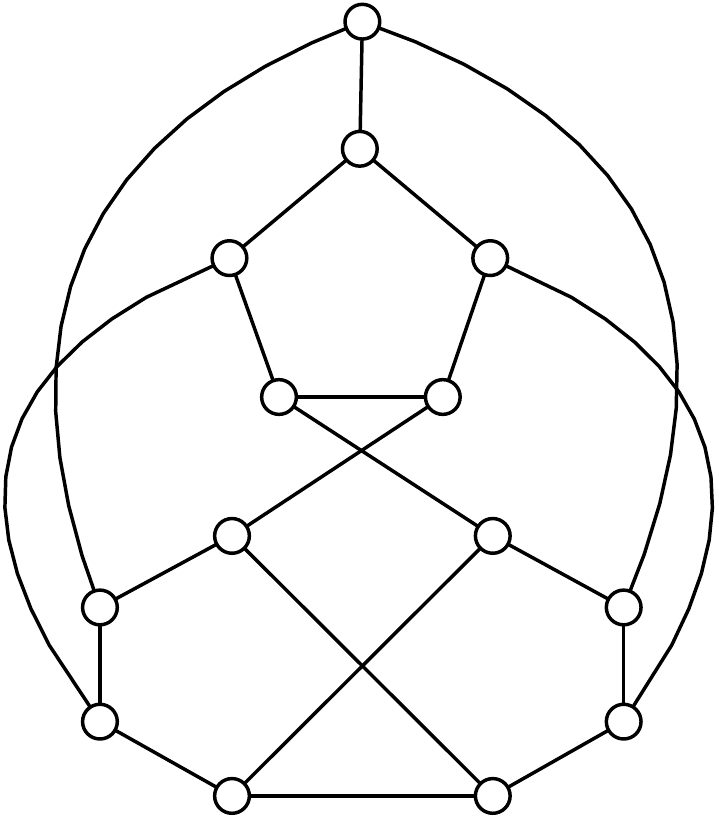}} \qquad
   \subfloat[$F_{14}^{(2)}$]{\label{fig:F14_2}\includegraphics[width=0.173\textwidth]{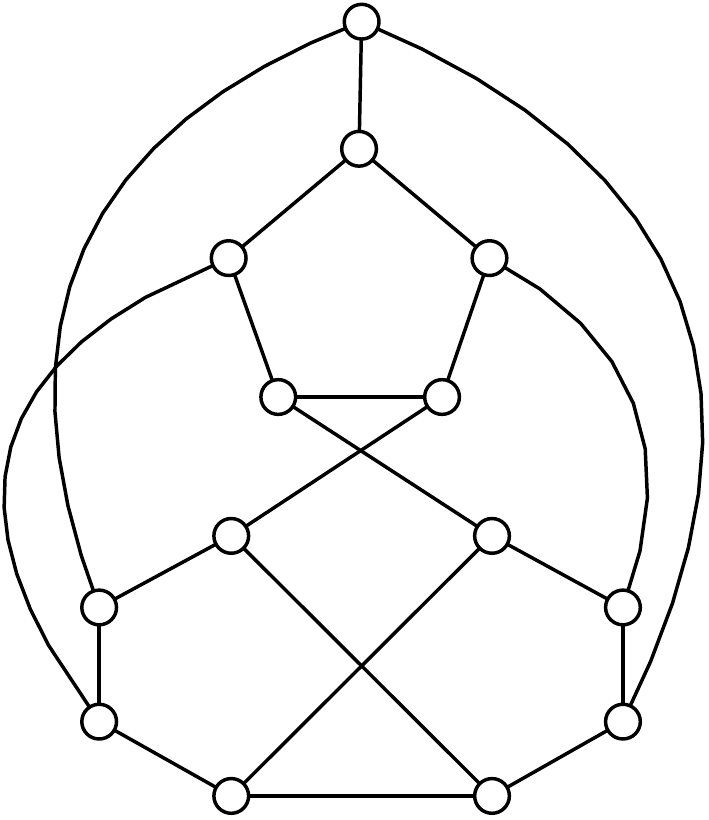}} \qquad
   \subfloat[$F_{22}$]{\label{fig:F22}\includegraphics[width=0.251\textwidth]{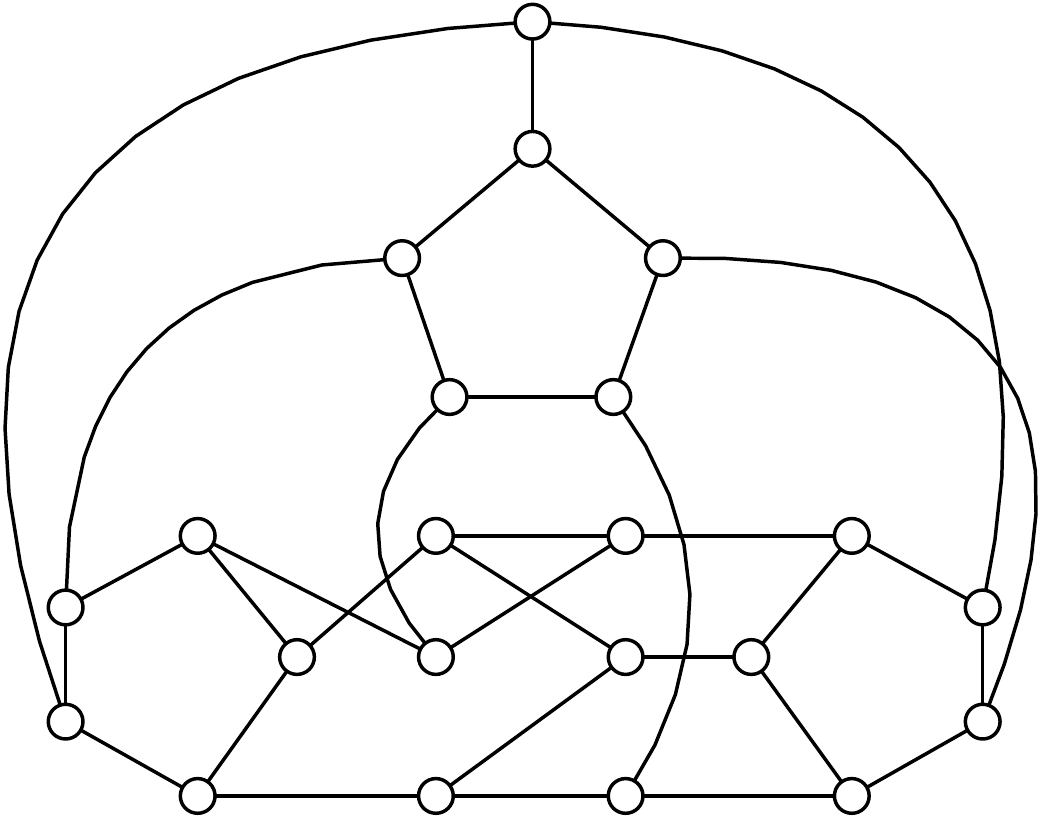}} \\
   \subfloat[$F_{11}$]{\label{fig:F11}\includegraphics[width=0.19\textwidth]{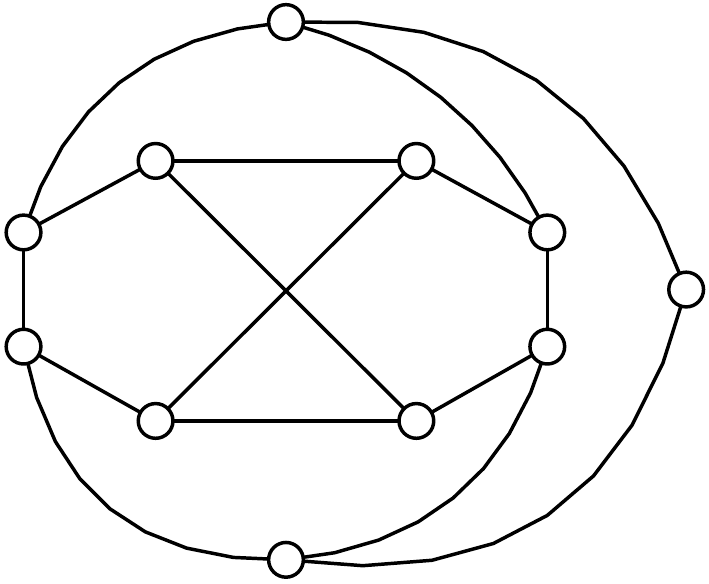}} \qquad
   \subfloat[$F_{19}^{(1)}$]{\label{fig:F19_1}\includegraphics[width=0.29\textwidth]{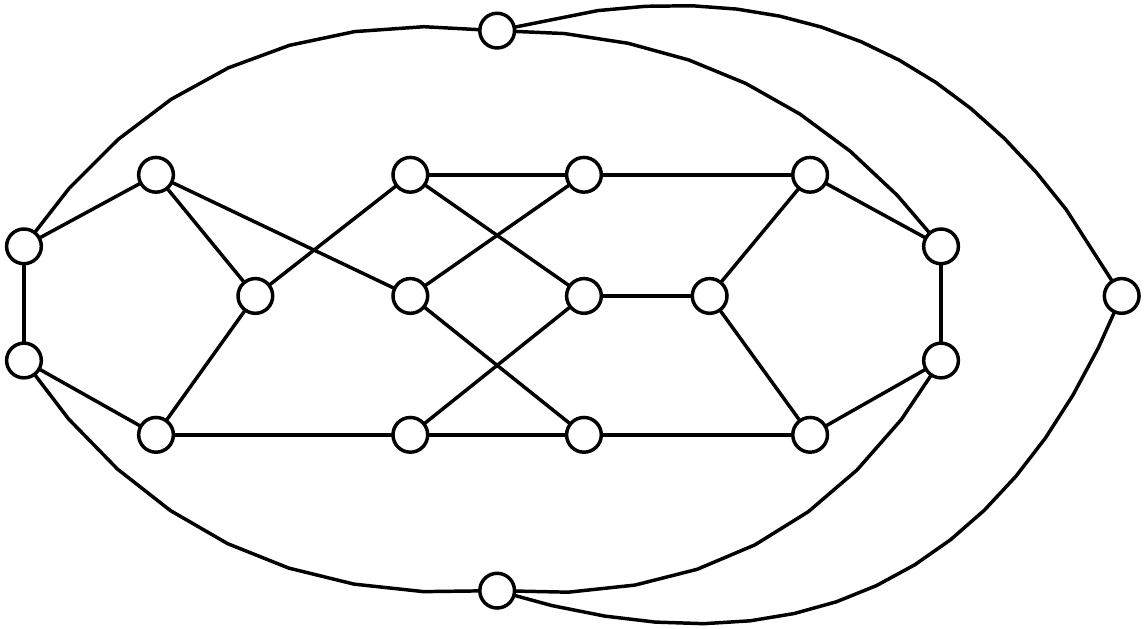}} \qquad
   \subfloat[$F_{19}^{(2)}$]{\label{fig:F19_2}\includegraphics[width=0.29\textwidth]{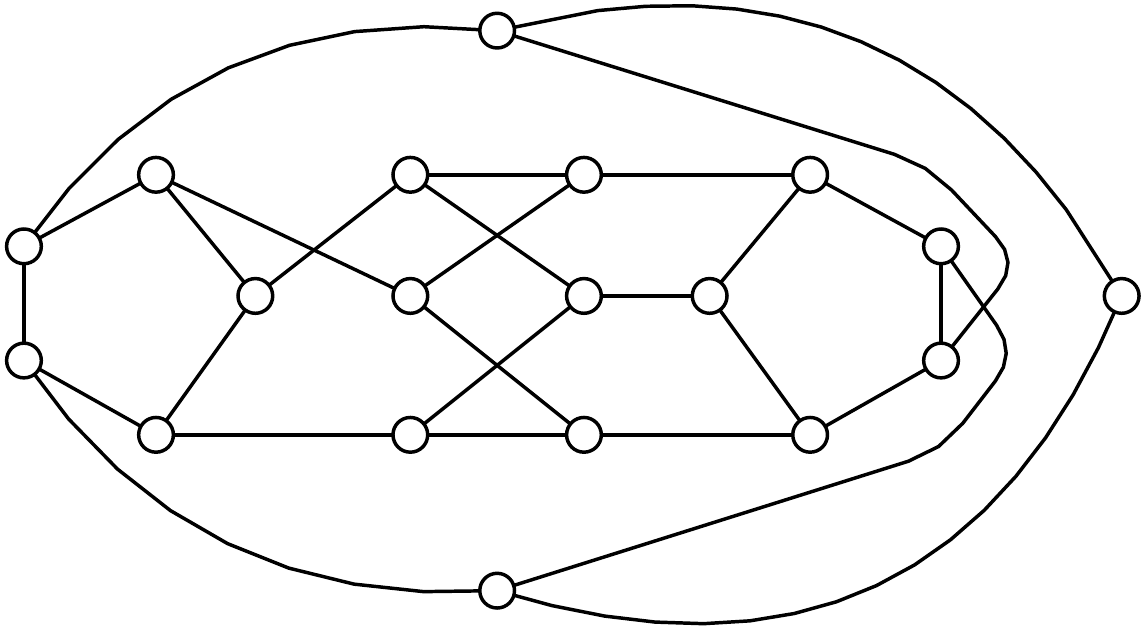}}   
\caption{The graphs $F_{14}^{(1)}$, $F_{14}^{(2)}$, $F_{22}$, $F_{11}$, $F_{19}^{(1)}$ and $F_{19}^{(2)}$.}
\label{fig:forbidden_graphs} 
\end{figure}

The factor $\frac{11}{30}$ is best possible, as shown by a construction of Fraughnaugh and Locke~\cite{FL95} illustrated in Figure~\ref{fig:FL_construction}. 
Note that this construction is far from being $2$-connected, thus it is natural to wonder whether the bound could be improved further under some extra connectivity assumption. 
In fact, already in 1986 Locke~\cite{L86} conjectured that there are only finitely many $3$-connected triangle-free cubic $n$-vertex graphs $G$ with $\alpha(G) < \frac{3}{8}n$. 
Fraughnaugh and Locke~\cite{FL95} made a similar conjecture under the assumption of $2$-connectivity: 

\begin{conjecture}[Fraughnaugh and Locke~\cite{FL95}]\label{conj:FL_2c} 
Let $G$ be a $2$-connected triangle-free subcubic $n$-vertex graph. 
Then, $\alpha(G) \geq \frac{3}{8}n - \frac{1}{4}$. 
\end{conjecture}

Bajnok and Brinkmann~\cite{BB98} investigated this conjecture using a computer search. 
While they found no counterexample in the range they considered, they also noted that they only found six $2$-connected triangle-free subcubic $n$-vertex graphs $G$ with $\alpha(G) < \frac{3}{8}n$, namely the graphs  $F_{11}$, $F_{14}^{(1)}$, $F_{14}^{(2)}$, $F_{19}^{(1)}$, $F_{19}^{(2)}$, $F_{22}$ from Figure~\ref{fig:forbidden_graphs}, which will be called the {\em forbidden graphs}.  
They conjectured that there are no other such graphs. 

\begin{conjecture}[Bajnok and Brinkmann~\cite{BB98}]\label{conj:BB}
There are exactly six 2-connected triangle-free subcubic $n$-vertex graphs $G$ with $\alpha(G) < \frac{3}{8}n$, namely $F_{11}$, $F_{14}^{(1)}$, $F_{14}^{(2)}$, $F_{19}^{(1)}$, $F_{19}^{(2)}$, $F_{22}$.
\end{conjecture}

\begin{figure}
\centering
\includegraphics[width=0.5\textwidth]{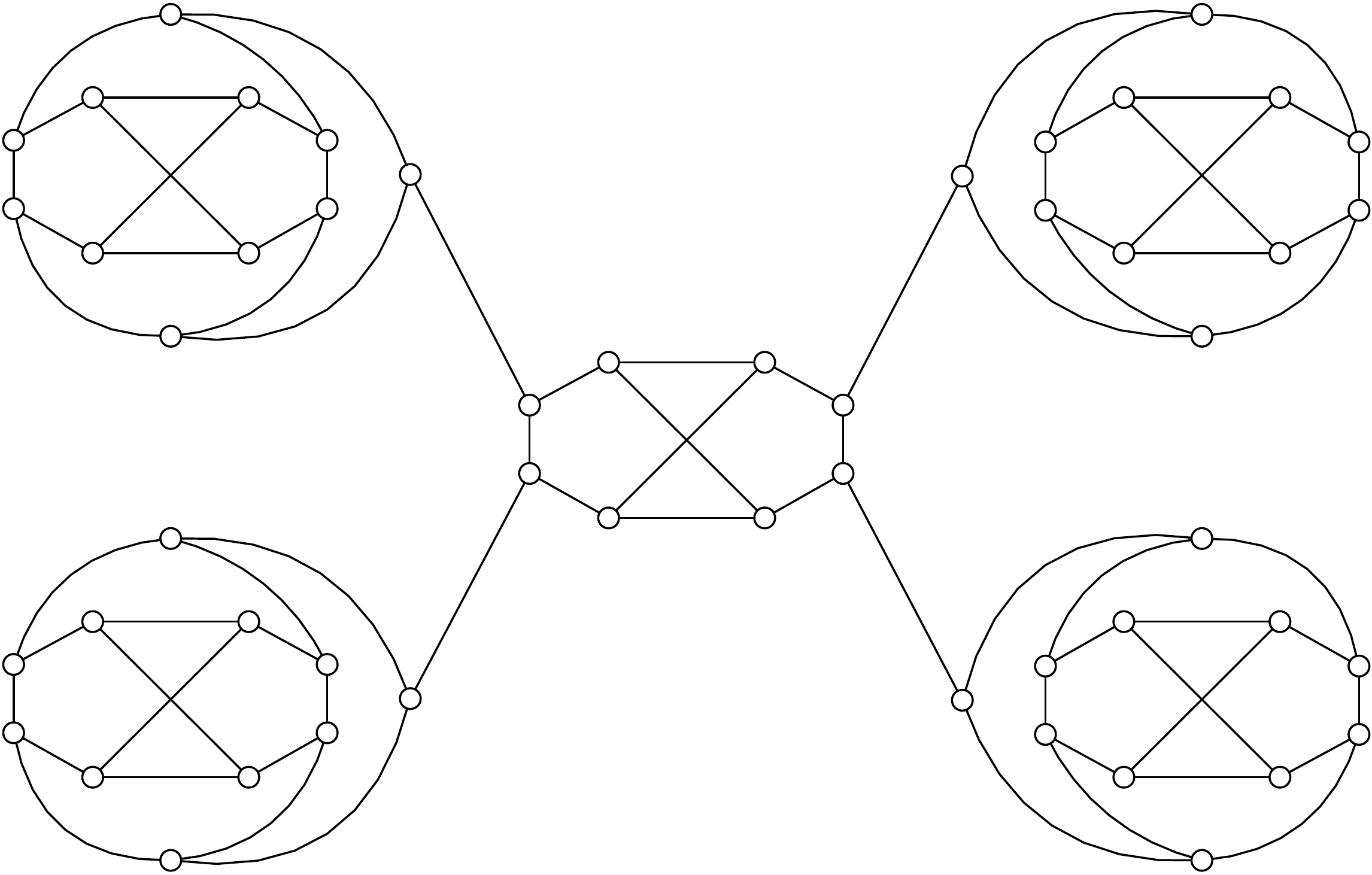}
\caption{Construction of connected triangle-free subcubic $n$-vertex graphs $G$ with $\alpha(G) = \frac{11}{30}n - \frac{1}{15}$: 
Start with a $4$-regular tree, replace each internal node with a copy of $B_8$ (see Figure~\ref{fig:bad_graphs})  and each leaf with a copy of $F_{11}$ (see Figure~\ref{fig:forbidden_graphs}), and link them using the degree-2 vertices. 
One can verify that for the resulting graph $G$, $\alpha(G)$ is three times the number of internal nodes plus four times the number of leaves. 
The construction is illustrated starting with a $4$-leaf star.}
\label{fig:FL_construction}
\end{figure}

Fraughnaugh and Locke~\cite{FL95}, aware of the six graphs found in~\cite{BB98}, formulated the following closely related conjecture.

\begin{conjecture}[Fraughnaugh and Locke~\cite{FL95}]\label{conj:FL}
If $G$ is a triangle-free subcubic $n$-vertex graph containing none of $F_{11}$, $F_{14}^{(1)}$, $F_{14}^{(2)}$, $F_{19}^{(1)}$, $F_{19}^{(2)}$, $F_{22}$ as a subgraph, then $\alpha(G) \geq \frac{3}{8}n$.  
\end{conjecture}

This conjecture implies Conjecture~\ref{conj:BB}, because if $G$ is $2$-connected and contains some graph $F$ among these six graphs as a subgraph then $G$ must be isomorphic to $F$. 
(Indeed, this is clear if $F$ is one of the three cubic graphs $F_{14}^{(1)}$, $F_{14}^{(2)}$, $F_{22}$, and if $F$ is one of the three subcubic graphs $F_{11}$, $F_{19}^{(1)}$, $F_{19}^{(2)}$ this is because they have a unique vertex of degree less than $3$, which would be a cutvertex of $G$ if $F$ were a proper subgraph of $G$.) 

Observe also that Conjecture~\ref{conj:FL} implies in particular that $\alpha(G) \geq \frac{3}{8}n$ holds for every planar triangle-free subcubic $n$-vertex graph $G$, since none of $F_{11}$, $F_{14}^{(1)}$, $F_{14}^{(2)}$, $F_{19}^{(1)}$, $F_{19}^{(2)}$, $F_{22}$ is planar. 
This was conjectured by Albertson, Bollob\'as, Tucker~\cite{ABT76} in 1976 and was still an open problem when the two papers~\cite{BB98,FL95} appeared. 
The conjecture was eventually proved by Heckman and Thomas~\cite{HT06} in 2006:

\begin{theorem}[Heckman and Thomas~\cite{HT06}]\label{th:planar}
Let $G$ be a triangle-free subcubic $n$-vertex planar graph on $n$ vertices. Then, $\alpha(G) \geq \frac{3}{8}n$. 
\end{theorem}

The goal of this paper is to show that the conjectures mentioned above are true:

\begin{theorem}[Main theorem]
\label{th:goal}
Let $G$ be a triangle-free subcubic $n$-vertex graph containing none of $F_{11}$, $F_{14}^{(1)}$, $F_{14}^{(2)}$, $F_{19}^{(1)}$, $F_{19}^{(2)}$, $F_{22}$ as a subgraph. Then, $\alpha(G) \geq \frac{3}{8}n$.  
\end{theorem}

\begin{corollary}\label{cor:twoconnected}
Let $G$ be a $2$-connected triangle-free subcubic $n$-vertex graph. 
Then, $\alpha(G) \geq \frac{3}{8}n - \frac{1}{4}$. 
Moreover, $\alpha(G) \geq \frac{3}{8}n$ if $G$ is not isomorphic to any of $F_{11}$, $F_{14}^{(1)}$, $F_{14}^{(2)}$, $F_{19}^{(1)}$, $F_{19}^{(2)}$, $F_{22}$.  
\end{corollary}

Both Theorem~\ref{th:goal} and Corollary~\ref{cor:twoconnected} are best possible in the following strong sense: There exist infinitely many $3$-connected girth $5$ graphs $G$ (containing none of $F_{11}$, $F_{14}^{(1)}$, $F_{14}^{(2)}$, $F_{19}^{(1)}$, $F_{19}^{(2)}$, $F_{22}$ as a subgraph) attaining $\alpha(G)= \frac{3}{8}n$. 
There are many ways to obtain such graphs, for example by taking a large collection of disjoint copies of the graph $B_{16}^{(2)}$ (see Figure~\ref{fig:bad_graphs}), and adding edges between degree-$2$ vertices of distinct copies, in such a way that no triangles or $4$-cycles are created and the resulting graph is $3$-connected. 
This is sufficient because $B_{16}^{(2)}$ has $16$ vertices and a size-$6$ maximum independent set avoiding all of its degree-$2$ vertices, which implies that the resulting graph satisfies $\alpha(G)= \frac{3}{8}n$.
 
Let us also mention a consequence for subcubic graphs with girth at least $6$. 
For $g \geq 3$, let
\[
i(g) := \inf \left\{ \frac{\alpha(G)}{|V(G)|}: G \textrm{ subcubic graph with girth at least } g  \right\}. 
\]
Then $i(3)=\frac{1}{4}$ and $i(4)=i(5)=\frac{5}{14}$, as follows from the discussion above. 
For $g=6$, the best known lower bound was $i(6) \geq \frac{11}{30}$, due to Pirot and Sereni~\cite{PS19}. 
Since $F_{11}$, $F_{14}^{(1)}$, $F_{14}^{(2)}$, $F_{19}^{(1)}$, $F_{19}^{(2)}$, $F_{22}$ each have a $5$-cycle, Theorem~\ref{th:goal} implies that $i(6) \geq \frac{3}{8}$. 

\begin{corollary}\label{cor:girth6}
Let $G$ be a triangle-free subcubic $n$-vertex graph with girth at least $6$. 
Then, $\alpha(G) \geq \frac{3}{8}n$. 
\end{corollary}

This is not far from best possible; there exist precisely three cubic girth $7$ graphs on $26$ vertices~\cite{HouseofGraphs}, two of which certify that $i(6)\leq i(7) \leq  \frac{3}{8} + \frac{1}{104}$. 
Using the generator for cubic graphs called \textit{minibaum}~\cite{brinkmann96}, we determined that there are no other graphs among the cubic graphs of girth at least 6 up to 40 vertices which meet or improve this bound. Using a modified version of the generator \textit{geng}~\cite{nauty-website,mckay_14}, we determined that there are no other graphs among the subcubic graphs of girth at least 6 up to 31 vertices which meet or improve this bound.
We also note that $\lim_{g\to +\infty} i(g)$ is a much studied quantity, which is known to be between 0.44533 (Cs{\'o}ka~\cite{C16}) and 0.454 (Balogh, Kostochka, and Liu~\cite{BKL17}).

\subsection*{Proof approach} 
Theorem~\ref{th:goal} is proved by induction on the size of the graph. 
In order to help the induction go through, we need to prove a slightly stronger version of the theorem. 
Before stating it, we introduce a few definitions. 

An edge $e$ of a graph $G$ is called \emph{critical} if $\alpha(G-e)> \alpha(G)$. A graph is called critical if each of its edges is critical. (Note that, in particular, $K_1$ is critical.)

For a graph $G$, we define 
$$\lb(G):= \frac{6|V(G)| -|E(G)| -\lambda}{12},$$ 
where $\lambda$ denotes the number of components of $G$. 
Equivalently, 
$$\lb(G)=\frac{1}{24}\left(9 n_3 +10 n_2 + 11 n_1 + 12 n_0  -2 \lambda \right),$$
where $n_i$ denotes the number of vertices of degree $i$ in $G$.

In order to state our main technical theorem, we also need to introduce two families of exceptional graphs. 
The first one is defined in terms of an {\em $8-$augmentation} operation, defined as follows. 
A {\em corner} of a subcubic graph $G$ is a triplet $(a, b, c)$ of vertices such that $a$ has degree $3$, $b$ and $c$ have degree $2$,  $ab, bc \in E(G)$ and $ac \notin E(G)$, and the unique neighbors of the three vertices outside $\{a, b, c\}$ are pairwise distinct; these neighbors are called the {\em interface} vertices.  
An $8-$augmentation of $G$ consists in replacing the three vertices of a corner of $G$ with a gadget on $11$ vertices as in Figure~\ref{fig:8-augmentation}. 
The size-$3$ matching between the three interface vertices and the three leftmost vertices of the gadget is chosen arbitrarily, thus different choices can produce different graphs. 
The $8-$augmentation operation was introduced explicitly in~\cite{HT06} and it also appears implicitly in one of the proofs in~\cite{FL95}.   

\begin{figure}[h]
  \centering
    \includegraphics[width=0.7\textwidth]{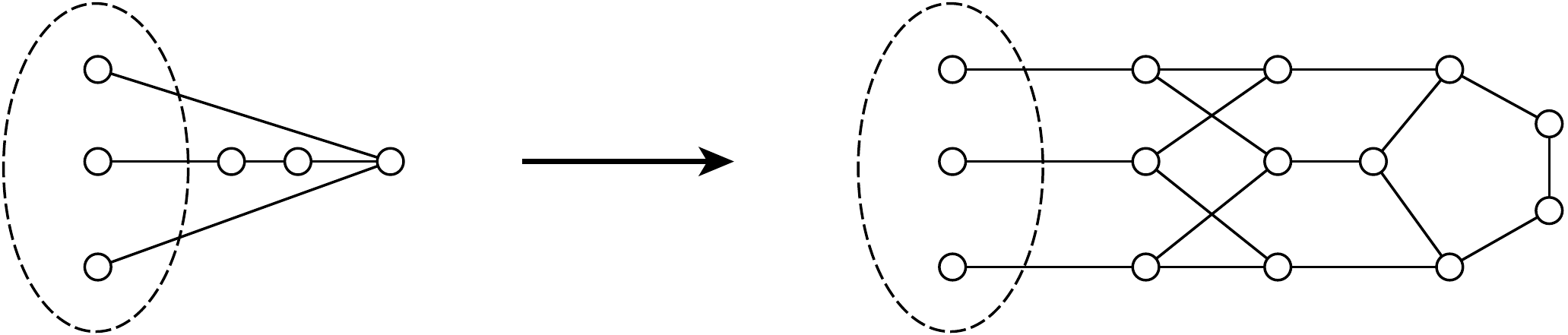}
    \caption{The $8-$augmentation operation.}
    \label{fig:8-augmentation}
\end{figure}

{\em Bad} graphs are defined inductively as follows: The graph $B_8$ (defined in Figure~\ref{fig:bad_graphs}) is bad, and every $8-$augmentation of a bad graph is bad. 
The three smallest bad graphs are illustrated in Figure~\ref{fig:bad_graphs}. 

In order to prove Theorem~\ref{th:goal}, it is enough to consider the case where $G$ is connected and critical.  
Ideally, we would like to show that $\alpha(G) \geq \lb(G)$ holds under these hypotheses; the fact that vertices of degree less than $3$ contribute more to the lower bound helps the induction go through. 
However, this is not true because of bad graphs: An $n$-vertex bad graph $G$ is connected, critical, has four degree-$2$ vertices, and yet $\alpha(G) = \frac{3}{8}n$, that is, bad graphs are tight examples for Theorem~\ref{th:goal}. 
In terms of $\lb(G)$, bad graphs $G$ satisfy $\alpha(G) = \lb(G) - \frac{1}{12}$. 
As it turns out, bad graphs form an exceptional family of graphs, $\alpha(G) \geq \lb(G)$ will always hold for non-bad graphs, as stated in Theorem~\ref{th:main} below.

\begin{figure}
\captionsetup[subfigure]{labelformat=empty}
\centering
   \subfloat[$B_{8}$]{\label{fig:B8}\includegraphics[width=0.19\textwidth]{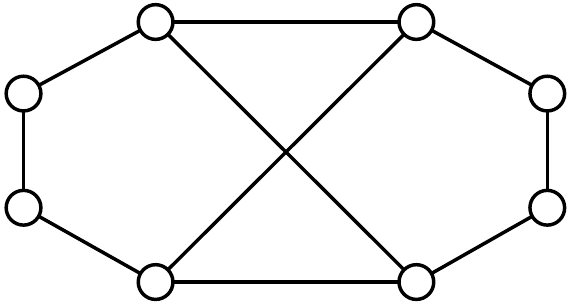}} \\[1ex]
   \subfloat[$B_{16}^{(1)}$]{\label{fig:B16_1}\includegraphics[width=0.33\textwidth]{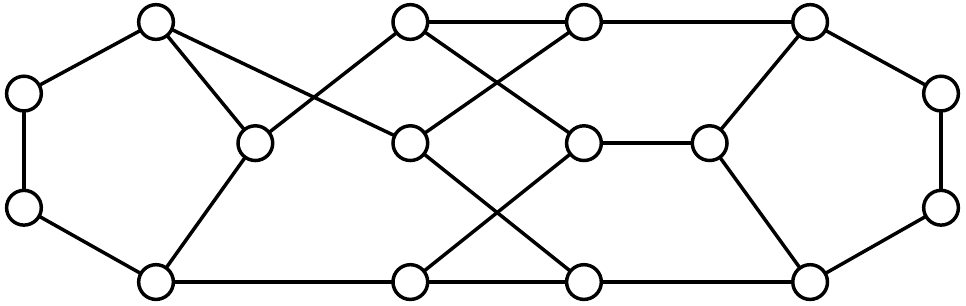}} 
   \quad \quad \quad \quad \quad \quad \quad \quad 
   \subfloat[$B_{16}^{(2)}$]{\label{fig:B16_2}\includegraphics[width=0.33\textwidth]{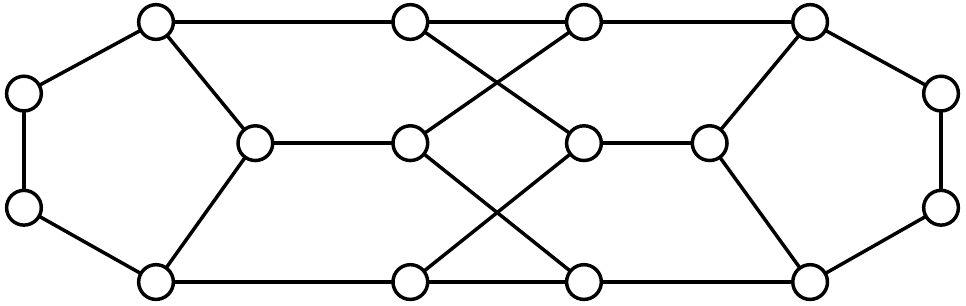}}      
\caption{The bad graphs $B_{8}$, $B_{16}^{(1)}$, and $B_{16}^{(2)}$.  
\label{fig:bad_graphs}}
\end{figure}

While bad graphs need to be treated separately, there is a second family of graphs that needs to be singled out.  
Suppose that a triangle-free subcubic graph $D$ has a vertex $u$ of degree $2$, both of whose neighbors $v_1,v_2$ have degree $3$. Suppose further that the graph $H:=D - \left\{u,v_1,v_2\right\}$ has exactly one cut-edge $e$, and let $D_1,D_2$ be the components of $H - e$.  Then we will call $D$ the {\em sum of $D_1$ and $D_2$} if for each $i\in \left\{1,2\right\}$,
\begin{itemize}
\item $v_i$ has exactly two neighbors in $H$;
\item $D_i$ has precisely five vertices of degree $2$.
\item $V(D_i)$ contains precisely two vertices that have degree $2$ in $D$, and these vertices are adjacent.
\end{itemize} 

A {\em join} of two bad graphs $H_1$ and $H_2$ consists in taking the disjoint union of $H_1$ and $H_2$, choosing a corner $(a_i, b_i, c_i)$ of $H_i$ for $i=1,2$, then removing the vertices $b_1, b_2, c_2$, and finally adding the edges $c_1a_2$ and $a_1y_2$, where $y_2$ is the neighbor of $c_2$ in $H_2$ that is distinct from $b_2$. 
(For instance, starting with two copies of $B_8$, one obtains the top left graph in Figure~\ref{fig:dangerous}.) 

{\em Dangerous graphs} are defined inductively as follows:
\begin{itemize}
  \item $C_5$ is dangerous; 
  \item every graph which is the sum of two dangerous graphs is also dangerous;
  \item every $8$-augmentation of a dangerous graph is dangerous; 
  \item every join of two bad graphs is dangerous. 
\end{itemize}

We remark that dangerous graphs correspond essentially to the `difficult graphs' in~\cite{HT06}.\footnote{To be precise, difficult graphs are defined in~\cite{HT06} using only the first three items from the definition of dangerous graphs.} 
Figure~\ref{fig:dangerous} shows the five dangerous graphs that can be obtained using the sum operation on two copies of $C_5$. 
While these five graphs are critical, we note that in general dangerous graphs need not be. 
(However, in the proofs we will mostly deal with dangerous graphs that are also critical.)

\begin{figure}
\centering
\includegraphics[width=1\textwidth]{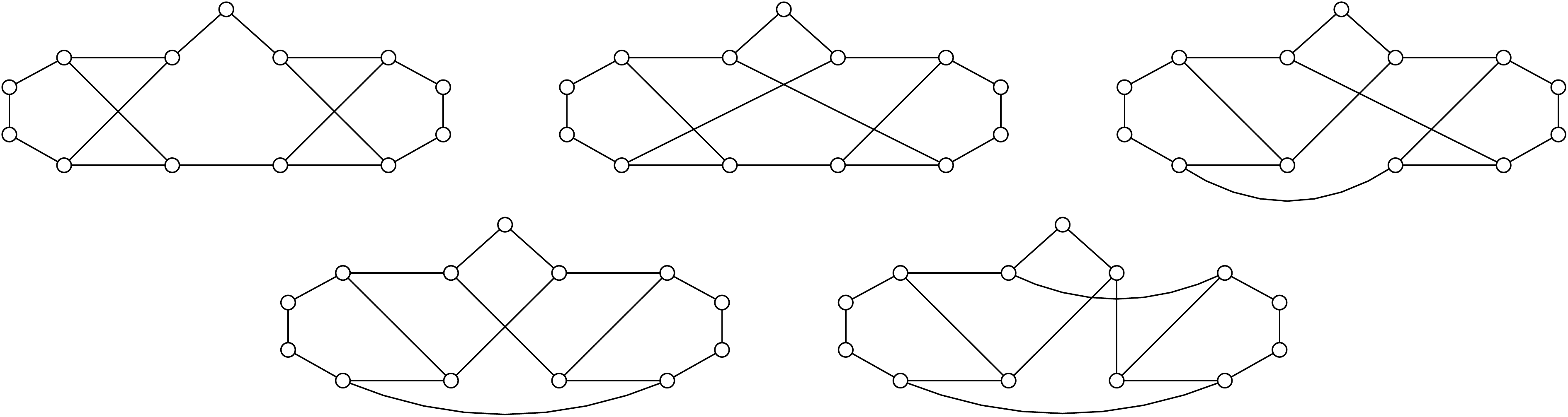}
\caption{Five dangerous graphs.}
\label{fig:dangerous}
\end{figure}

\begin{theorem}[Main technical theorem]\label{th:main}
Let $G$ be a connected critical triangle-free subcubic graph which is not isomorphic to any of $F_{11}$, $F_{14}^{(1)}$, $F_{14}^{(2)}$, $F_{19}^{(1)}$, $F_{19}^{(2)}$, $F_{22}$. 
Then 
\begin{itemize}
	\item $\alpha(G) = \lb(G) - \frac{1}{12}$ if $G$ is bad \\[0.1ex]
	\item $\alpha(G) \geq \lb(G)$ otherwise.  
\end{itemize}
Furthermore, if $G$ has at least three vertices of degree 2, $G$ is not dangerous, and $G$ has no bad subgraph, then $\alpha(G) \geq \lb(G) + \frac{1}{12}$. 
\end{theorem}

Theorem~\ref{th:goal} follows from Theorem~\ref{th:main}, as we now explain.

\begin{proof}[Proof of Theorem~\ref{th:goal}]
The proof is by induction. 
We may assume that $G$ is connected, since otherwise we are done by induction. 
We are similarly done by induction if $\alpha(G-e)=\alpha(G)$ holds for some edge $e\in E(G)$. 
Thus we may suppose that $G$ is critical. 
If $G$ is bad then it is easily checked from the inductive definition of bad graphs that $\lb(G) = \frac{3}{8}n + \frac{1}{12}$, and thus $\alpha(G) = \lb(G) - \frac{1}{12} = \frac{3}{8}n$ by Theorem~\ref{th:main}.  
If $G$ is not bad, then Theorem~\ref{th:main} gives
\[
\alpha(G) \geq \lb(G) \geq \frac{3}{8}n - \frac{1}{12}, 
\]
which can be rewritten as $3(8\alpha(G)- 3n) \geq -2$. 
However, since $\alpha(G)$ and $n$ are both integers, it follows that $3(8\alpha(G)- 3n) \geq 0$, i.e.\ $\alpha(G) \geq \frac{3}{8}n$, as desired. 
\end{proof}

Here is a very brief outline of the proof of Theorem~\ref{th:main}.  
Arguing by contradiction, we consider a minimum counterexample $G$ to Theorem~\ref{th:main}. 
The fact that $G$ is connected and critical implies that $G$ is $2$-connected. 
We first show that $G$ is almost $3$-connected in the following sense: 
If $X$ is cutset of $G$ of size $2$, then $G-X$ has exactly two components, one of which is isomorphic to $K_1$ or $K_2$. 

Next, we prove that no subgraph of $G$ is bad. 
Our goal after that is to show that there is no vertex of degree $2$ in $G$. 
To do so, we analyze the degrees of the two neighbors of an hypothetical degree-$2$ vertex. 
First, we show that they cannot both have degree $2$, then that they cannot both have degree $3$, and finally that they cannot have degree $2$ and $3$. 

At this point, we conclude that $G$ is cubic and in fact $3$-connected. Successively, we prove that $G$ has no $4$-cycle and no $6$-cycle. Then we show that $G$ has no $5$-cycle either. 

The fact that $G$ has girth at least $7$ implies in turn that no subgraph of $G$ is dangerous. 
We then derive a final contradiction by studying the local structure around a shortest even cycle in $G$ (which exists because $G$ is cubic), exploiting heavily the nonexistence of bad or dangerous subgraphs. 

We note that some of the steps in the outline of the proof above appear also in~\cite{FL95} or in~\cite{HT06}. 
However, the proofs are different and typically longer in our setting.   
To give just one illustration of this phenomenon, when trying to rule out the existence of a bad subgraph in $G$, one first shows that each of the four degree-$2$ vertices in a bad subgraph $B$ needs to send an edge out of $B$, and that $G-V(B)$ is connected. 
This directly implies that $G$ has a $K_5$-minor, so in the planar case we can just stop there. 
In our setting however we need to keep studying the local structure around $B$ in order to eventually find one of the forbidden subgraphs.

\subsection*{Additional results} 

Using Theorem~\ref{th:main} we recover Theorem~\ref{th:FL} of Fraughnaugh and Locke~\cite{FL95} but with a slight improvement of the additive constant, matching exactly the bound of the construction in Figure~\ref{fig:FL_construction}. 
This answers a small question from~\cite{FL95}.

\begin{theorem}\label{th:trianglefreeconnected}
Let $G$ be a connected triangle-free subcubic $n$-vertex graph. Then, $\alpha(G) \geq \frac{11}{30}n - \frac{1}{15}$, 
unless $G$ is isomorphic to $F_{14}^{(1)}$ or $F_{14}^{(2)}$, in which case $\alpha(G)= \frac{11}{30}n - \frac{2}{15}$. 
\end{theorem}
 
Theorem~\ref{th:trianglefreeconnected} is proved in Section~\ref{se:fraughnaughlocke}. 

In Section~\ref{sec:triangles}, we point out that another simple application of Theorem~\ref{th:main} gives a variant of Theorem~\ref{th:goal} where triangles are allowed. 
The lower bound on $\alpha(G)$ is then decreased by some function of the maximum number of disjoint triangles in the graph and of the maximum number of bad or `almost bad' disjoint subgraphs; see Corollary~\ref{cor:connectedwithtriangles} for the precise statement. 

\subsection*{Paper organization} 

In Section~\ref{sec:prelim} we review some standard results about critical graphs. 
In Section~\ref{sec:bad_dangerous} we consider properties of bad graphs and dangerous graphs. 
Section~\ref{sec:proof}, which is the bulk of this paper, is devoted to the proof of Theorem~\ref{th:main}.  
The proof of Theorem~\ref{th:trianglefreeconnected} is given in Section~\ref{se:fraughnaughlocke}. 
The case where triangles are allowed is briefly discussed in Section~\ref{sec:triangles}. 
Finally, we conclude the paper in Section~\ref{sec:conclusion} with some open problems.

\section{Critical graphs}
\label{sec:prelim}

In this section we introduce a few standard lemmas about critical graphs that will be needed in our proofs. 
We omit the proofs here, these results appear e.g.\ in Lov\'asz's textbook {\em Combinatorial problems and exercises}~\cite{L93} and are easy to prove (with the exception of Lemma~\ref{le:gluing_critical} whose proof is a bit tedious).

A \emph{clique cutset} is a cutset that induces a complete graph.
\begin{lemma}\label{le:nocliquecutset}
A connected critical graph has no clique cutset.
\end{lemma}

In particular, the above lemma implies the following: 

\begin{lemma}\label{le:twoconnected1}
Every connected critical graph distinct from $K_1$ and $K_2$ is $2$-connected.
\end{lemma}

Subdividing an edge twice keeps criticality: 

\begin{lemma}\label{le:oddsubdivision}
Let $G'$ be a graph with an edge $ad$. Let $G$ be the graph obtained from $G'$ by replacing the edge $ad$ with a path $abcd$, where $b,c$ are two new vertices. Then $\alpha(G)=\alpha(G')+1$. Moreover, $G$ is critical if and only if $G'$ is. 
\end{lemma}

The above lemma is a special case of the following {\em gluing} operation for critical graphs. 
This operation explains how all non $3$-connected critical graphs are built from smaller critical graphs, see Figure~\ref{fig:gluing} for an illustration. 
Recall that if $\{x, y\}$ is a $2$-cutset of a critical graph, then $x$ and $y$ are not adjacent, by Lemma~\ref{le:nocliquecutset}.

\begin{lemma}\label{le:gluing_critical}
Let $G_0, G_1$ be two connected critical graphs distinct from $K_1$ and $K_2$ on disjoint vertex sets. 
Let $xy \in E(G_0)$ and let $v\in V(G_1)$. 
Let $G$ be obtained by first taking the disjoint union of $G_0 - xy$ and $G_1 - v$ and then making each neighbor of $v$ in $G_1$ adjacent to exactly one of $x,y$, in such a way that $x, y$ are each chosen at least once. 
Then $G$ is critical and $\alpha(G) = \alpha(G_0) + \alpha(G_1)$. 

Conversely, if $G$ is a connected critical graph and $\{x, y\}$ is a $2$-cutset of $G$, then $G$ can be obtained from two connected critical graphs $G_0, G_1$ distinct from $K_1,K_2$  in this way.    
That is, $G - \{x, y\}$ has exactly two components, and they can be denoted $C_0, C_1$ in such a way that every vertex of $C_1$ has at most one neighbor in $\{x, y\}$, and that the graphs $G_0$ and $G_1$ are critical, where $G_0$ is obtained from $G[V(C_0) \cup \{x,y\}]$ by adding the edge $xy$, and $G_1$ is obtained from $G[V(C_1) \cup \{x,y\}]$ by identifying $x$ and $y$. 
\end{lemma}

\begin{figure}
\centering
\includegraphics[width=0.85\textwidth]{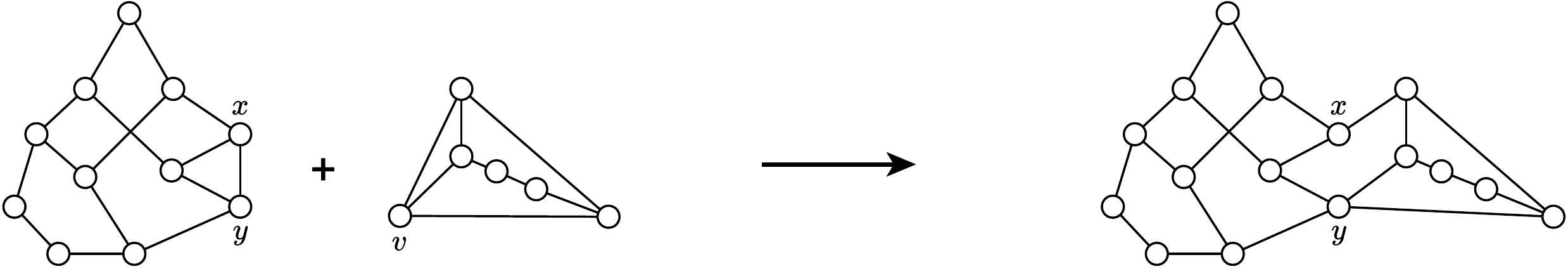}
\caption{Gluing two critical graphs.}
\label{fig:gluing}
\end{figure}

Regarding the second part of Lemma~\ref{le:gluing_critical}, we remark that if some component of $G - \{x, y\}$ has a vertex adjacent to both $x$ and $y$, then that component must be $C_0$ (and the other $C_1$). 
For instance, this happens for the graph in Figure~\ref{fig:gluing} (right). 
Thus, in such a case we know in advance which decomposition will be given to us by the lemma. 
This observation will be used in the proofs.

We conclude this section with an easy observation about $4$-cycles in critical graphs. 

\begin{lemma}\label{le:nodegreetwoinfourcycle}
In a critical graph $G$, no vertex of degree $2$ lies on a $4$-cycle.
\end{lemma}
\begin{proof}
Suppose that $abcd$ is a $4$-cycle and $a$ has degree two in $G$. Since in particular the edge $bc$ is critical, every maximum independent set $S$ of $G-bc$ contains both $b$ and $c$ and therefore contains neither $a$ nor $d$. But then $S \cup \left\{a\right\} - \left\{b\right\}$ is an independent set of $G$ of size $|S|=\alpha(G)+1$, which is a contradiction. 
\end{proof}

\section{Bad graphs and dangerous graphs}
\label{sec:bad_dangerous}

In this section we list some useful properties of bad graphs and dangerous graphs. 
We begin with a lemma about $8-$augmentations. 

\begin{lemma}\label{le:8augmentation} 
Let $G$ be an $8-$augmentation of a subcubic graph $G'$. 
Then, $\alpha(G)=\alpha(G')+3$, $\lb(G)=\lb(G')+3$, and $G$ is $2$-connected if and only if $G'$ is. 
\end{lemma}
\begin{proof}
It is immediate that $\lb(G)=\lb(G')+3$ and that $G$ is $2$-connected if and only if $G'$ is.  
Let us show that $\alpha(G)=\alpha(G')+3$.  
Let $(a, b, c)$ denote the corner of $G'$ used in the $8-$augmentation operation when producing $G$. 
Consider the $11$-vertex {\em gadget} that replaced these three vertices in $G$, as drawn in Figure~\ref{fig:8-augmentation}. 
We classify its vertices into three {\em leftmost vertices}, three {\em middle vertices}, and five {\em rightmost vertices} (the latter inducing $C_5$), in the obvious way. 
The gadget has independence number $5$ and every size-$5$ independent set must contain the three leftmost vertices. 

Let $S$ be a maximum independent set of $G$. 
If $S$ contains five vertices of the gadget, then it contains the three leftmost vertices, and we can replace the five vertices from the gadget with $a$ and $c$ to obtain an independent set of $G'$. 
If $S$ contains at most four vertices of the gadget, then we can replace them with vertex $b$ and obtain an independent set of $G'$. 
In both cases we deduce that $\alpha(G') \geq \alpha(G) - 3$. 

Now, let $S'$ be a maximum independent set of $G'$.  
Observe that $S'$ contains either one or two vertices from $\{a,b,c\}$. 
If $S'$ contains only one, then removing that vertex and adding the three middle vertices plus one of the rightmost degree-$2$ vertices of the gadget gives an independent set of $G$. 
If $S'$ contains two vertices from $\{a,b,c\}$ then these vertices are $a$ and $c$.  
Replacing them with the three leftmost vertices plus two independent rightmost vertices of the gadget gives an independent set of $G$. 
In both cases we deduce that $\alpha(G) \geq \alpha(G') + 3$. 

Combining the two inequalities it follows that $\alpha(G) = \alpha(G') + 3$, as claimed. 
\end{proof}

There are two bad graphs on $16$ vertices, $B_{16}^{(1)}$ and $B_{16}^{(2)}$; see Figure~\ref{fig:bad_graphs}.  Even though these two graphs look almost the same, only $B_{16}^{(1)}$ behaves similarly to $B_8$, in the sense of property~\ref{le:badproperties_5} of the following lemma.

\begin{figure}
\captionsetup[subfigure]{labelformat=empty}
\centering
   \subfloat{\label{fig:B24_a}\includegraphics[width=0.4\textwidth]{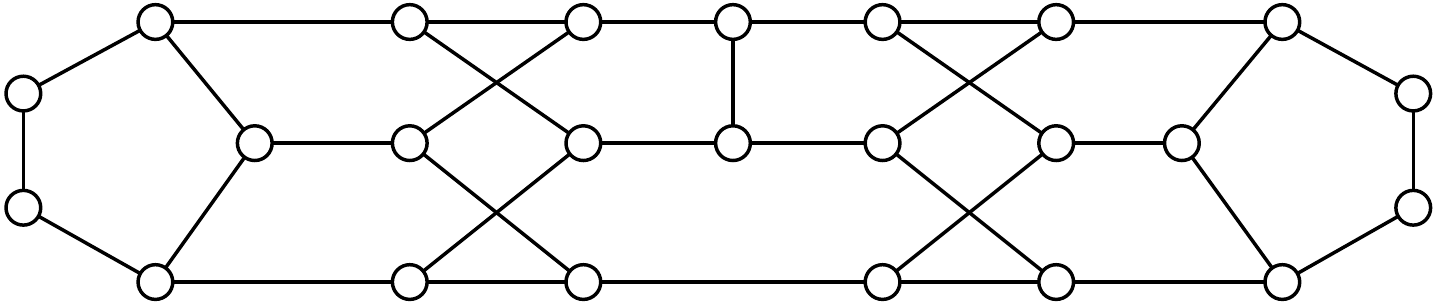}} \qquad
   \subfloat{\label{fig:B24_b}\includegraphics[width=0.4\textwidth]{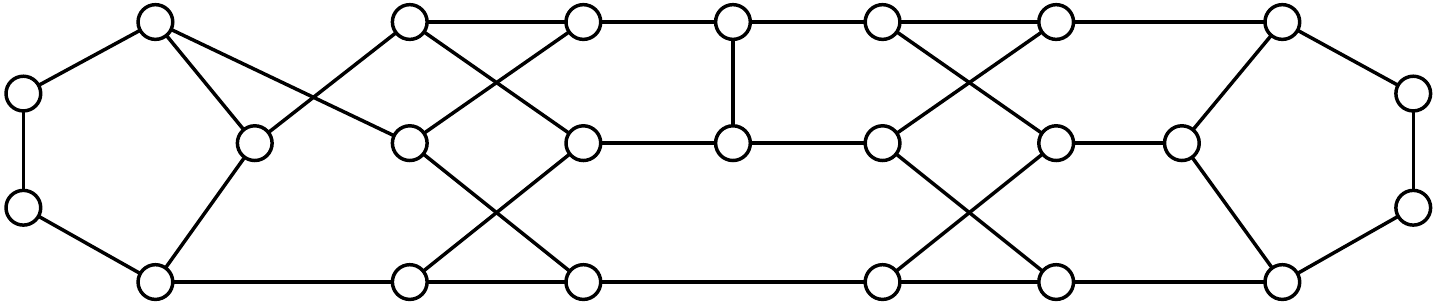}}

   \subfloat{\label{fig:B24_c}\includegraphics[width=0.4\textwidth]{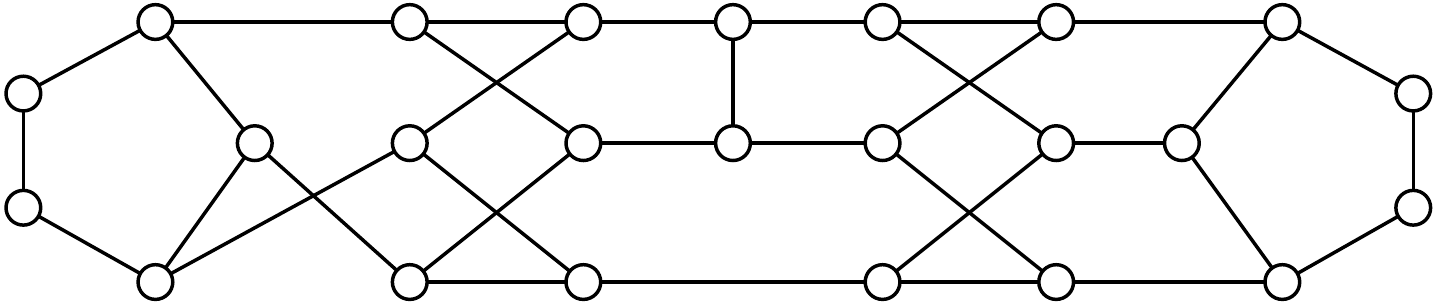}} \qquad
   \subfloat{\label{fig:B24_d}\includegraphics[width=0.4\textwidth]{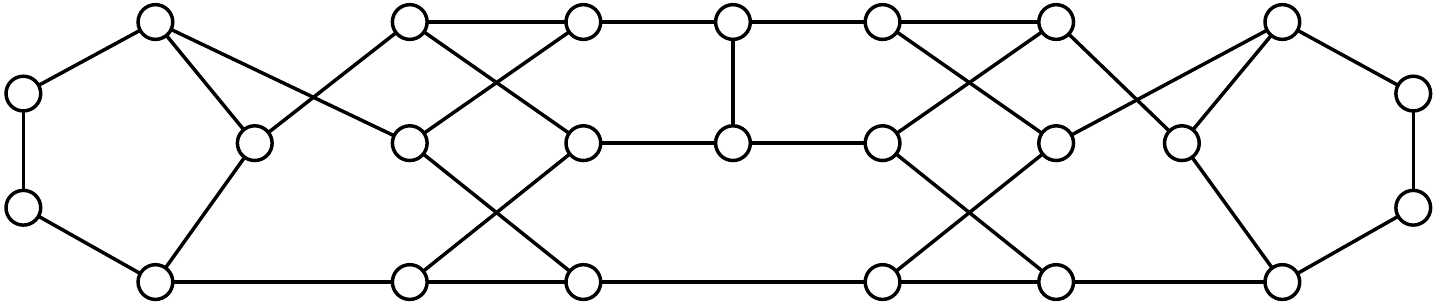}}
\caption{All four bad graphs on $24$ vertices.} 
   \label{fig:bad_24}
\end{figure}

\begin{lemma}\label{le:badproperties} 
The following properties hold for every bad graph $B$.  
\begin{enumerate}
\item \label{le:badproperties_1} $B$ is triangle-free, subcubic, and $2$-connected;  
\item \label{le:badproperties_2} $B$ has minimum degree $2$ and has exactly four degree-$2$ vertices;
\item \label{le:badproperties_3} The degree-$2$ vertices of $B$ induce a size two matching;
\item \label{le:badproperties_alpha_mu} $\alpha(B)=\lb(B) - \frac{1}{12}$; 
\item \label{le:badproperties_5} Let $T\subseteq V(B)$ denote a set of three degree-$2$ vertices. Then $B$ has a maximum independent set that avoids $T$.  Furthermore, $B$ has a maximum independent set avoiding all of its degree-$2$ vertices if and only if $B$ is neither $B_8$ nor $B_{16}^{(1)}$; 
\item \label{le:badproperties_6} If $B$ is not isomorphic to $B_8$ nor to $B_{16}^{(1)}$ then no edge of $B$ is contained in all $6$-cycles of $B$. 
\end{enumerate}
\end{lemma}
\begin{proof}
Note that the lemma is true if $B$ is the graph $B_8$. Therefore properties~\ref{le:badproperties_1} and~\ref{le:badproperties_alpha_mu} follow inductively from Lemma~\ref{le:8augmentation}, and the fact that $8-$augmentations preserve triangle-freeness.  
Properties~\ref{le:badproperties_2} and~\ref{le:badproperties_3} also follow straightforwardly, because they are invariant under taking an $8-$augmentation. 

It can be checked that property~\ref{le:badproperties_5} holds true for $B_8$ and $B_{16}^{(1)}$, as well as for each graph that can be obtained from $B_8$ or $B_{16}^{(1)}$ by a single $8-$augmentation:  
Up to isomorphism, there are exactly two graphs that are $8-$augmentations of $B_8$, namely $B_{16}^{(1)}$ and $B_{16}^{(2)}$, and exactly four graphs that are $8-$augmentations of $B_{16}^{(1)}$, see Figure~\ref{fig:bad_24}. 
It happens that every $8-$augmentation of $B_{16}^{(2)}$ is also an $8-$augmentation of $B_{16}^{(1)}$, hence the four graphs in Figure~\ref{fig:bad_24} constitute the full list of bad graphs on $24$ vertices. 

Thus it suffices to show that the property of having a maximum independent set avoiding all degree-$2$ vertices is preserved when taking an $8-$augmentation. 
Suppose that $B$ is an $8-$augmentation of a bad graph $B'$ that has a maximum independent set $S'$ avoiding all of its degree-$2$ vertices. 
Let $(a,b,c)$ denote the corner of $B'$ used in the $8-$augmentation operation. 
(Recall that $a$ has degree $3$ in $B'$, while $b$ and $c$ have degree $2$ in $B'$.) 
Then $b, c \notin S'$ since $S'$ avoids all degree-$2$ vertices of $B'$, and thus $a\in S'$ (otherwise $S'$ is not maximal). 
Reusing the same terminology as in the proof of Lemma~\ref{le:8augmentation}, select the two leftmost vertices of the gadget that together have in $B$ the same neighbors outside the gadget as vertex $a$ in $B'$. 
Then take also the two independent rightmost degree-$3$ vertices of the gadget. 
Replacing vertex $a$ in $S'$ with the four selected vertices gives an independent set of $B$ of size $|S'|+3=\alpha(B)$ that avoids all degree-$2$ vertices of $B$, as desired. 

Finally, let us consider property~\ref{le:badproperties_6}.  
If $B$ is isomorphic to $B_{16}^{(2)}$ then this property is easily seen to be true. 
Otherwise, $B$ is constructed from $B_8$ using $k\geq 2$ iterated applications of the $8-$augmentation operation, and it can be checked that $B$ then has at least $k$ vertex-disjoint $6$-cycles. 
\end{proof}

\begin{lemma}\label{le:dangerousproperties}
The following properties hold for every dangerous graph $D$.
\begin{enumerate}
\item \label{le:dangerousproperties_tfs2c} $D$ is triangle-free, subcubic, and $2$-connected; 
\item \label{le:dangerousproperties_degree2} $D$ has exactly five vertices of degree $2$, and if $D$ is not $C_5$ then the degree-$2$ vertices induce an isolated vertex plus a size two matching; 
\item \label{le:dangerousproperties_alpha_mu} $\alpha(D) \geq \lb(D)$; 
\item \label{le:dangerousproperties_T_two_vertices} If $T$ is a set of two degree-$2$ vertices of $D$ then $D$ has an independent set of size $\lb(D)$ that avoids $T$;   
\item \label{le:dangerousproperties_T} If $T$ is a set of three degree-$2$ vertices of $D$ such that the remaining two degree-$2$ vertices are not adjacent then $D$ has an independent set of size $\lb(D)$ that avoids $T$;   
\item \label{le:dangerousproperties_5cycle} $D$ contains a $5$-cycle. 
\end{enumerate}
\end{lemma}
\begin{proof}
Properties~\ref{le:dangerousproperties_tfs2c} and~\ref{le:dangerousproperties_degree2} follow from the inductive definition of dangerous graphs. 
Let us show property~\ref{le:dangerousproperties_alpha_mu}, by induction. 
This property is true if $D$ is $C_5$. 
If $D$ is the sum of two smaller dangerous graphs $D_1$ and $D_2$, then $\lb(D) = \lb(D_1) + \lb(D_2) + 1$, and it is easily seen that $\alpha(D) \geq \alpha(D_1) + \alpha(D_2) + 1$, implying $\alpha(D) \geq \lb(D)$ using the induction hypothesis on $D_1$ and $D_2$. 
If $D$ is an $8-$augmentation of a smaller dangerous graph $D'$, then $\alpha(D) =  \alpha(D') + 3 \geq \lb(D') + 3 = \lb(D)$ by Lemma~\ref{le:8augmentation} and the induction hypothesis. 
If $D$ is a join of two bad graphs $H_1$ and $H_2$, then $\lb(D) = \lb(H_1) + \lb(H_2) - 1 - \frac{1}{6} = \alpha(H_1) + \alpha(H_2) -1$ by Lemma~\ref{le:badproperties}.\ref{le:badproperties_alpha_mu}.  
Thus it only remains to observe that $\alpha(D) \geq \alpha(H_1) + \alpha(H_2) -1$, which follows easily from the definition of the join operation using Lemma~\ref{le:badproperties}.\ref{le:badproperties_5}. 

Note that property~\ref{le:dangerousproperties_T_two_vertices} follows directly from property~\ref{le:dangerousproperties_T} by adding an extra degree-$2$ vertex to $T$ in such a way that the remaining two degree-$2$ vertices are not adjacent (which is always possible by property~\ref{le:dangerousproperties_degree2}).  

Let us prove property~\ref{le:dangerousproperties_T}, by induction.  
The property is true for $C_5$. 
Suppose that $D$ is the sum of two dangerous graphs $D_1$ and $D_2$. 
By the definition of sum, there is a degree two vertex $u$ with degree three neighbors $v_1,v_2$ such that each of $v_1,v_2$ has two neighbors in $D_1 \cup D_2$. 
Let $i\in \left\{1,2\right\}$ and let $a_i,b_i,c_i,d_i,e_i$ denote the degree-$2$ vertices of $D_i$, such that $a_i$ and $b_i$ each have a neighbor in $\left\{v_1,v_2\right\}$, and $c_1c_2, d_1e_1, d_2e_2 \in E(D)$. 

By symmetry, we may assume that either $T=\left\{u,d_1,d_2\right\}$ or $T=\left\{d_1,e_1,d_2\right\}$. 
Suppose first that  $T= \left\{ u,d_1,d_2\right\}$. By induction there exists an independent set $S_i$ of $D_i$ of size $\lb(D_i)$ avoiding $\left\{a_i,b_i,d_i \right\}$, for $i=1,2$. 
Then  $(S_1 \cup S_2 \cup \left\{v_1,v_2\right\}) - \left\{c_1\right\}$ is an independent set of $D$ avoiding $T$ and with size $\lb(D_1) + \lb(D_2) + 1 = \lb(D)$.  
Next, suppose that $T=\left\{d_1,e_1,d_2\right\}$. 
By induction, there is an independent set $S_2$ of $D_2$ of size $\lb(D_2)$ avoiding $\left\{c_2,d_2\right\}$. 
Similarly, there is an independent set $S_1$ of $D_1$ of size $\lb(D_1)$ that avoids $\left\{d_1,e_1\right\}$. 
Then $S_1 \cup S_2 \cup \left\{u\right\}$ is an independent set of $D$ avoiding $T$ and with size $\lb(D)$. 

Next, suppose that $D$ is a join of two bad graphs $H_1$ and $H_2$, say $D$ is obtained from the disjoint union of $H_1$ and $H_2$ by choosing a corner $(a_i, b_i, c_i)$ of $H_i$ for $i=1,2$, and then removing the vertices $b_1, b_2, c_2$, and adding the edges $c_1a_2$ and $a_1y_2$, where $y_2$ is the neighbor $c_2$ in $H_2$ that is distinct from $b_2$, as in the definition. 
Let $d_i, e_i$ denote the two degree-$2$ vertices of $H_i$ distinct from $b_i, c_i$, for $i=1,2$ (thus $d_ie_i \in E(H_i)$). 
First suppose that $c_1 \in T$. 
Then $T$ contains exactly one of $d_i, e_i$ for $i=1,2$, say without loss of generality $d_i$. 
Let $S_i$ ($i=1,2$) be a maximum independent set of $H_i$ avoiding $b_i, c_i, d_i$, which exists by Lemma~\ref{le:badproperties}.\ref{le:badproperties_5}. 
Let $S$ be obtained from $S_1 \cup S_2$ by removing $y_2$ in case $y_2 \in S_2$. 
Then $S$ is an independent set of $D$ avoiding $T$ and of size at least $\alpha(H_1) + \alpha(H_2) - 1 = \lb(D)$, as desired. 
Now assume that $c_1 \notin T$. 
Then relabeling if necessary we may assume that either $T=\{ d_1,d_2,e_1\}$ or $T=\{d_1, d_2, e_2\}$. 
First, if $T=\{ d_1,d_2,e_1\}$, then let $S_1$ be a maximum independent set of $H_1$ avoiding $c_1,d_1,e_1$, and let $S_2$ be a maximum independent set of $H_2$ avoiding $b_2,c_2,d_2$. 
Let $S$ be obtained from $S_1\cup S_2$ by removing $a_1$ in case $a_1 \in S_1$ and removing $b_1$ in case $b_1 \in S_1$. Note that we have removed at most one vertex since $a_1b_1\in E(H_1)$. Then $S$ is an independent set of $D$ avoiding $T$ and of size at least $\alpha(H_1)+\alpha(H_2)-1=\mu(D)$, as desired.
If on the other hand $T=\{d_1, d_2, e_2\}$, then let $S_1$ be a maximum independent set of $H_1$ avoiding $b_1, c_1, d_1$, and let $S_2$ be a maximum independent set of $H_2$ avoiding $b_2, d_2, e_2$. 
Let $S$ be obtained from $S_1 \cup S_2$ by removing $c_2$ in case $c_2 \in S_2$ and removing $y_2$ in case $y_2 \in S_2$.  
Note that we removed at most one vertex since $c_2y_2 \in E(H_2)$. 
Thus again $S$ is an independent set of $D$ avoiding $T$ and of size at least $\alpha(H_1) + \alpha(H_2) - 1 = \lb(D)$, as desired. 

The remaining case where $D$ is an $8-$augmentation of a smaller dangerous graph is easier.  
We leave it to the reader. 

Finally, property~\ref{le:dangerousproperties_5cycle} is easily seen to be true from the inductive definition of dangerous graphs.  
(The only case requiring some thoughts is when $D$ is a join of two bad graphs. By an induction argument, every bad graph $B$ contains two $5$-cycles that do not have any degree-$2$ vertex of $B$ in common. The join operation preserves at least one of these $5$-cycles.)
\end{proof}

\begin{lemma}\label{le:outdegree4bad} 
Let $G$ be a connected critical subcubic graph containing a bad graph $B$ as a proper subgraph (that is, $B \neq G$). 
Then $V(B) \subsetneq V(G)$, $B$ is an induced subgraph of $G$, $B$ is isomorphic to $B_8$ or $B_{16}^{(1)}$, and there are exactly four edges between $V(B)$ and $V(G)-V(B)$. 
\end{lemma}
\begin{proof}

If $V(B)=V(G)$ then $G$ is the graph $B$ plus one or two edges linking degree-$2$ vertices of $B$. 
However, by Lemma~\ref{le:badproperties}.\ref{le:badproperties_5} there is a maximum independent set of $B$ avoiding three of its degree-$2$ vertices, which is thus also an independent set of $G$, and hence $\alpha(B)=\alpha(G)$. 
This implies that these one or two extra edges in $G$ cannot be critical, a contradiction. 
Therefore, $V(B) \subsetneq V(G)$. 


Since $G$ is connected, there exists an edge $ua$ with $u\in V(G) - V(B)$ and $a\in V(B)$.
 If $B$ is not induced in $G$, then there is exactly one edge $bc$ in $G$ linking two degree-$2$ vertices of $B$.
 By Lemma~\ref{le:badproperties}.\ref{le:badproperties_5}, there is a maximum independent set $S_B$ of $B$ that avoids each of its degree-$2$ vertices that are distinct from $c$.
 By criticality of $G$, each maximum independent set $S$ of $G-bc$ contains both $b$ and $c$. 
 However, since $|S\cap V(B)| \leq  |S_B|$, we can replace $S\cap V(B)$ with $S_B$ to obtain a new independent set of $G-bc$ of size at least $|S|$ that avoids $b$. Contradiction.
 Therefore, $B$ is an induced subgraph of $G$. 

Suppose that $B$ is not isomorphic to $B_8$ or $B_{16}^{(1)}$. 
Then by Lemma~\ref{le:badproperties}.\ref{le:badproperties_5}, $B$ has a maximum independent set $S_B$ that avoids each of its degree-$2$ vertices. 
In particular, $S_B$ avoids $a$. 
By criticality of $G$, each maximum independent set $S$ of $G-ua$ contains both $u$ and $a$. 
However, since $|S\cap V(B)|\leq |S_B|$, we can replace $S\cap V(B)$ with $S_B$ to obtain a new independent set of $G-ua$ of size at least $|S|$ that avoids $a$. Contradiction. 
This proves that $B$ is isomorphic to $B_8$ or $B_{16}^{(1)}$. 

It remains to show that there are exactly four edges between $V(B)$ and $V(G)-V(B)$.  
Let $T=\left\{a,b,c,d\right\} \subseteq V(B)$ denote the set of vertices that have degree $2$ in $B$. 
Suppose for a contradiction that some vertex of $T$, say $b$, has no edge to $V(G)-V(B)$. 
Recall that $a$ does have an edge to $V(G)-V(B)$, namely $ua$. 
Since $G$ is critical, every maximum independent set $S$ of $G-ua$
 must contain $u$ and $a$. 
By Lemma~\ref{le:badproperties}.\ref{le:badproperties_5}, $B$ contains a maximum independent set $S_B$ that avoids $\left\{a,c,d\right\}$. 
This means that we can locally modify $S$ by replacing $S \cap V(B)$ with $S_B$, thus obtaining an independent set of $G$ of size at least $|S|=\alpha(G)+1$. Contradiction.
We have shown that each vertex of $T$ has an edge to $V(G)-V(B)$, and so $B$ has exactly four edges to $V(G)-V(B)$. 
\end{proof}

\section{Proof of Theorem~\ref{th:main}}
\label{sec:proof}

The following Table~\ref{table:parameters} lists the relevant parameters for the graphs $F_{11}$, $F_{14}^{(1)}$, $F_{14}^{(2)}$, $F_{19}^{(1)}$, $F_{19}^{(2)}$, $F_{22}$. 
Observe in particular that the three cubic graphs $F$ satisfy $\alpha(F) = \lb(F)-\tfrac{1}{6}$ while the three noncubic graphs $F$ satisfy $\alpha(F)= \lb(F)-\tfrac{1}{12}$. 
These values will be used in various places in the proof. 

  \begin{table}[htb!]
  \centering
  \begin{tabular}{ |c || c | c | c | c | c } 
    \hline 
    $ $ & $n_3$ & $n_2$ & $\alpha$  & $\lb$ \\[.5ex] 
    \hline \hline 
    \rule{0pt}{1.1\normalbaselineskip} $F_{11}$ & $10$ & $1$ & $4$  &  $4 + \tfrac{1}{12}$ \\[.7ex]  
    $F_{14}^{(1)},F_{14}^{(2)}$ & $14$ & $0$ & $5$ &  $5 + \tfrac{1}{6}$ \\[.7ex]   
    $F_{19}^{(1)},F_{19}^{(2)}$ & $18$ & $1$ & $7$  & $7 + \tfrac{1}{12}$ \\[.7ex]  
        $F_{22}$ & $22$ & $0$ & $8$  & $8 + \tfrac{1}{6}$ \\[.7ex]   
    \hline
  \end{tabular}
  \caption{The values of some relevant parameters for the graphs $F_{11}$, $F_{14}^{(1)}$, $F_{14}^{(2)}$, $F_{19}^{(1)}$, $F_{19}^{(2)}$, $F_{22}$. (Recall that $n_i$ denotes the number of vertices of degree~$i$).}
  \label{table:parameters}
  \end{table}

Let $G$ be a graph. An \emph{independence packing} for $G$ is a collection of vertex-disjoint subgraphs $G_1,G_2,\ldots, G_k$ of $G$ such that
\begin{enumerate}
\item Each $G_i$ is connected and critical;
\item $V(G)=\bigcup_{1=1}^{k} V(G_i)$;
\item $\alpha(G)=\sum_{i=1}^{k} \alpha(G_i)$. 
\end{enumerate}

\begin{lemma}
Every graph $G$ has an independence packing.
\end{lemma}
\begin{proof}
We proceed by induction. If $G$ itself is critical, then all of its components are critical, and they together define an independence packing. Otherwise, there is a noncritical edge $e$, and it suffices to observe that every independence packing for $G-e$ is also one for $G$. 
\end{proof}

If $G$ is connected and critical, then a trivial independence packing for $G$ is obtained by taking $G$ itself. 
Let us point out that this is also the only one: 

\begin{lemma}
\label{lem:trivial_alpha_packing}
If $G$ is a connected and critical graph, then $G$ has a unique independence packing, which is $G$ itself. 
\end{lemma}
\begin{proof}
Suppose that $G_1,\ldots, G_k$ is an independence packing of $G$. 
Then by definition $J=G_1 \cup \cdots \cup G_k$ is a spanning subgraph of $G$ with $\alpha(J)=\alpha(G)$. 
However, since every edge of $G$ is critical, every proper spanning subgraph of $G$ has independence number strictly greater than $\alpha(G)$, and hence we must have $J=G$.  
\end{proof}

Now we turn to the proof of Theorem~\ref{th:main}. 
Arguing by contradiction, we fix for the rest of this section a hypothetical counterexample $G$ to Theorem~\ref{th:main} that minimizes $|V(G)|+|E(G)|$. 
We know from Lemmas~\ref{le:badproperties}.\ref{le:badproperties_alpha_mu} and~\ref{le:dangerousproperties}.\ref{le:dangerousproperties_alpha_mu} that $G$ is neither bad nor dangerous. 
Thus, in order to arrive at a contradiction we need to show that $\alpha(G)\geq \lb(G)$, and moreover that $\alpha(G)\geq \lb(G) + \tfrac{1}{12}$ in case $G$ has at least three vertices of degree $2$ and contains no bad subgraph. 

Since $G$ is connected, critical and not isomorphic to $K_1$ or $K_2$, we already know from Lemma~\ref{le:twoconnected1} that $G$ is $2$-connected. 
Our first goal is to prove that $G$ is $3$-connected.

\subsection{Reduction to the $3$-connected case}

\begin{lemma}\label{le:twoedgecutset} Suppose that $G$ has a $2$-edge cutset $\left\{e_1,e_2\right\}$. 
Then exactly one component of $G-\left\{e_1,e_2\right\}$ is isomorphic to $K_1$ or $K_2$.
\end{lemma}
\begin{proof}
Since $G-e_1$ is connected, $G-\left\{e_1,e_2\right\}$ has precisely two components. 
It is checked that $G$ is not a counterexample to Theorem~\ref{th:main} if both components are isomorphic to $K_1$ or $K_2$. 
Arguing by contradiction, let us assume that neither of them are isomorphic to $K_1$ or $K_2$. 
Let $C_0,C_1$ denote the two components of $G-\left\{e_1,e_2\right\}$.  
Thus, each of these two components has at least three vertices. 

If $e_1$ and $e_2$ have a vertex in common, then this vertex is a cutvertex of $G$, contradicting the fact that $G$ is $2$-connected. 
So $e_1$ and $e_2$ must form a matching.
Say $e_1=ab$ and $e_2=cd$, with $a,c \in V(C_0)$ and $b,d \in V(C_1)$. 
Note that $ac$ is not an edge of $G$, for otherwise $\left\{a,c\right\}$ would be a clique cutset of $G$, contradicting the fact that $G$ is critical (see Lemma~\ref{le:nocliquecutset}). 
Also, $ad$ cannot be an edge, for otherwise $\left\{ab,cd\right\}$ would not be an edge cutset of $G$. 
By symmetry, it follows that $ab$ and $cd$ are the only two edges in $G[\left\{a,b,c,d\right\}]$.

We now apply Lemma~\ref{le:gluing_critical} to $G$ with the $2$-cutset $\left\{b,c\right\}$. 
Exchanging $C_0$ and $C_1$ if necessary, we may assume that the outcome of the lemma consists of two connected critical graphs $G_0, G_1$ such that
\begin{itemize}
\item $G_0$ is $C_0$ plus an extra vertex $w$ adjacent to $a$ and $c$ (and no other vertices);
\item $G_1$ is $C_1$ plus the edge $bd$;
\item $\alpha(G)=\alpha(G_0)+\alpha(G_1)$.
\end{itemize}

Note that $G_0$ is triangle-free, because $C_0$ was triangle-free and $ac\notin E(G_0)$. 
On the other hand, $G_1$ is not necessarily triangle-free, because the addition of the edge $bd$ to $C_1$ could create a triangle. 
For this reason, we define a new triangle-free graph $G_2$ from $G_1$, by subdividing twice the edge $bd$. 
By Lemma~\ref{le:oddsubdivision}, $G_2$ is still critical, and $\alpha(G_2)=\alpha(G_1)+1$. 
Hence, $\alpha(G)=\alpha(G_0)+\alpha(G_2)-1$. 

The graph $G_0 \cup G_2$ has three more vertices than $G$, and these extra vertices have degree $2$, while the remaining vertices have the same degrees as in $G$.  
Thus, $\lb(G) = \lb(G_0)+\lb(G_2) + \tfrac{1}{12} - 3\cdot \tfrac{5}{12}= \lb(G_0)+\lb(G_2) - 1 - \tfrac{1}{6}$.

Let us show that $\alpha(G_i)\geq \lb(G_i)-\tfrac{1}{12}$ holds for $i\in \left\{0,2\right\}$.  
Both graphs $G_0$ and $G_2$ are connected, critical, and strictly smaller than $G$. 
(Note that $|V(G_2)| + |E(G_2)| < |V(G)| + |E(G)|$ because $C_0$ has at least three vertices.) 
Thus they are not counterexamples to Theorem~\ref{th:main}. 
Hence, if $G_i$ ($i\in \left\{0,2\right\}$) has none of $F_{11}$, $F_{14}^{(1)}$, $F_{14}^{(2)}$, $F_{19}^{(1)}$, $F_{19}^{(2)}$, $F_{22}$ as a subgraph, then $\alpha(G_i)\geq \lb(G_i)-\tfrac{1}{12}$. 
If, on the other hand, $G_i$ contains one graph $F$ from these six graphs as a subgraph then $G_i=F$ (because $G_i$ is $2$-connected) and $F$ is one the three noncubic graphs (because $G_i$ has a degree-$2$ vertex), and $i=0$ (because $G_2$ has two degree-$2$ vertices). 
Therefore, $\alpha(G_0)=\lb(G_0)-\tfrac{1}{12}$, as indicated in Table~\ref{table:parameters}. 

In conclusion, 
\[
\alpha(G)=\alpha(G_0)+\alpha(G_2)-1 \geq \lb(G_0)+\lb(G_2) - 1 - 2\cdot \tfrac{1}{12} = \lb(G).
\]
This is the desired contradiction, unless $G$ has at least three degree-$2$ vertices, is not dangerous, and has no bad subgraph, in which case we need to show that $\alpha(G) \geq \lb(G) + \tfrac{1}{12}$. 
So let us consider that case. 

As argued above, $G_2$ contains none of $F_{11}$, $F_{14}^{(1)}$, $F_{14}^{(2)}$, $F_{19}^{(1)}$, $F_{19}^{(2)}$, $F_{22}$ as a subgraph. 
If $G_2$ is not bad then $\alpha(G_2)\geq \lb(G_2)$, implying that $\alpha(G)\geq \lb(G)+\tfrac{1}{12}$; contradiction. 
So $G_2$ must be bad. 

If $G_0$ contains one graph $F$ among $F_{11}$, $F_{14}^{(1)}$, $F_{14}^{(2)}$, $F_{19}^{(1)}$, $F_{19}^{(2)}$, $F_{22}$ as a subgraph then $G_0=F$ and $F$ has exactly one degree-$2$ vertex, as already mentioned. 
However, this is not possible because $G$ would then have only two degree-$2$ vertices. 

If $G_0$ is not bad then we obtain $\alpha(G_0)\geq \lb(G_0)$, and so $\alpha(G)\geq \lb(G)+\tfrac{1}{12}$; contradiction.  
Hence, $G_0$ must also be bad.  
But then $G$ is a join of the two bad graphs $G_0$ and $G_2$, and hence $G$ is dangerous, a contradiction. 
\end{proof}

\begin{lemma}\label{le:twovertexcutset}
Suppose that $G$ has a $2$-cutset $\left\{v, w\right\}$. 
Then $G-\left\{v,w\right\}$ has exactly two components, and one of the two components is isomorphic to $K_1$ or $K_2$. 
\end{lemma}
\begin{proof}
We already know from Lemma~\ref{le:gluing_critical} that $G - \{v, w\}$ has exactly two components, say $C_1$ and $C_2$, and from Lemma~\ref{le:nocliquecutset} that $v$ and $w$ are not adjacent.  
We may assume that $v$ has only one neighbor in $C_1$, say $v^*$. 
If $w$ also has only one neighbor in $C_1$, say $w^*$, then $C_1$ is one of the two components of $G-\left\{vv^*,ww^*\right\}$ and that component must be isomorphic to $K_1$ or $K_2$ by Lemma~\ref{le:twoedgecutset}, since the other component contains $v$ and $w$, which are not adjacent. 
If, on the other hand, $w$ has two neighbors in $C_1$, then $w$ has exactly one neighbor in $C_2$, say $w^*$, and $\left\{vv^*,ww^*\right\}$ is again a $2$-edge cutset of $G$ (since $vw\notin E(G)$). 
Both components of $G-\left\{vv^*,ww^*\right\}$ contain at least two vertices, thus one is isomorphic to $K_2$ by Lemma~\ref{le:twoedgecutset}, and it follows that one of $C_1, C_2$ is isomorphic to $K_1$. 
\end{proof}

Our next goal is to show that $G$ does not have a bad subgraph.

\begin{lemma}\label{le:nobadsubgraph}
$G$ does not have a bad subgraph.
\end{lemma}
\begin{proof}
Arguing by contradiction, suppose that $G_0$ is a bad subgraph of $G$. 
By Lemma~\ref{le:outdegree4bad}, $G_0$ is induced in $G$ and $G_0$ is isomorphic to $B_8$ or $B_{16}^{(1)}$. 
Moreover, $V(G)-V(G_0) \neq \emptyset$, and there are exactly four edges between $V(G_0)$ and $V(G)-V(G_0)$. 
Let $v_1,v_2,v_3,v_4$ denote the degree-$2$ vertices of $G_0$, in such a way that $v_1v_3$ and $v_2v_4$ are edges of $G_0$. 
For each $i \in \{1, \dots, 4\}$, let $w_i$ denote the neighbor of $v_i$ in $V(G)-V(G_0)$. 
Note that $w_1,w_2,w_3,w_4$ are not necessarily pairwise distinct. 
However, we do know that $w_1\neq w_3$ and $w_2\neq w_4$ since $G$ is triangle-free. 
Therefore, $|\left\{w_1,w_2,w_3,w_4\right\}| \in \left\{2,3,4\right\}$. 

First, we show:
\begin{equation}\label{eq:surroundingtheindependentset}
\left\{w_1,w_2,w_3,w_4\right\} \subseteq S \text{ for every maximum independent set } S \text{ of } G-V(G_0).
\end{equation}
Suppose for a contradiction that there exists a maximum independent set $S$ of $G-V(G_0)$ that avoids $w_i$ for some $i\in \left\{1,2,3,4\right\}$.  
By Lemma~\ref{le:badproperties}.\ref{le:badproperties_5}, $G_0$ has a maximum independent set $S_0$ avoiding all three vertices $v_j$ with $j\neq i$, and thus $S \cup S_0$ is an independent set of $G$. 
Hence, $\alpha(G) \geq \alpha(G-V(G_0))+\alpha(G_0)$. 
However, this contradicts the fact that $\alpha(G) < \alpha(J)$ holds for every spanning subgraph $J$ of $G$ with $J \neq G$ (since $G$ is critical). 
Therefore, \eqref{eq:surroundingtheindependentset} holds. 

Note in particular that $\left\{w_1,w_2,w_3,w_4\right\}$ is an independent set of $G$, by \eqref{eq:surroundingtheindependentset}. 
We now distinguish three cases, depending on the size of $\left\{w_1,w_2,w_3,w_4\right\}$. 

\textbf{Case $1$}: $|\left\{w_1,w_2,w_3,w_4\right\}|=2.$\\
In this case either $w_1=w_2$ and $w_3=w_4$, or $w_1=w_4$ and $w_2=w_3$.  
Recalling that $G_0$ is isomorphic to $B_8$ or $B_{16}^{(1)}$, it is easily seen that $V(G) \neq V(G_0) \cup \left\{w_1,w_3\right\}$, because the graph induced by $V(G_0) \cup \left\{w_1,w_3\right\}$ cannot be a counterexample to Theorem~\ref{th:main}. 
Thus, $\left\{w_1,w_3\right\}$ is a cutset of $G$. 
Hence, by Lemma~\ref{le:twovertexcutset}  $G-\left\{w_1,w_3\right\}$ has exactly two components, $G_0$ and another component $C$ which must be isomorphic to $K_1$ or $K_2$.  

Suppose first that $C$ is isomporphic to $K_2$. 
Then the graph $G-V(G_0)$ is isomorphic to a $4$-vertex path with endpoints $w_1, w_3$, and in particular has a maximum independent set avoiding one of its endpoints, contradicting \eqref{eq:surroundingtheindependentset}. 
Next suppose that $C$ is isomorphic to $K_1$. 
Then $G$ is either isomorphic to $F_{11}$ (if $G_0$ is isomorphic to $B_8$), or to one of $F_{19}^{(1)}, F_{19}^{(2)}$ (if $G_0$ is isomorphic to $B_{16}^{(1)}$); contradiction.

\textbf{Case $2$}: $|\left\{w_1,w_2,w_3,w_4\right\}|=3.$\\
Without loss of generality $w_1=w_2$ and $w_1,w_3,w_4$ are pairwise distinct. 
Let $H$ denote the graph obtained from $G$ by removing all vertices of $V(G_0)$ except the three vertices $v_1,v_3,v_4$, and adding the two edges $v_1v_4, v_3v_4$. 
In other words: we replace $G_0$ with the triangle $v_1v_3v_4$. This is illustrated in Figure~\ref{fig:no_bad_subgraph_case2}.

\begin{figure}
\centering
\includegraphics[width=0.95\textwidth]{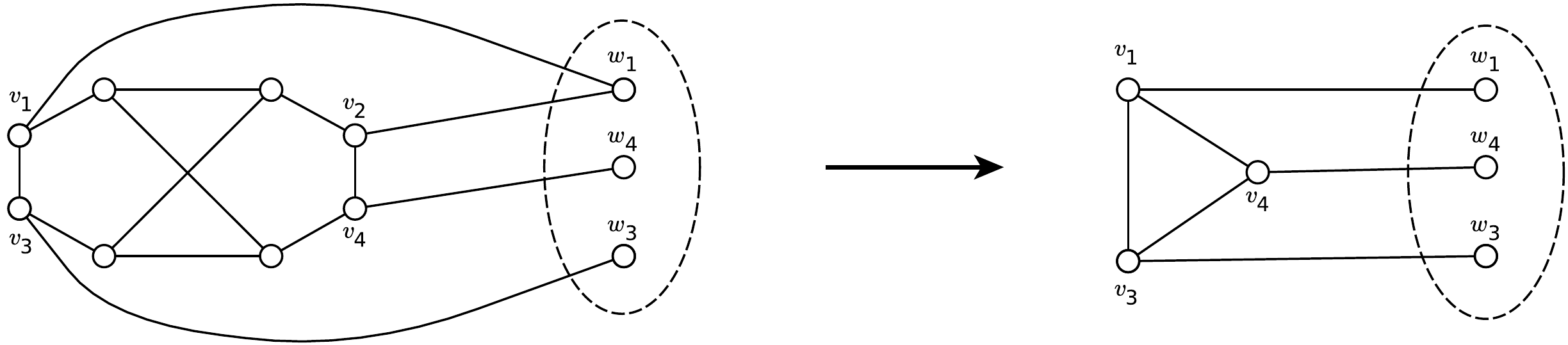}
\caption{Replacing $G_0$ with a triangle in Case $2$ of the proof of Lemma~\ref{le:nobadsubgraph} illustrated for $G_0=B_8$.}
\label{fig:no_bad_subgraph_case2}
\end{figure}

Recall that every maximum independent set of $G-V(G_0)=H-\{v_1,v_3,v_4\}$ contains all three vertices $w_1,w_3,w_4$, by \eqref{eq:surroundingtheindependentset}. 
Thus, $\alpha(H) = \alpha(G-V(G_0))=\alpha(G) - \alpha(G_0)+1$.  

Now, consider an edge $e$ of $H$ distinct from $v_1v_4$ and $v_3v_4$. 
Thus $e$ also exists in $G$. 
Let $S_e$ be an independent set of $G-e$ of size $\alpha(G)+1$. 
The set $S_e$ can be mapped to an independent set $T_e$ of $H-e$ of size $\alpha(H)+1$ in the following way. 
First, suppose that $e$ has no endpoint in $\{v_1,v_3,v_4\}$.  
If $S_e$ contains all three vertices $w_1, w_3, w_4$, then we know that $|S_e \cap V(G_0)| = \alpha(G_0)-1$, and it suffices to take $T_e=S_e - V(G_0)$. 
If, on the other hand, $S_e$ avoids some vertex $w_i$ with $i \in \{1,3,4\}$, then $|S_e \cap V(G_0)| = \alpha(G_0)$, and we take the set $T_e=(S_e - V(G_0)) \cup \{v_i\}$. 
Next, suppose that $e=v_iw_i$ for some $i \in \{1,3,4\}$. 
Here, it suffices to take a maximum independent set of $G-V(G_0)$ plus the vertex $v_i$ for the set $T_e$. 
Finally, assume that $e = v_1v_3$. 
Since $S_e$ avoids $w_1$ and $w_3$, we know from \eqref{eq:surroundingtheindependentset} that $|S_e - V(G_0)| \leq \alpha(G-V(G_0)) -1$. 
It then follows that $|S_e \cap V(G_0)| \geq \alpha(G_0) + 1$, and thus these two inequalities hold with equality. 
In this case, it suffices to take $T_e = (S_e - V(G_0)) \cup \{v_1,v_3\}$. 

This shows that all edges $e$ considered above are critical in $H$. 
Note that, since $v_1v_3$ is critical in $H$, so is the edge $v_1v_4$, by symmetry. 
(Indeed, exchanging $v_1$ and $v_2$ in the definition of $H$ results in the same graph $H$, except that $v_1$ is relabeled $v_2$, and the above argument then shows that the edge $v_2v_4$ is critical in $H$.) 
In conclusion, all edges of $H$ are critical in $H$, except perhaps the edge $v_3v_4$. 
We now distinguish two cases, depending on whether that last edge is critical or not. 

First suppose that $v_3v_4$ is not critical in $H$. 
Then the graph $H-v_3v_4$ is connected and critical, and thus $2$-connected, and moreover triangle-free. 
If $H-v_3v_4$ contains one of $F_{11}$, $F_{14}^{(1)}$, $F_{14}^{(2)}$, $F_{19}^{(1)}$, $F_{19}^{(2)}$, $F_{22}$ as a subgraph, then it is isomorphic to that graph (by $2$-connectivity). 
However, this is not possible since $H-v_3v_4$ has two degree-$2$ vertices ($v_3$ and $v_4$). 
Since $H-v_3v_4$ is smaller than $G$, we may apply Theorem~\ref{th:main} to $H-v_3v_4$, giving that $\alpha(H-v_3v_4) \geq \lb(H-v_3v_4) - \frac{1}{12}$.  
Recalling once more that $G_0$ is isomorphic to $B_8$ or $B_{16}^{(1)}$, it is easily checked that $\lb(G) = \lb(H-v_3v_4) + \alpha(G_0)-1 - \frac{1}{4}$. 
Thus 
\[
\alpha(G) = \alpha(H) + \alpha(G_0)-1 = \alpha(H-v_3v_4) + \alpha(G_0)-1 \geq \lb(H-v_3v_4) + \alpha(G_0)-1 - \frac{1}{12} \geq \lb(G),  
\]
showing that $G$ is not a counterexample, a contradiction.  

Next, assume that $v_3v_4$ is critical in $H$. 
Thus $H$ is critical. 
Let $H'$ be the graph obtained from $H$ by subdividing twice the edge $v_3v_4$. 
Then $H'$ is also critical, and moreover triangle-free. 
Also, $H'$ does not contain any of $F_{11}$, $F_{14}^{(1)}$, $F_{14}^{(2)}$, $F_{19}^{(1)}$, $F_{19}^{(2)}$, $F_{22}$ as a subgraph, similarly as before.  
Moreover, $H'$ cannot be bad, since $w_1$ does not have degree three in $H'$ (see the construction of $H$, during which a neighbor of $w_1$ is deleted) and $w_3\neq w_4$. 
Since $H'$ is smaller than $G$, we may apply Theorem~\ref{th:main} to $H'$, giving $\alpha(H') \geq \lb(H')$.  
It follows 
\[ 
\alpha(G) 
= \alpha(H) + \alpha(G_0) - 1 
= \alpha(H') + \alpha(G_0) - 2
\geq \lb(H') + \alpha(G_0) - 2 
= \lb(G), 
\]
a contradiction.

\textbf{Case $3$}: $|\left\{w_1,w_2,w_3,w_4\right\}|=4.$\\ 
Recall that $\left\{w_1,w_2,w_3,w_4\right\}$ is an independent set, by \eqref{eq:surroundingtheindependentset}. 
First we show that there is at least one pair of vertices from that set that has no common neighbor. 
Suppose not, and consider the set $X$ of vertices of $G$ that see at least two vertices in $\left\{w_1,w_2,w_3,w_4\right\}$.  
Clearly, $X \subseteq V(G) - (V(G_0) \cup \left\{w_1,w_2,w_3,w_4\right\})$. 
One can check that $X$ must consist of exactly three vertices $x_1,x_2,x_3$, the first two having three neighbors in $\left\{w_1,w_2,w_3,w_4\right\}$ and the last one having two neighbors in $\left\{w_1,w_2,w_3,w_4\right\}$, as depicted in Figure~\ref{fig:x1x2x3}. 
Since $G$ is $2$-connected, it then follows that $x_3$ has no other neighbor in $G$, and that $G$ is precisely the graph induced by $V(G_0) \cup \left\{w_1,w_2,w_3,w_4, x_1,x_2,x_3\right\}$. 
However, this is a contradiction because this graph is not critical. 
This can be seen as follows: $\left\{w_1,w_2,w_3,w_4\right\}$ is the unique maximum independent set of $G - V(G_0)$, and this remains true even if we remove an edge incident to $x_1$. 
It follows that $w_jx_1$ is not critical in $G$.

\begin{figure}
\centering
\includegraphics[width=0.4\textwidth]{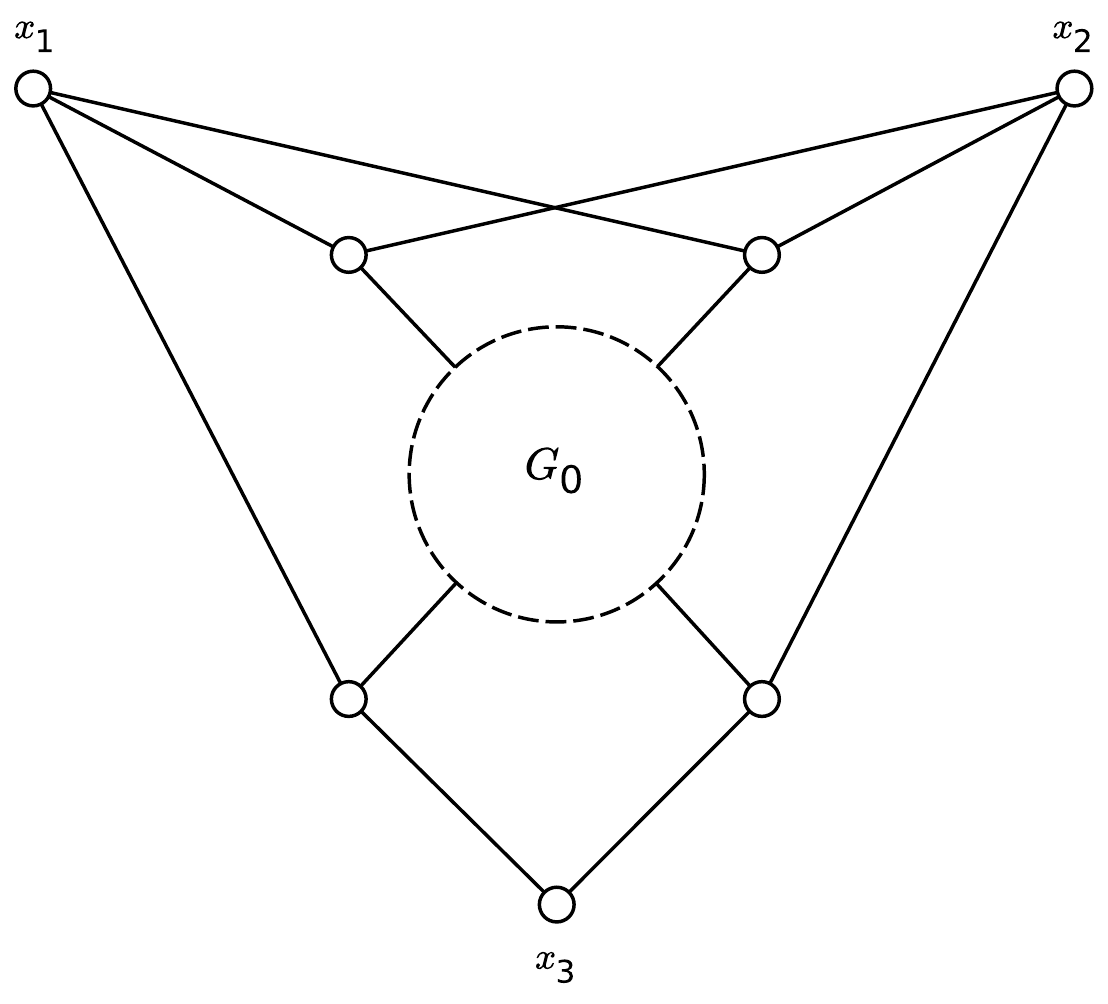}
\caption{Situation if every two vertices in $\left\{w_1,w_2,w_3,w_4\right\}$ have a common neighbor in Case $3$ of the proof of Lemma~\ref{le:nobadsubgraph}. 
The four unlabeled vertices are $w_1,w_2,w_3,w_4$ (in some order).}
\label{fig:x1x2x3}
\end{figure}

Therefore, we conclude that there is a pair of vertices in $\left\{w_1,w_2,w_3,w_4\right\}$ with no common neighbor. 
Choose such a pair and let $H$ denote the graph obtained from $G-V(G_0)$ by adding an edge $e^*$ connecting these two vertices. 
Thus $H$ is subcubic and triangle-free. 
Since every maximum independent of $G - V(G_0)$ includes $w_1,w_2,w_3,w_4$, it follows that $\alpha(H) = \alpha(G - V(G_0))-1$.  

Let $G_1,G_2,\ldots,G_k$ be an independence packing of $H$. 
Then, 
\[
\alpha(G) = \alpha(G - V(G_0)) + \alpha(G_0) - 1 = \alpha(H) + \alpha(G_0) = \sum_{i=0}^k \alpha(G_i).
\]
Since $e^*$ is a critical edge of $H$, it is included in some $G_i$ with $i\geq 1$, say without loss of generality in $G_k$. 

Let $m_i$ ($i \in \{0, 1, \dots, k\}$) denote the numbers of edges of $G$ that have exactly one endpoint in $V(G_i)$. 
For each $i\in \{0, 1, \dots, k-1\}$, the graph $G_i$ is a subgraph of $G$.  
Using Lemmas~\ref{le:outdegree4bad} and~\ref{le:twoedgecutset}, and that $G$ is $2$-connected, we obtain 
\[
m_i \geq \left\{ 
\begin{array}{ll}
4 & \textrm{ if } G_i \textrm{ is bad} \\
2 & \textrm{ if } G_i \textrm{ is isomorphic to } K_1 \textrm{ or } K_2 \\
3 & \textrm{ if } G_i \textrm{ is dangerous} \\
3 & \textrm{ if } |G_i| \geq 3 \textrm{ and } G_i \textrm{ is neither bad nor dangerous.} \\
\end{array}
\right.
\]
Let $\beta, \sigma, \delta, \nu$ (for {\em bad}, {\em small}, {\em dangerous}, {\em not} bad or dangerous) denote the number of graphs $G_i$ ($i\in \{0, 1, \dots, k-1\}$) in the first, second, third and fourth category above. 
Thus $\beta+\sigma+\delta+\nu = k$. 
Let $m :=|E(G)|-\sum_{i=0}^{k} |E(G_i)|$.  
By the previous discussion
\begin{equation}\label{eq:bound_on_2m}
2m \geq \sum_{i=0}^k m_i - 2 
\geq 4\beta + 2\sigma + 3\delta+3\nu + m_k - 2 
= 2k + 2\beta + \delta+\nu + m_k - 2. 
\end{equation}
(The $-2$ is due to the edge $e^*$.) 
Observe also that 
\[
\lb(G) = \sum_{i=0}^{k} \lb(G_i) +  \tfrac{1}{12} (k-m). 
\]
Now let us consider the graph $G_k$. 
Note that this graph could be isomorphic to one of $F_{11}$, $F_{14}^{(1)}$, $F_{14}^{(2)}$, $F_{19}^{(1)}$, $F_{19}^{(2)}$, $F_{22}$, since $G_k$ is not a subgraph of $G$. 
However, if this is the case then $G_k - e^*$ is connected, and thus $m_k \geq 3$ by Lemma~\ref{le:twoedgecutset}, implying that $G_k$ must be isomorphic to one of $F_{11}$, $F_{19}^{(1)}$, $F_{19}^{(2)}$, and $m_k = 3$. 

Note also that if $G_k$ is bad then we cannot apply Lemma~\ref{le:outdegree4bad} and deduce that $m_k = 4$ as above (again because $G_k$ is not a subgraph of $G$) but at least we know that $m_k \geq 3$ by Lemma~\ref{le:twoedgecutset}, since $G_k - e^*$ is connected.

Applying Theorem~\ref{th:main} to each graph $G_i$ ($i\in \{0, 1, \dots, k\}$), we obtain that 
\[
\alpha(G_i) \geq \left\{ 
\begin{array}{ll}
\lb(G_i) - \frac{1}{12} & \textrm{ if } G_i \textrm{ is bad or isomorphic to one of } F_{11}, F_{19}^{(1)}, F_{19}^{(2)}  \\[.5ex]
\lb(G_i) + \frac{1}{6} & \textrm{ if } G_i \textrm{ is isomorphic to } K_1 \textrm{ or } K_2 \\[.5ex]
\lb(G_i) & \textrm{ otherwise }
\end{array}
\right.
\]

For convenience, let $I$ be the indicator variable being equal to $1$ if $G_k$ is bad or isomorphic to one of $F_{11}$, $F_{19}^{(1)}$, $F_{19}^{(2)}$, and equal to $0$ otherwise. 
Combining the previous observations together, we obtain that 
\[
\alpha(G) = \sum_{i=0}^k \alpha(G_i)  
\geq \sum_{i=0}^k \lb(G_i) + \frac{1}{12}(2\sigma - \beta - I) 
= \lb(G) + \frac{1}{12}(m - k + 2\sigma - \beta - I).  
\] 
Thus, in order to deduce that $\alpha(G) \geq \lb(G)$ (as desired), it suffices to show that $m - k + 2\sigma - \beta - I \geq 0$. This can be seen as follows. Substituting inequality~(\ref{eq:bound_on_2m}) yields 
\[
2(m - k + 2\sigma - \beta - I) 
\geq 2k + 2\beta + \delta+\nu + m_k - 2 - 2k + 4\sigma - 2\beta - 2I
= 4\sigma + \delta+\nu + m_k - 2 - 2I,   
\]
which is at least $0$ if $I=0$ (because $m_k \geq 2$) or if $I=1$ and $m_k \geq 4$. 
If $I=1$ but $m_k=3$ then all we need is that $\sigma+\delta+\nu \geq 1$. 
But this is true by the handshaking lemma: 
Since $m_k$ is odd, at least one $m_i$ with $i < k$ must be odd as well, and the graph $G_i$ is not bad (since we would have $m_i=4$ if $G_i$ were bad). 
This concludes the proof of the lemma. 
\end{proof}

In the following lemmas we rule out various local structures around a degree-$2$ vertex of $G$. 

\begin{lemma}\label{le:nothreeconsecutivedegreetwo}
Every degree-$2$ vertex of $G$ has at least one neighbor of degree $3$. 
\end{lemma}
\begin{proof} 
Arguing by contradiction, suppose that $x$ is a degree-$2$ vertex with two degree-$2$ neighbors $y$ and $z$. 
Then, the two edges going out of $\{x,y,z\}$ form a $2$-edge cutset of $G$, and Lemma~\ref{le:twoedgecutset} implies that $G-\left\{x,y,z\right\}$ is isomorphic $K_1$ or $K_2$. 
That is, $G$ is a $4$-cycle or a $5$-cycle, neither of which is a counterexample to Theorem~\ref{th:main}. 
\end{proof}

In subsequent lemmas, we will often remove a well-chosen set of vertices from $G$ and then study the properties of an independence packing of the remaining graph. 
For this reason, it will be convenient to define a {\em near independence packing} of $G$ as a sequence $\mathcal{P}=(G_0, G_1, \dots, G_k)$ of vertex-disjoint connected subgraphs of $G$ covering all vertices of $G$ such that $G_0$ is a proper induced subgraph of $G$, and $G_1, \dots, G_k$ form an independence packing of $G- V(G_0)$. 
Given a near independence packing $\mathcal{P}=(G_0, G_1, \dots, G_k)$, we let $m_i(\mathcal{P})$ ($i \in \{0,1,\dots, k\}$) denote the numbers of edges of $G$ that have exactly one endpoint in $V(G_i)$. 
The graphs $G_1, \dots, G_k$ of the independence packing of $G-V(G_0)$ are naturally classified into three categories: 
\begin{itemize} 
	\item at least three vertices and dangerous; 
	\item at least three vertices but not dangerous; 
	\item at most two vertices. 	
\end{itemize} 
We let $\delta(\mathcal{P})$, $\nu(\mathcal{P})$, and $\sigma(\mathcal{P})$ denote respectively the number of graphs in the first, second, and third categories.\footnote{We remark that, while the notations $\delta, \nu, \sigma$ already appeared in Case~3 of the proof of Lemma~\ref{le:nobadsubgraph}, there $G_0, G_1, \dots, G_{k}$ was not exactly a near independence packing of $G$ because of the extra edge $e^*$, and moreover $\delta, \nu, \sigma$ were defined w.r.t.\ $G_0, G_1, \dots, G_{k-1}$ instead of $G_1, G_2, \dots, G_{k}$.}  
Also, let us define $\gamma(\mathcal{P})$ as the difference $\alpha(G) - \sum_{i=1}^{k}\alpha(G_i)$, that is, the gap between the independence numbers of $G$ and that of $G-V(G_0)$.  
(Since these notations will be used often, let us suggest the following mnemonic device: $\delta$ {\em dangerous}, $\nu$ {\em not} dangerous,  $\sigma$ {\em small}, and $\gamma$ {\em gap}.)  
We will drop $\mathcal{P}$ from the notations defined above when $\mathcal{P}$ is clear from the context. 

The following technical lemma will be useful in ruling out the remaining types of degree-$2$ vertices in $G$. 

\begin{lemma}\label{le:contract}
Fix some near independence packing $\mathcal{P}=(G_0, G_1, \dots, G_k)$ of $G$.  
If $m_0 \geq 3$, then 
\[
	\alpha(G) \geq \lb(G) +  \gamma  -\lb(G_0) +  \left(m_0 - \delta + \sigma \right) \cdot \tfrac{1}{12} 
\]
and
\[
\alpha(G) \geq \lb(G) +  \gamma  -\lb(G_0) + \left\lceil \frac{\sum_{i=0}^k m_i}{2} -\delta + \sigma \right\rceil  \cdot \tfrac{1}{12}
\geq \lb(G) +  \gamma  -\lb(G_0) + \left\lceil \frac{m_0 + \delta + 3\nu + 4\sigma}{2} \right\rceil  \cdot \tfrac{1}{12}.  
\]
Moreover, if $G$ is cubic then 
\[
\alpha(G) \geq \lb(G) +  \gamma -\lb(G_0) + \left\lceil \frac{m_0 + 3\delta + 3\nu + 4\sigma}{2} \right\rceil \cdot \tfrac{1}{12}.  
\]
\end{lemma}
\begin{proof}
Since $G$ is $2$-connected and since $m_0 \geq 3$, Lemma~\ref{le:twoedgecutset} implies that, for each $i \in \{1, \dots, k\}$, 
\[
m_i \geq \left\{ 
\begin{array}{ll}
3 & \textrm{ if } |G_i| \geq 3 \\
2 & \textrm{ otherwise.}
\end{array}
\right.
\]
Let $m:=|E(G)|-\sum_{i=0}^{k}|E(G_i)|$. 
Then
\begin{equation}\label{eq:estimating_m_one}
2m \geq \sum_{i=0}^k m_i 
\geq m_0 + 3\delta + 3\nu + 2\sigma.
\end{equation}
Observe also that 
\begin{equation}\label{eq:notwocubeneighbors2a}
\lb(G) = \sum_{i=0}^{k} \lb(G_i) + (k-m) \frac{1}{12}.  
\end{equation}

For each $i\in \{1, \dots, k\}$, the graph $G_i$ is not a counterexample to our theorem. 
Thus, 
\[
\alpha(G_i) \geq \left\{ 
\begin{array}{ll}
\lb(G_i) & \textrm{ if $|G_i|\geq 3$ and $G_i$ dangerous } \\
\lb(G_i) + \frac{1}{12} & \textrm{ if $|G_i|\geq 3$ and $G_i$ not dangerous } \\
\lb(G_i) + \frac{1}{6} & \textrm{ if $|G_i|\leq 2$}
\end{array}
\right.
\]
(In the second case, we used that $G_i$ has at least three degree-$2$ vertices since $m_i \geq 3$.)  
It follows that 
\[
\alpha(G) = \gamma + \sum_{i=1}^{k} \alpha(G_i) 
\geq \gamma + \sum_{i=1}^{k} \lb(G_i) + (\nu + 2\sigma) \frac{1}{12}
= \gamma + \lb(G) - \lb(G_0)  + (\nu + 2\sigma + m - k) \frac{1}{12}.  
\]
The first part of the lemma then follows by using $k=\delta+\nu+\sigma$ and the following two lower bounds: $m\geq m_0$ 
and $m \geq \left\lceil \frac{\sum_{i=0}^k m_i}{2}\right\rceil \geq \left\lceil \frac{m_0 + 3\delta + 3\nu + 2\sigma}{2}\right\rceil$. 

For the second part of the lemma, when $G$ is cubic, it suffices to notice that for each dangerous graph $G_i$ ($i\in \{1, \dots, k\}$), either $G_i$ is induced in $G$, in which case we know $m_i \geq 5$ by Lemma~\ref{le:dangerousproperties} (instead of just $m_i \geq 3$), or we find an edge of $G$ with both endpoints in $V(G_i)$ but not in $G_i$. 
Thus, we deduce 
\[
2m \geq m_0 + 5\delta + 3\nu + 2\sigma, 
\]
and plugging in this inequality in the previous proof gives the desired result. 
\end{proof}

\begin{lemma}\label{le:notwocubeneighbors}
Every degree-$2$ vertex of $G$ has at most one neighbor of degree $3$. 
\end{lemma}
\begin{proof}
Arguing by contradiction, suppose that there exists a degree-$2$ vertex $u$ with degree-$3$ neighbors $v_1$ and $v_2$. 
Let $\mathcal{P}=(G_0, G_1, \dots, G_k)$ be a near independence packing of $G$ with $G_0 = G[\{u,v_1,v_2\}]$. 
Since every maximum independent set of $G-V(G_0)$ can be extended to an independent set of $G$ by adding $u$, it follows that $\gamma \geq 1$.  
We also have $\lb(G_0)=\tfrac{5}{4}$ and $m_0=4$. 
Using Lemma~\ref{le:contract}, we obtain 
\[
\alpha(G) 
\geq \lb(G) + \gamma - \frac{5}{4} + \left\lceil \frac{4 + \sum_{i=1}^k m_i}{2} - \delta + \sigma \right\rceil  \cdot \tfrac{1}{12}
\geq \lb(G) - \frac{1}{4} + \left\lceil \frac{4 + \delta + 3\nu + 4\sigma}{2} \right\rceil  \cdot \tfrac{1}{12}.  
\]
Recall that, since $G$ is a counterexample, we have $\alpha(G) < \lb(G) + \frac{1}{12}$. 
Thus, it follows $\gamma=1$, $\nu=0$, $\sigma = 0$, and $\delta \leq 2$.  
Note that $\delta \geq 1$ since $G \neq G_0$. 
The above inequality then implies $\alpha(G) \geq \lb(G)$. 
It follows that $G$ has at least three degree-$2$ vertices (since otherwise $G$ is not a counterexample). 
This in turn implies $\delta \neq 1$ (since otherwise $G$ would consist of $G_0$ and a dangerous graph $G_1$ joined by $m_0=4$ edges; every dangerous graph has five degree-$2$ vertices, so whis would force $G$ to have only two degree-$2$ vertices). 
Thus $\delta = k = 2$.  
Furthermore, $m_1 + m_2 \leq 6$ (since otherwise $\alpha(G)\geq \lb(G)+ \frac{1}{12}$ by the above inequality).

Since $m_i \geq 3$ for $i=1,2$ by Lemma \ref{le:twoedgecutset}, it follows that $m_1=m_2=3$. 
Recalling that $m_0=4$, we deduce that there is precisely one edge $e$ between $G_1$ and $G_2$ and that there are precisely two edges between $G_i$ and $\left\{v_1,v_2\right\}$, for each $i \in \left\{1,2\right\}$. 
Let $a_i,b_i \in V(G_i)$ denote the two vertices having a neighbor in $\left\{v_1,v_2\right\}$ and let $c_i \in V(G_i)$ denote the vertex incident to $e$.  
Because $G_i$ has minimum degree $2$, it follows that $a_i,b_i,c_i$ have degree $2$ in $G_i$ and are pairwise distinct. Let $d_i,e_i$ denote the two remaining degree-$2$ vertices of $G_i$. 

We claim that $d_i$ and $e_i$ are adjacent for $i=1,2$. 
Suppose not, say without loss of generality $d_1$ and $e_1$ are not adjacent. Then by Lemma~\ref{le:dangerousproperties}.\ref{le:dangerousproperties_T}, there is an independent set $S_1$ of $G_1$ of size $\lb(D_1)$ avoiding $\{a_1,b_1,c_1\}$. 
From Lemma~\ref{le:dangerousproperties}.\ref{le:dangerousproperties_T} we also obtain an independent set $S_2$ of $G_2$ of size $\lb(D_2)$ that avoids $\left\{a_2,b_2\right\}$ as well as some vertex of $\left\{c_2,d_2,e_2\right\}$. Thus $S_1 \cup S_2 \cup \left\{v_1,v_2\right\}$ is an independent set of $G$ of size $\lb(D_1) + \lb(D_2) + 2$, which implies that $\alpha(G) \geq \lb(G) + 1$, a contradiction. 

Hence, $d_i$ and $e_i$ are adjacent for $i=1,2$, as claimed, and it follows in turn that $G$ is a sum of the dangerous graphs $G_1$ and $G_2$, and therefore $G$ is dangerous; contradiction.
\end{proof}

\begin{lemma}\label{le:evennumberofdegreetwos}
$G$ has an even number of degree-$2$ vertices. 
\end{lemma}
\begin{proof}
If not, then there is a degree-$2$ vertex with two degree-$2$ neighbors, which is forbidden by Lemma~\ref{le:nothreeconsecutivedegreetwo}, or with two degree-$3$ neighbors, which is forbidden by Lemma~\ref{le:notwocubeneighbors}. 
\end{proof}

\begin{lemma}\label{le:notwoadjacentdegreetwo}
$G$ has no two adjacent degree-$2$ vertices.
\end{lemma}
\begin{proof}
Arguing by contradiction, suppose $u_1,u_2 \in V(G)$ are two adjacent vertices of degree $2$. 
By Lemma~\ref{le:nothreeconsecutivedegreetwo}, the neighbor $v_1$ of $u_1$ and the neighbor $v_2$ of $u_2$ have degree $3$. 
By Lemma~\ref{le:nodegreetwoinfourcycle}, $v_1v_2$ is not an edge.

First suppose that $v_1$ and $v_2$ have no common neighbor. Then the graph $G'$ obtained by contracting the path $v_1u_1u_2v_2$ to an edge $v_1v_2$ is triangle-free. By Lemmas~\ref{le:oddsubdivision} and~\ref{le:twoconnected1}, $G'$ is critical and $2$-connected and $\alpha(G)=\alpha(G')+1$. Note that $G$ has two extra vertices of degree $2$ compared to $G'$, so $\lb(G)=\lb(G')+\tfrac{5}{6}$. 

If $G'$ contains a forbidden graph $F$ as a subgraph then $G'=F$, since $G'$ is $2$-connected, and thus $\alpha(G')\geq \lb(G')-\tfrac{1}{6}$. 
If $G'$ contains no forbidden graph as a subgraph, then $\alpha(G')\geq \lb(G')$ since $G'$ is smaller than $G$, and thus not a counterexample to the theorem. 
Thus in both cases we deduce that 
\[
\alpha(G) = \alpha(G') + 1 \geq  \lb(G') - \tfrac{1}{6} + 1 =\lb(G).   
\] 
Moreover, if $G$ has at least three vertices of degree $2$, then $G'$ has at least one vertex of degree $2$, in which case we know that $G'$ is not one of the forbidden cubic graphs, implying 
$\alpha(G')\geq \lb(G')-\tfrac{1}{12}$, and 
$\alpha(G)\geq \lb(G) +  \tfrac{1}{12}$. 
Therefore, $G$ is not a counterexample, a contradiction.  

We may thus assume from now on that $v_1$ and $v_2$ have at least one common neighbor; let $v_3$ denote one such common neighbor. 
The following summarizes our progress so far.
\begin{equation}\label{eq:fivecycleobservation}
 u_1 \text{ and } u_2 \text{ are contained in an induced $5$-cycle }u_1u_2v_2v_3v_1.
\end{equation}

Next, let us show that $v_3$ is the only common neighbor of $v_1$ and $v_2$. 
Indeed suppose that they have a second common neighbor, say $v_4$. 
By Lemma~\ref{le:nodegreetwoinfourcycle} applied to the $4$-cycle $v_4v_1v_3v_2$, both $v_3$ and $v_4$ have degree $3$. Therefore $\left\{v_3,v_4\right\}$ is a cutset of size $2$. By Lemma~\ref{le:twovertexcutset}, this is only possible if $G-\left\{u_1,u_2,v_1,v_2,v_3,v_4\right\}$ is isomorphic to $K_1$ or $K_2$. 
In the first case $\alpha(G)=3$ and $\lb(G)= 3- \tfrac{1}{3}$. 
In the second case $G$ is isomorphic to $B_8$. 
Thus we obtain a contradiction in both cases. 
Therefore, $v_3$ is the only common neighbor of $v_1$ and $v_2$, as claimed.

Vertices $v_1$ and $v_2$ each have a neighbor outside $\left\{u_1,u_2,v_1,v_2,v_3\right\}$, let us denote them respectively $w_1$ and $w_2$. 
Note that $w_1 \neq w_2$, since $v_3$ is the only common neighbor of $v_1$ and $v_2$. 
Since $v_1$ and $v_2$ have degree $3$, so does $v_3$ by Lemma~\ref{le:notwocubeneighbors}, thus $v_3$ has a neighbor $w_3 \notin\left\{v_1,v_2,w_1,w_2 \right\}$.

Let $G'$ be the critical graph obtained by contracting the path $v_1u_1u_2v_2$ into the edge $v_1v_2$. By Lemma~\ref{le:oddsubdivision}, $\alpha(G)=\alpha(G')+1$. Furthermore, $\lb(G)=\lb(G')+\tfrac{5}{6}$. 

Let us show that $\{w_1,w_2,w_3\}$ is an independent set. 
Suppose not. 
We will derive a contradiction using the criticality of $G'$. 
Since $v_1,v_2,v_3$ form a triangle in $G'$, by symmetry, we may assume without loss of generality that $w_1w_3$ is an edge. 
Using Lemma~\ref{le:nodegreetwoinfourcycle} with the $4$-cycle $w_1w_3v_3v_1$ in $G'$, we deduce that the vertices $w_1$ and $w_3$ have degree $3$. 
Since $G$ has no triangle, at least one of $w_1,w_3$, say without loss of generality $w_3$, has a neighbor $x$ in $V(G') - \left\{v_1,v_2,v_3,w_1,w_2,w_3 \right\}$. 
Now consider a maximum independent set $S$ in $G'-w_3v_3$. 
Since $w_3v_3$ is a critical edge of $G'$, the set $S$ must contain both $w_3$ and $v_3$. 
This implies that $w_1,v_1,v_2\notin S$, and thus  $(S \cup \left\{ v_1\right\}) - \left\{v_3 \right\}$ is an independent set of $G'$ of size $|S|=\alpha(G'-w_3v_3)=\alpha(G')+1$, a contradiction. 
Thus, $\{w_1,w_2,w_3\}$ must be an independent set, as claimed.

In the remainder of the proof, we distinguish three cases depending on the local structure around the vertices $w_1,w_2,w_3$. 

{\bf Case 1: There are two vertices in $\left\{w_1,w_2,w_3\right\}$ that have no common neighbor.}\\
Let $w_i,w_j$ be two such vertices. 

Let $G_0 := G[\left\{u_1,u_2,v_1,v_2,v_3 \right\}]$ and let $H$ be the graph obtained from $G-V(G_0)$ by adding the edge $e^*=w_iw_j$. 
Note that $H$ is subcubic and triangle-free, because $w_i$ and $w_j$ have no common neighbor. 

Since $G_0$ is a $5$-cycle, we have $\lb(G_0)= \tfrac{ 5\cdot 10-2}{24}= 2$. 
Next, consider a maximum independent set $S$ of $G - V(G_0)$. 
Because of the edge $w_iw_j$, the set $S$ avoids at least one of $w_i,w_j$, say $w_i$. 
Thus $S$ can be extended to an independent set of $G$ by adding $v_i$ and a vertex from $\{u_1, u_2\}$ not adjacent to $v_i$. 
This shows that $\alpha(G)- \alpha(H) \geq 2$.  

Let $G_1,G_2,\ldots,G_k$ be an independence packing of $H$. 
If the edge $e^*$ is included in one of the graphs $G_1, \dots, G_k$, let us assume without loss of generality that it is in $G_k$. 

First, let us deal quickly with the case that $e^*$ is {\em not} in $G_k$. 
Then $\mathcal{P}=(G_0, G_1, \dots, G_k)$ is a near independence packing of $G$. 
Using the corresponding notations and Lemma~\ref{le:contract} combined with $\gamma \geq 2$ and $m_0=3$, we obtain 
\[
\alpha(G) 
\geq \lb(G) + \gamma - \lb(G_0) + \left\lceil \frac{m_0 + \delta + 3\nu + 4\sigma}{2} \right\rceil  \cdot \tfrac{1}{12}
\geq \lb(G) + \tfrac{1}{12}, 
\]
implying that $G$ is not a counterexample. 

Thus, the edge $e^*$ must be in $G_k$. 
Our analysis of this case will go along the same lines as that of Case~3 in the proof of Lemma~\ref{le:nobadsubgraph}. 
Let $m_i$ ($i \in \{0, 1, \dots, k\}$) denote the numbers of edges of $G$ that have exactly one endpoint in $V(G_i)$. 
Let $m:=|E(G)|-\sum_{i=0}^{k} |E(G_i)|$. 

Note that $G_k$ could be bad or isomorphic to one of $F_{11}$, $F_{14}^{(1)}$, $F_{14}^{(2)}$, $F_{19}^{(1)}$, $F_{19}^{(2)}$, $F_{22}$, since $G_k$ is not a subgraph of $G$. 
However, if this is the case then $G_k - e^*$ is connected and not isomorphic to $K_1$ or $K_2$, and thus $m_k \geq 3$ by Lemma~\ref{le:twoedgecutset}, implying that $m_k = 3$ and that $G_k$ cannot be isomorphic to one of $F_{14}^{(1)}$, $F_{14}^{(2)}$, $F_{22}$. 
Let $I$ be the indicator variable being equal to $1$ if $G_k$ is bad or isomorphic to one of $F_{11}$, $F_{19}^{(1)}$, $F_{19}^{(2)}$, and equal to $0$ otherwise. 

Applying Theorem~\ref{th:main} to each graph $G_i$ ($i\in \{0, 1, \dots, k\}$), we obtain that 
\[
\alpha(G_i) \geq \left\{ 
\begin{array}{ll}
\lb(G_i) - \frac{1}{12} & \textrm{ if } G_i \textrm{ is bad or isomorphic to one of } F_{11}, F_{19}^{(1)}, F_{19}^{(2)}  \\[.5ex]
\lb(G_i) + \frac{1}{6} & \textrm{ if } G_i \textrm{ is isomorphic to } K_1 \textrm{ or } K_2 \\[.5ex]
\lb(G_i) & \textrm{ otherwise }
\end{array}
\right.
\]

Note that the first outcome is only possible for $G_k$. 
Recall that $G_0$ is a $5$-cycle, and thus $\lb(G_0)=2=\alpha(G_0)$. 
Let $\sigma$ denote the number of indices $i\in \{1, \dots, k\}$ such that $G_i$ is isomorphic to $K_1$ or $K_2$. 
We obtain
\[
\alpha(G) = \sum_{i=0}^k \alpha(G_i)  
\geq \sum_{i=0}^k \lb(G_i) + \frac{1}{12}(2\sigma - I) 
= \lb(G) + \frac{1}{12}(m - k + 2\sigma - I).  
\] 
We will show that $m - k + 2\sigma - I \geq 0$, and moreover $m - k + 2\sigma - I \geq 1$ in case $G$ has at least three degree-$2$ vertices, implying that $G$ is not a counterexample, as desired. 

Recall that $m_0=3$. 
Observe that 
\[
2m 
\geq  \sum_{i=0}^{k} m_i - 2
=  \sum_{i=1}^{k} m_i + 1 
\geq 2\sigma + 3(k-\sigma) + 1. 
\]
Above, the $-2$ comes from the edge $e^*$, and we used the fact that $m_i \geq 2$ for each $i\in \{1, \dots, k\}$, and furthermore $m_i \geq 3$ when $G_i$ is not isomorphic to $K_1$ or $K_2$. 
Note that this implies $m \geq k+1$, since $m$ is an integer. 
Thus if $\sigma \geq 1$ or $I=0$ we directly obtain $m - k + 2\sigma - I \geq 1$. 
Hence, we may assume $\sigma=0$ and $I=1$. 
Now, if $k\geq 2$, the above inequality implies $2m \geq 3k + 1 \geq 2k + 3$ and thus $m \geq k+2$, and hence $m - k + 2\sigma - I \geq 1$. 
So we may assume $k=1$. 
In this case we know $m \geq 2$, which already implies $m - k + 2\sigma - I \geq 0$. 
Thus $G$ must have at least three degree-$2$ vertices. 
Moreover, we must have $m=2$, that is, $G$ is obtained from the union of $G_0$ and $G_1$ by removing $e^*$ and adding three edges each having one endpoint in $V(G_0)$ and the other in $V(G_1)$. 
If $G_1$ is isomorphic to one of $F_{11}$, $F_{19}^{(1)}$, $F_{19}^{(2)}$, then it follows that the only degree-$2$ vertices in $G$ are $u_1, u_2$, and thus $G$ has less than three degree-$2$ vertices, a contradiction. 
If $G_1$ is bad, then $G$ has exactly five degree-$2$ vertices, namely $u_1, u_2$ and three others in $V(G_1)$. 
However, this contradicts the fact that $G$ has an even number of degree-$2$ vertices (c.f.\ Lemma~\ref{le:evennumberofdegreetwos}). 

This finishes the proof that every two vertices in $\left\{w_1,w_2,w_3\right\}$ have a common neighbor. Note that therefore either $w_1,w_2,w_3$ have a common neighbor $x_{123}$, or: each two vertices $w_i,w_j \in \left\{w_1,w_2,w_3\right\}$ have a \emph{private} common neighbor $x_{ij}$, meaning that $x_{12},x_{23},x_{13}$ are pairwise distinct.\\

{\bf Case 2: Every two vertices in $\left\{w_1,w_2,w_3\right\}$ have a common neighbor but no vertex is a common neighbor of all three vertices.} \\
By Lemma~\ref{le:notwocubeneighbors}, each private common neighbor $x_{ij}$ has degree $3$, and therefore has a neighbor $y_{ij}$ in  $V(G') \backslash\left\{v_1,v_2,v_3,w_1,w_2,w_3, x_{12},x_{23},x_{13}\right\}$. Using Lemma~\ref{le:twovertexcutset}, it is not difficult to check that $y_{12},y_{23},y_{13}$ must be pairwise distinct.
 
Next, we construct a new graph as follows. 
Starting from $G$, we first contract $\{u_1,u_2,v_1,v_2,v_3,w_1,w_2,w_3, x_{12},x_{23},x_{13}\}$ into a single vertex $x$. 
Let $G^*_2$ be the resulting graph. 
Note that $x$ might be contained in some triangles. 
We choose an edge $e\in \left\{xy_{12},xy_{23},xy_{13}\right\}$ that is contained in all triangles (if any). 
It is easily checked that such an edge exists using the fact that every triangle contains $x$.  
We subdivide the edge $e$ into a $3$-edge path and call the resulting graph $G_2$.   
Note that $G_2$ is triangle-free and that by construction, $G$ is an $8-$augmentation of $G_2$ (see Figure~\ref{fig:8-augmentation}).  
Hence, we can apply Lemma~\ref{le:8augmentation}, yielding that $G_2$ is $2$-connected, and $\alpha(G)=\alpha(G_2) + 3$ and $\lb(G)=\lb(G_2) + 3$. 

The graph $G_2$ cannot be bad nor dangerous since $G$ is neither bad nor dangerous.  
Furthermore, $G_2$ has at least two vertices of degree $2$, which implies that $G_2$ has no forbidden graph as a subgraph, since $G_2$ is $2$-connected. 

Finally, $G_2$ must be critical, for the following reason. 
First, note that $\alpha(G^*_2) = \alpha(G_2)-1 = \alpha(G)-4$, and that $G_2$ is critical if and only if $G^*_2$ is, by Lemma~\ref{le:oddsubdivision}. 
Thus, let us show that $G^*_2$ is critical. 
Consider an edge $ab$ of $G^*_2$.  

First suppose $a,b \neq x$. 
We consider the edge $ab$ in the graph $G'$. 
Recall that $\alpha(G') = \alpha(G) - 1 = \alpha(G^*_2) + 3$. 
Since $G'$ is critical, there is an independent set $S$ of $G'-ab$ of size $\alpha(G')+1$ (and containing both $a$ and $b$). 
If $S$ avoids $y_{12},y_{23},y_{13}$ then $S$ contains at most four vertices of $X:=\left\{v_1,v_2,v_3,w_1,w_2,w_3, x_{12},x_{23},x_{13}\right\}$, as is easily checked, and thus $(S-X)\cup\{x\}$ is an independent set of $G^*_2-ab$ of size at least $\alpha(G')+1 -4 +1= \alpha(G^*_2)+1$. 
If $S$ contains some vertex $y_{ij}$ then $S$ avoids $x_{ij}$ and then one can check that $S$ may only contain up to three vertices of $X$. 
Thus $S-X$ is then an independent set of $G^*_2-ab$ of size at least $\alpha(G')+1 -3= \alpha(G^*_2)+1$.  

Next assume $ab$ is of the form $xy_{ij}$. 
Let $S$ be an independent set of $G'-x_{ij}y_{ij}$ of size $\alpha(G')+1$, thus $S$ contains $x_{ij}$ and $y_{ij}$. 
If $S\cap \left\{y_{12},y_{23},y_{13}\right\} = \{y_{ij}\}$, then noting again that $S$ contains at most four vertices of $X$ (counting $x_{ij}$), we see that $(S-X)\cup\{x\}$ is an independent set of $G^*_2-ab$ of size at least $\alpha(G')+1 -4 +1= \alpha(G^*_2)+1$. 
If $S\cap \left\{y_{12},y_{23},y_{13}\right\} = \{y_{ij}, y_{i'j'}\}$, then it can be checked that $S$ contains up to three vertices of $X$. 
Thus $(S-(X \cup \{y_{i'j'}\}) \cup\{x\}$ is an independent set of $G^*_2-ab$ of size at least $\alpha(G')+1 -4 +1= \alpha(G^*_2)+1$. 
Finally, the case where $S$ contains all three vertices $y_{12},y_{23},y_{13}$ cannot happen, because $S$ avoids $v_k$ for some $k\in \{i,j\}$ (since the edge $v_iv_j$ is always there in $G'$), and it follows that $(S-\{x_{ij}\}) \cup \{w_k\}$ is an independent set of $G'$ of size $|S|=\alpha(G')+1$, a contradiction. 

In summary, in all possible cases we found an independent set of $G^*_2-ab$ of size $\alpha(G^*_2)+1$. 
Hence, $G^*$ is critical, as claimed. 
We conclude that $G_2$ satisfies the assumptions of Theorem~\ref{th:main}, yielding $\alpha(G_2)\geq \lb(G_2)$, since $G_2$ is smaller than $G$ and thus not a counterexample.  
Hence, $\alpha(G)=\alpha(G_2) +3 \geq  \lb(G_2) + 3 = \lb(G)$. 

Furthermore, if $G$ has at least three vertices of degree $2$, then $G_2$ also has at least three degree vertices of degree $2$, and then Theorem~\ref{th:main} gives  $\alpha(G_2)\geq \lb(G_2) + \tfrac{1}{12}$ since $G_2$ is not dangerous, implying $\alpha(G)\geq \lb(G) + \frac{1}{12}$, as desired. \\

{\bf Case 3: There exists a common neighbor $x_{123}$ of $w_1,w_2$ and $w_3$. } \\
Each of $w_1,w_2,w_3$ has two neighbors of degree $3$ in $G$, so by Lemma~\ref{le:notwocubeneighbors} it follows that $w_1,w_2$ and $w_3$ each have degree $3$. 
Let $x_i$  denote the neighbor of $w_i$ outside $\left\{v_1,v_2,v_3,w_1,w_2,w_3,x_{123} \right\}$, for $i=1,2,3$. Recall that $G'$ is the critical graph obtained by contracting the path $v_1u_1u_2v_2$ into the edge $v_1v_2$; the benefit of working with $G'$ rather than $G$ is that we can treat the vertices in the triangle $v_1v_2v_3$ symmetrically.

First, we argue that $x_1, x_2, x_3$ are pairwise distinct. 
Arguing by contradiction, suppose $x_i=x_j$ for two distinct indices $i,j$, and let $k$ denote the remaining third index. 
Let $S$ be a maximum independent set of $G'-\left\{w_ix_{123}\right\}$.  
By criticality of the edge $w_ix_{123}$ of $G'$, the set $S$ contains $w_i$ and $x_{123}$. 
Thus $S$ avoids $v_i, w_j, x_j, w_k$. 
Note that $S$ contains at most one of $v_j, v_k$ because of the edge $v_jv_k$ of $G'$. 
Furthermore, if $v_j \in S$, we can replace $v_j$ with $v_k$ in $S$, since $v_i, w_k \notin S$. 
Thus we may assume $v_j \notin S$. 
It follows that $(S - \{x_{123}\})\cup \{w_j\}$ is an independent set of $G'$ of size $|S|=\alpha(G')+1$, a contradiction. 

A similar argument shows that $x_1, x_2, x_3$ form an independent set, as we now explain. 
Suppose for a contradiction that $x_i$ and $x_j$ are adjacent, and let $k$ denote the remaining third index. 
Let $S$ be a maximum independent set of $G'-\left\{w_ix_{123}\right\}$.  
By criticality of the edge $w_kx_{123}$ of $G'$, the set $S$ contains $w_k$ and $x_{123}$. 
Thus $S$ avoids $v_k, w_i, w_j$. 
The set $S$ contains at most one of $x_i, x_j$ because of the edge $x_ix_j$. 
Exchanging $i$ and $j$ if necessary, we may assume that $x_i \notin S$. 
Also, $S$ contains at most one of $v_i, v_j$ because of the edge $v_iv_j$ of $G'$. 
Furthermore, if $v_i \in S$, we can replace $v_i$ with $v_j$ in $S$, since $w_j \notin S$. 
Thus we may assume $v_i \notin S$. 
It follows that $(S - \{x_{123}\})\cup \{w_i\}$ is an independent set of $G'$ of size $|S|=\alpha(G')+1$, a contradiction. 	 
	 
Next, we argue that $x_1, x_2, x_3$ all have degree $3$. 
Arguing by contradiction, suppose some $x_i$ has degree $2$, and let $y_i$ denote its neighbor distinct from $w_i$.  
Since $w_i$ has degree $3$, Lemma~\ref{le:notwocubeneighbors} implies that $y_i$ has degree $2$. 
Thus $x_i$ and $y_i$ are two adjacent vertices of degree $2$. 
By our observation~\eqref{eq:fivecycleobservation} at the beginning of this proof, we know that there is a $5$-cycle of $G$ containing the edge $x_iy_i$. 
Such a cycle necessarily contains $w_i$, and it follows that $y_i=x_j$ for some $j\in \{1,2,3\}-\{i\}$, contradicting the fact that $x_ix_j \notin E(G)$.  

Having gathered all necessary structural information about $G$, we are ready to finish the proof. 
Let $Z:=\left\{u_1,u_2,v_1,v_2,v_3,w_1,w_2,w_3,x_{123}, x_1,x_2\right\}$.\footnote{At first sight it might seem odd that $Z$ does not include $x_3$. Note however that $Z$ consists of the closed neighborhood of the $6$-cycle $v_1v_3v_2w_2x_{123}w_1$. In Section~\ref{subsection:3connected} we will use closed neighborhoods of even cycles more often.}	 
Let $\mathcal{P}=(G_0, G_1, \dots, G_k)$ be a near independence packing of $G$ with $G_0 = G[Z]$. 
Since every maximum independent set of $G-Z$ can be extended to an independent set of $G$ by adding $\left\{w_1,w_2,v_3,u_1\right\}$, it follows that $\gamma \geq 4$.  
We also have $\lb(G_0)=\tfrac{6\cdot 9+3\cdot 10+ 2 \cdot 11 -2}{24}=4 + \tfrac{1}{3}$ and $m_0=5$, as follows from the properties of $G$ established in the previous paragraphs. 
Using Lemma~\ref{le:contract}, we obtain 
\begin{equation}
\label{eq:case3_delta}
\alpha(G) \geq \lb(G) +  \gamma  -\lb(G_0) +  \left(m_0 - \delta + \sigma \right) \cdot \tfrac{1}{12} 
= \lb(G)- \frac{1}{3} +  \left(5 - \delta + \sigma \right) \cdot \tfrac{1}{12}. 
\end{equation}
Recall that, since $G$ is a counterexample, we have $\alpha(G) < \lb(G) + \frac{1}{12}$. 
Thus, we must have $\delta\geq 1$. 
By Lemma~\ref{le:contract}, we also have 
\[
\alpha(G) 
\geq \lb(G) +  \gamma  -\lb(G_0) + \left\lceil \frac{m_0 + \delta + 3\nu + 4\sigma}{2} \right\rceil  \cdot \tfrac{1}{12}   
\geq \lb(G) - \frac{1}{3} + \left\lceil \frac{5 + \delta + 3\nu + 4\sigma}{2} \right\rceil  \cdot \tfrac{1}{12}.  
\]
Thus, it follows that $\delta\in \{1,2,3\}$, $\nu = 0$ and $\sigma = 0$.  
That is, $G_1, \dots, G_k$ are all dangerous graphs, and $k\in \{1,2,3\}$. 

If $\delta=1$, then because $G_1$ has exactly five degree-$2$ vertices and $m_0=5$, we deduce that $u_1, u_2$ are the only two vertices of $G$ with degree $2$.  
Thus in this case we know that $\alpha(G) < \lb(G)$ since $G$ is a counterexample. 
However, this contradicts~\eqref{eq:case3_delta} when $\delta=1$. 

If $\delta=2$, then we reach a contradiction as follows. 
In total, there are exactly ten degree-$2$ vertices in $G_1$ and $G_2$ (five in each), exactly five of which are incident in $G$ to an edge with an endpoint in $G_0$ (since $m_0=5$). 
Consider the remaining five vertices. 
Every edge $e$ of $G$ having both endpoints in $V(G_1) \cup V(G_2)$ but which is not in $G_1$ nor in $G_2$ connects two of these vertices. 
We deduce that an odd number of these five vertices have degree $2$ in $G$. 
Since only two vertices of $V(G_0)$ have degree $2$ in $G$, namely $u_1$ and $u_2$, it follows that $G$ has an odd number of degree-$2$ vertices, contradicting Lemma~\ref{le:evennumberofdegreetwos}.

Thus we may assume that $\delta=3$. 
By Lemma~\ref{le:twoedgecutset} (and also using that $G$ is $2$-connected and that no dangerous graph is isomorphic to $K_1$ or $K_2$), we know that for all $i\in \left\{1,2,3\right\}$, there are at least three edges with precisely one endpoint in $V(G_i)$. 
Suppose that for some $i \in \left\{1,2,3\right\}$, either $G[V(G_i)]$ contains an edge that is not in $G_i$ or: $G$ has \emph{more} than three edges with precisely one endpoint in $V(G_i)$. Then (also using that $\sigma=\nu=0$, $\delta=3$ and $m_0=5$) we gain at least an additive term of $2$ in inequality (\ref{eq:estimating_m_one}), which propagates to an extra additive term of $\tfrac{1}{12}$ in the lower bound of Lemma~\ref{le:contract}. Plugging this into the argument given just below inequality~(\ref{eq:case3_delta}) yields $\alpha(G) \geq \lb(G) + \frac{1}{12}$; contradiction.

So we may assume that for each $i\in \left\{1,2,3\right\}$, $G_i$ is an induced subgraph of $G$ and there are precisely $m_i=3$ edges with one endpoint in $G_i$. This means that $V(G_1),V(G_2),V(G_3)$ each contain precisely two vertices that have degree $2$ in $G$. 
Moreover, the number of edges in $E(G)- \bigcup _{0 \leq i \leq 3} E(G_i)$ is $\frac{1}{2}\sum_{i=0}^{3} m_i = \frac{5+ 3+3+3}{2}= 7$, of which $m_0=5$ edges have an endpoint in $G_0$. Thus $G$ has precisely two edges which have their endpoints in distinct $V(G_i)$, $V(G_j)$ with $i, j\in \left\{1,2,3\right\}$, which we will call the two \emph{heterogeneous edges}. 
Note that these two edges are independent since $G$ has maximum degree $3$. 

Without loss of generality, $x_3 \in V(G_3)$. Let $y_1,y_2$ denote the neighbors of $x_1$ in $V(G_1)\cup V(G_2)\cup V(G_3)$. Note that $y_1$ and $y_2$ cannot both be in $V(G_3)$, for else $w_2x_2$ would be a bridge (since $m_3=3$).  
Let us call $x_3,y_1$ and $y_2$ the \emph{pre-blocked vertices}. 

For each $i\in \left\{1,2,3\right\}$, we have $m_i=3$, and the graph $G_i$ has at most two pre-blocked vertices. 
Using these facts, a quick case analysis shows that it is possible to select a vertex $y_3$ incident to one heterogenerous edge and another vertex $y_4$ incident to the other heterogeneous edge such that for all $i\in \left\{1,2,3\right\}$,  the set $\left\{x_3,y_1,y_2,y_3,y_4\right\} \cap V(G_i)$ has size at most $2$. Let us  call $\left\{x_3,y_1,y_2,y_3,y_4\right\}$ the set of \emph{blocked vertices}.

Since each of $G_1,G_2,G_3$ contains at most two blocked vertices, it follows from property \ref{le:dangerousproperties_T_two_vertices} of Lemma~\ref{le:dangerousproperties} that for each $i \in \left\{1,2,3\right\}$, there exists an independent set $S_i$ of $G_i$ that avoids all blocked vertices and satisfies $\alpha(G_i)\geq \lb(G_i)$. Because $S:=S_1\cup S_2 \cup S_3$ avoids all neighbors of $\left\{x_1,w_3\right\}$ and also avoids a vertex of every heterogeneous edge, it follows that $S\cup \left\{x_1,w_3,w_2,v_1,u_2\right\}$ is an independent set of $G$. From this we obtain
 $$\alpha(G) \geq 5 + \sum_{i=1}^{3} \lb(G_i).$$  
Since furthermore $m= |E(G)|-\sum_{i=0}^{3}|E(G_i)|=7$ and $\lb(G)= \sum_{i=0}^{3} \lb(G_i) + (3-m)\frac{1}{12}= 4 + \sum_{i=1}^{3} \lb(G_i)$, we conclude that $\alpha(G) \geq \lb(G)+1$; contradiction.
\end{proof}

Our progress so far can be summarized in the following lemma.
\begin{lemma}
$G$ is cubic and moreover $3$-connected.
\end{lemma}
\begin{proof}
Recall that $G$ is $2$-connected. 
Lemmas~\ref{le:nothreeconsecutivedegreetwo},~\ref{le:notwocubeneighbors} and~\ref{le:notwoadjacentdegreetwo} together imply that $G$ has no vertex of degree $2$.  
Combined with Lemma~\ref{le:twovertexcutset}, this yields that $G$ is $3$-connected.
\end{proof}

\subsection{The $3$-connected case}\label{subsection:3connected}

Let us first remark that from now on, in order to obtain the desired contradiction, it suffices to show that  
\[
\alpha(G) \geq \lb(G) -\tfrac{1}{12}.
\] 
Indeed, because $G$ is cubic, we know that $|V(G)|$ is even and $\lb(G) = \tfrac{3}{8}|V(G)| - \tfrac{1}{12}$. 
Thus, $\alpha(G) \geq \lb(G) -\tfrac{1}{12}$ implies $8\alpha(G) \geq 3|V(G)| -\tfrac{16}{12}$, and hence $8\alpha(G) \geq 3|V(G)|$, that is, $\alpha(G) \geq \lb(G) + \tfrac{1}{12}$, since $8\alpha(G)$ and $3|V(G)|$ are both even integers. 

Using the fact that $G$ is cubic, we obtain the following corollary from Lemma~\ref{le:contract}. 

\begin{lemma}\label{le:removeForcubic}
Fix some near independence packing $\mathcal{P}=(G_0, G_1, \dots, G_k)$ of $G$ such that $G_0$ is connected and has at least three vertices. 
Then 
\[
\alpha(G) \geq \lb(G) + \gamma  + \tfrac{2  - 9 \cdot |V(G_0)|}{24} + \left(m_0 - 2\delta + 2 \sigma \right)\cdot \tfrac{1}{24}, 
\]
and
\[
\alpha(G) \geq \lb(G) + \gamma  + \tfrac{2  - 9 \cdot |V(G_0)|}{24} + \left(3\delta + 3\nu + 4\sigma \right)\cdot \tfrac{1}{24}. 
\]
\end{lemma}
\begin{proof}
Because $G$ is cubic, it follows that $\lb(G_0)=\frac{|V(G_0)|\cdot 9 + m_0-2}{24}$. 
The two inequalities follows then respectively from the first and the third inequality from Lemma~\ref{le:contract}. 
\end{proof}

\begin{corollary}\label{cor:removeForcubic}
Fix some near independence packing $\mathcal{P}=(G_0, G_1, \dots, G_k)$ of $G$ such that $G_0$ is connected, $3 \leq |V(G_0)| \leq 14$, and $\gamma \geq 5$. 
Then $G=G_0$, $|V(G_0)| = 14$, and $\alpha(G)=\gamma=5$. 
\end{corollary}
\begin{proof}
If $G \neq G_0$ then $\delta + \nu + \sigma \geq 1$, and Lemma~\ref{le:removeForcubic} yields 
\[
\alpha(G) \geq \lb(G) + \gamma  + \tfrac{2  - 9 \cdot 14}{24} + \left(3\delta + 3\nu + 4\sigma\right)\cdot \tfrac{1}{24}
\geq \lb(G) -\frac{1}{12}, 
\]
a contradiction. 
Thus $G = G_0$. 
If $|V(G_0)| < 14$ then $\alpha(G) \geq \lb(G) + \gamma  + \tfrac{2  - 9 \cdot 13}{24} \geq \lb(G) -\frac{1}{12}$, 
a contradiction.  Thus $|V(G_0)|=14$. Finally, if $\alpha(G) > 5$, then $\alpha(G) >  \frac{9\cdot 14-4}{24} = \lb(G) -\frac{1}{12}$, a contradiction. 
\end{proof}

\begin{lemma}\label{le:nofourcycle}
$G$ has no $4$-cycle. 
\end{lemma}
\begin{proof}
Suppose that $G$ has a $4$-cycle $v_1v_2v_3v_4$. For each $i\in \left\{1,2,3,4\right\}$, let $w_i \notin \left\{v_1,v_2,v_3,v_4\right\}$ denote the third neighbor of $v_i$. If $w_1=w_3$ and $w_2=w_4$, then $\left\{w_1,w_2\right\}$ is a $2$-cutset separating $G$ into two components, neither of which is isomorphic to $K_1$ or $K_2$, contradicting Lemma~\ref{le:twovertexcutset}. 
Thus, without loss of generality we may assume that $w_1\neq w_3$. 

Consider a near independence packing $\mathcal{P}=(G_0, G_1, \dots, G_k)$ of $G$ with $G_0 = G[\left\{v_1,v_2,v_3,v_4,w_1,w_3\right\}]$. 
Since each maximum independent set $S$ of $G-V(G_0)$ can be extended to the independent set $S \cup \left\{v_1,v_3\right\}$ of $G$, we have $\gamma\geq 2$.  
Clearly, $G \neq G_0$, since $G_0$ is not a counterexample to Theorem~\ref{th:main}. 
Then, from Lemma~\ref{le:removeForcubic} it follows that
\[
\alpha(G) \geq \lb(G) + 2  + \tfrac{2  - 9 \cdot 6}{24} + \left(3\delta + 3\nu + 4\sigma\right)\cdot \tfrac{1}{24}
\geq \lb(G) + \left(-4 + 3\delta + 3\nu + 4\sigma\right)\cdot \tfrac{1}{24}
\geq \lb(G) -\frac{1}{12}, 
\]
a contradiction. 
\end{proof}

\begin{lemma}\label{le:nosixcycle}
$G$ has no $6$-cycle. 
\end{lemma}
\begin{proof}
Suppose that $G$ has a $6$-cycle $C:=v_1v_2v_3v_4v_5v_6$. This is an induced $6$-cycle, because $G$ has girth at least $5$. Let $w_i$ denote the neighbor of $v_i$ in $V(G)-V(C)$, for each $i\in \{1, \dots, 6\}$. Since the girth of $G$ is at least $5$, we have $w_i\notin \left\{ w_{i+1}, w_{i+2} \right\}$ (where the indices are taken cyclically). We may however have $w_{i}=w_{i+3}$ for some $i\in \{1, \dots, 6\}$, in which case we call $w_i$ a \emph{sneaky vertex}. Similarly, each edge of the form $w_iw_{i+2}$ is called a \emph{sneaky edge}. 

We define two \emph{parity classes}, $C_1:=\left\{v_1,v_3,v_5\right\}$ and $C_2:=\left\{v_2,v_4,v_6\right\}$ and . If $w_iw_{i+2}$ is a sneaky edge then we say that $w_iw_{i+2}$ \emph{belongs to} the parity class $C_{i \text{ mod } 2}$.

\textbf{Case $1.1$: Some parity class has no sneaky edge and there is at most one sneaky vertex.} \\
Without loss of generality, we may assume that $C_1$ has no sneaky edge. 
If there is a sneaky vertex, we assume without loss of generality that it is $w_2=w_5$. 

Consider a near independence packing $\mathcal{P}=(G_0, G_1, \dots, G_k)$ of $G$ with $G_0 = G[\left\{v_1,v_2,v_3,v_4,v_5,v_6, w_1,w_3,w_5\right\}]$. 
Note that $w_1,w_3,w_5$ are pairwise distinct, thus $|V(G_0)|=9$. 
Every maximum independent set of $G-V(G_0)$ can be extended to an independent set of $G$ by adding $\left\{v_1,v_3,v_5 \right\}$, hence $\gamma \geq 3$.  
Since there are no sneaky edges in the parity class $C_1$, we either have $m_0=9$ if there is no sneaky vertex, or $m_0=7$ if there is one (namely, $w_2$). 
Thus, $m_0 \geq 7$. 
Applying Lemma~\ref{le:removeForcubic}, we obtain
\[
\alpha(G) \geq \lb(G) + \gamma  + \tfrac{2  - 9 \cdot |V(G_0)|}{24} + \max\left(m_0 - 2\delta + 2\sigma, 3\delta + 3\nu + 4\sigma\right)\cdot \tfrac{1}{24} 
\geq \lb(G) -\frac{1}{12}, 
\]
a contradiction. 

\textbf{Case $1.2$: Some parity class has no sneaky edge and there are two or three sneaky vertices.} \\
Without loss of generality, we may assume that $C_1$ has no sneaky edge. 
Also, without loss of generality $w_2=w_5$ and $w_3=w_6$ are two sneaky vertices, and possibly $w_1=w_4$ in case there is a third one. 
Let $x_2 \notin \left\{v_2,v_5\right\}$ denote the third neighbor of $w_2$. 

Consider a near independence packing $\mathcal{P}=(G_0, G_1, \dots, G_k)$ of $G$ with $G_0 = G[\left\{v_1,v_2,v_3,v_4,v_5,v_6, w_2,w_3,x_2\right\}]$. 
Note that $|V(G_0)|=9$. 
Every maximum independent set of $G-V(G_0)$ can be extended to an independent set of $G$ by adding $\left\{v_3,w_2,v_6 \right\}$, hence $\gamma \geq 3$.  
Since $G$ has girth at least $5$, the vertices $v_1,v_4,w_3$ are pairwise non-adjacent, and so are $v_1,v_4,x_2$. 
Note however that we could possibly have the edge $x_2w_3$. 
It follows that either $m_0=3$ (in case $x_2w_3 \in E(G)$) or $m_0=5$ (in case $x_2w_3 \notin E(G)$). 
Applying Lemma~\ref{le:removeForcubic}, we obtain
\[
\alpha(G) \geq \lb(G) + \gamma  + \frac{2  - 9 \cdot |V(G_0)|}{24} + \max\left(m_0 - 2\delta + 2\sigma, 3\delta + 3\nu + 4\sigma\right)\cdot \tfrac{1}{24}. 
\]
Since $m_0 \geq 3$, this is at least $\lb(G) -\frac{1}{12}$ (yielding a contradiction) unless $\delta+\nu=1$ and $\sigma=0$. 
Let us analyze that last possibility. 
First note that $k=\delta+\nu+\sigma=1$. 
Moreover, $w_1\neq w_4$ (that is, there are precisely two sneaky vertices), since $\sigma=0$ and thus $G_1$ has minimum degree $2$. 

Suppose first that $\delta=1$. 
Consider the dangerous graph $G_1$, which has five degree-$2$ vertices. 
Let us show that $\alpha(G) \geq \lb(G_1) + 4$. 
If $G_1$ is an induced subgraph of $G$, then we must have $m_0=m_1=5$.  
The endpoints in $V(G_1)$ of the edges between $V(G_0)$ and $V(G_1)$ are $x_3,w_1,w_4,y_2,y_2^*$, where $x_3$ is a neighbor of $w_3$ and $y_2,y_2^* \neq w_2$ are two neighbors of $x_2$. 
Since $y_2y_2^*$ is not an edge (otherwise $x_2y_2y_2^*$ is a triangle), it follows from Lemma~\ref{le:dangerousproperties}.\ref{le:dangerousproperties_T} that there is an independent set $S$ of $G_1$ of size at least $\lb(G_1)$ that avoids $\left\{w_1,w_4,x_3\right\}$. 
Hence, $S \cup \left\{v_1,w_2,v_4,w_3\right\}$ is an independent set of $G$ of size at least $\lb(G_1) + 4$. 
If $G_1$ is not an induced subgraph of $G$, then since $m_1 \geq 3$ there is exactly one edge $z_1z_2$ of $G$ that links two degree-$2$ vertices of $G_1$, and $m_0=m_1=3$. 
It follows that $x_2w_3 \in E(G)$ (since $m_0=3$). 
Let $y_2$ be the neighbor of $x_2$ in $V(G_1)$. 
Then $w_1,w_4,y_2,z_1,z_2$ are the five degree-$2$ vertices of $G_1$, and by Lemma~\ref{le:dangerousproperties}.\ref{le:dangerousproperties_T_two_vertices} there is an independent set $S$ of $G_1$ of size at least $\lb(G_1)$ that avoids $\left\{w_1,w_4\right\}$. 
Hence, $S \cup \left\{v_1,w_2,v_4,w_3\right\}$ is again an independent set of $G$ of size at least $\lb(G_1) + 4$. 
Therefore, in both cases we obtain
\[
\alpha(G) \geq \lb(G_1) + 4 = \lb(G) - \frac{9 \cdot |V(G_0)|}{24} + \frac{5}{24} + 4 \geq \lb(G) -\frac{1}{12}, 
\]
a contradiction. 

Next, suppose that $\nu=1$. If $m_0=5$ then we directly obtain $\alpha(G)\geq \lb(G)-\frac{1}{12}$. 
So we may assume that $m_0=3$, which implies that $x_2$ and $w_3$ are adjacent. 
Let $y_2 \notin \left\{w_2,w_3\right\}$ denote the third neighbor of $x_2$ and let $z_2,z_2^*$ denote the two neighbors of $y_2$ distinct from $x_2$.  
Consider a near independence packing $\mathcal{P'}=(G'_0, G'_1, \dots, G'_k)$ of $G$ with $G'_0 = G[\left\{v_1,v_2,v_3,v_4,v_5,v_6,w_1,w_2,w_3,w_4,x_2,y_2,z_2,z_2^* \right\}]$. 
It could be that not all vertices in the latter set are pairwise distinct, but we do know that $|V(G'_0)| \leq 14$, and that $S_0:=\left\{v_1,w_2,v_4,w_3,y_2  \right\}$ is an independent set consisting of five pairwise distinct vertices.  
Every independent set of $G-V(G'_0)$ can be extended to an independent set of $G$ by adding $S_0$. 
Hence, $\alpha(G)-\alpha(G-V(G'_0))\geq 5$. 
By Corollary~\ref{cor:removeForcubic}, we find that $G=G'_0$, and that all vertices in the set $\left\{v_1,v_2,v_3,v_4,v_5,v_6,w_1,w_2,w_3,w_4,x_2,y_2,z_2,z_2^* \right\}$ are pairwise distinct. 
However, in the subgraph of $G'_0$ that we identified so far (see Figure~\ref{fig:nosixcycle_case1point2}), $z_2,z_2^*,w_1$ and $w_4$ have degree $1$ and the remaining vertices have degree $3$. 
It is impossible to complete this graph to a cubic graph without creating a $4$-cycle. Contradiction.

 \begin{figure}
\centering
\includegraphics[width=0.6\textwidth]{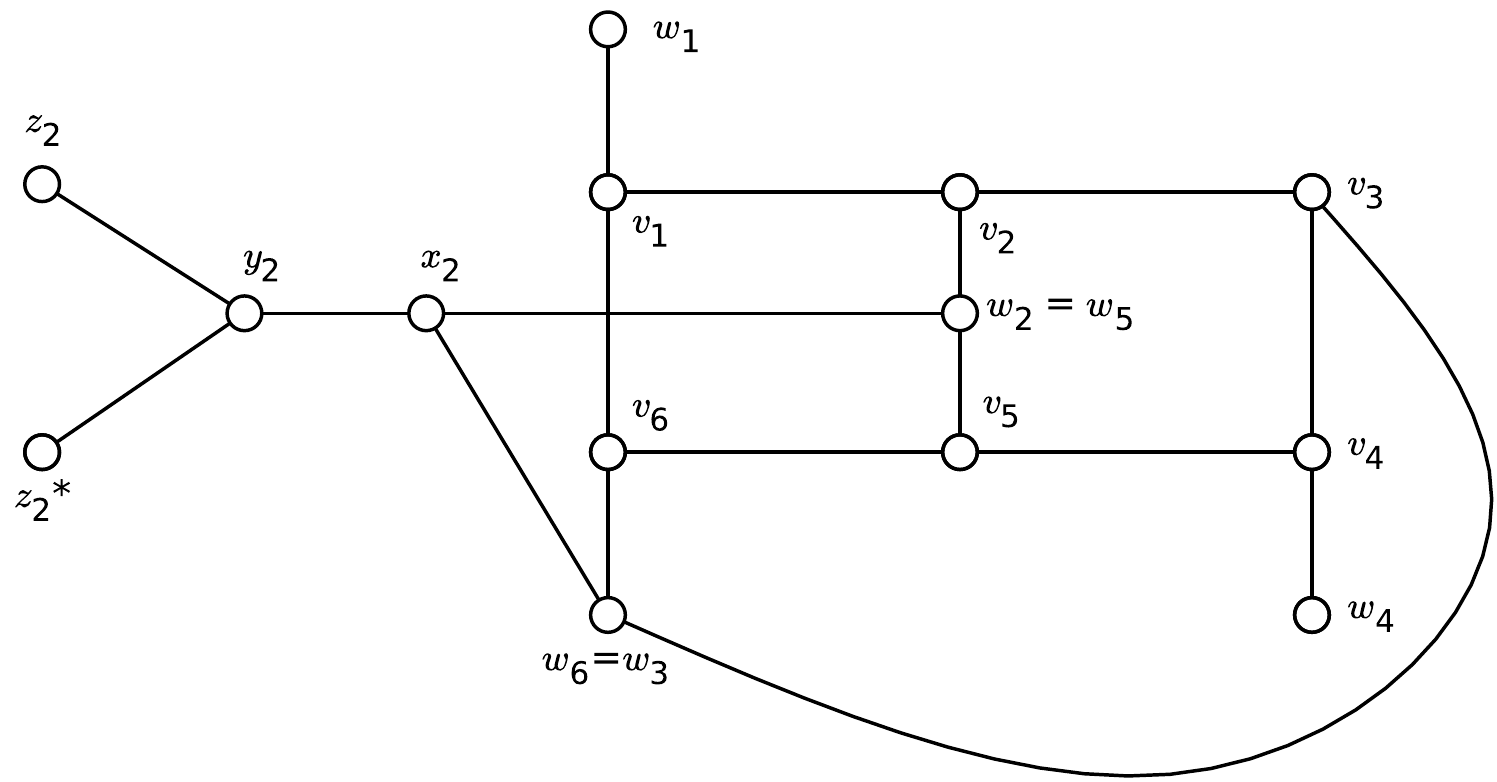}
\caption{Illustration of Case $1.2$ of Lemma~\ref{le:nosixcycle}, when $\nu=1$ and $G=G'_0$.  \label{fig:nosixcycle_case1point2}}
\end{figure}

\textbf{Case $2$: Both parity classes have at least one sneaky edge, and one has exactly one sneaky edge.} \\
Without loss of generality, $C_1$ has exactly one sneaky edge, say $w_1w_3$.  
If $w_1=w_4$ or $w_3=w_6$, then there is a $4$-cycle, so $w_1\neq w_4$ and $w_3\neq w_6$. Thus, either there is exactly one sneaky vertex (if $w_2=w_5$), or none (if $w_2\neq w_5$). 

Consider a near independence packing $\mathcal{P}=(G_0, G_1, \dots, G_k)$ of $G$ with $G_0 = G[\left\{v_1,v_2,v_3,v_4,v_5,v_6, w_1,w_3,w_5\right\}]$. 
Note that $|V(G_0)|=9$ and $m_0 \leq 7$. 
Suppose first that $w_2\neq w_5$. If $m_0 < 7$ then $G[\left\{ w_1,w_3,w_5,v_2,v_4,v_6 \right\}]$ must contain an edge different from $w_1w_3$. But then $w_2=w_5$ or there is a triangle or a $4$-cycle; contradiction. So $m_0=7$. Furthermore, exactly as in case $1.1$ we obtain that $\gamma \geq 3$. Applying Lemma~\ref{le:removeForcubic} yields 
\begin{align*}
\alpha(G) &\geq \lb(G) + \frac{-7 + \max\left(m_0 - 2 \delta + 2\sigma, 3\delta + 3\nu + 4\sigma  \right)}{24} \\
&\geq \lb(G) + \frac{\max(-2\delta,3\delta -7)}{24} \\
&\geq \lb(G) -\frac{1}{12}, 
\end{align*}
a contradiction. 

Next, suppose that $w_2 = w_5$. Here we will use that there is at least one sneaky edge in the parity class $C_2$. If $w_2w_6$ is an edge then $w_2w_6v_6v_5$ is a $4$-cycle, so $w_2w_6$ cannot be a sneaky edge. By symmetry, $w_2w_4$ is not a sneaky edge either. Therefore $w_4w_6$ must be a sneaky edge. 
At least one of $w_1w_4, w_3w_6$ is not an edge, for otherwise $w_1w_3w_6w_4$ would be a $4$-cycle. Without loss of generality (by symmetry of the graph structure derived so far) $w_1w_4 \notin E(G)$. 

Using these assumptions and the fact that $G$ has no triangle nor $4$-cycle, we see that for $i\in \left\{1,2,4 \right\}$, the vertex $w_i$ has its third neighbor outside the set $\left\{  v_1,\ldots, v_6,w_1,\ldots,w_6\right\}$; we let $x_i$ denote that neighbor.  

Consider a near independence packing $\mathcal{P}=(G_0, G_1, \dots, G_k)$ of $G$ with $G_0 = G[\left\{v_1,v_2,v_3,v_4,v_5,v_6,w_1,w_2,w_3,w_4,w_6,x_1,x_2,x_4\right\}]$. Observe that $|V(G_0)|\leq 14$ (possibly $x_1,x_2,x_4$ are not all distinct). Also, by the discussion above and absence of triangles and $4$-cycles, $S_0:= \left\{ v_3,v_6,w_1,w_2,w_4 \right\}$ is an independent set that has no neighbors in $V(G) - V(G_0)$. Therefore every independent set of $G-Z$ can be extended to an independent set of $G$ by adding $S_0$. This means that $\gamma \geq 5$. Applying Corollary~\ref{cor:removeForcubic}, we find that $G=G_0$ and $|V(G)|=14$. In particular, $x_1,x_2,x_4$ are pairwise distinct.
The subgraph of $G$ that we have identified so far (see Figure~\ref{fig:nosixcycle_case2_a}) has three vertices of degree $1$  ($x_1,x_2$ and $x_4$) and two vertices of degree $2$ ($w_3$ and $w_6$). We need to complete this to the cubic graph $G$. 
If $x_1$ (resp.\ $x_4$) is adjacent to both $x_2$ and $x_4$ (resp.\ both $x_2$ and $x_1$) then every extra edge incident to $x_4$ (resp.\ $x_1$) will create a triangle or $4$-cycle; contradiction. On the other hand, if $x_2$ is adjacent to both $x_1$ and $x_4$, then we need to add the edges $x_1w_6, x_4w_3$ to obey $3$-regularity and triangle-freeness. But then $G$ is the forbidden graph $F_{14}^{(2)}$ (see Figure~\ref{fig:nosixcycle_case2_b}); contradiction. Thus, in $G$ each vertex in $\left\{x_1,x_2,x_4\right\}$ is incident to at most one other vertex in $\left\{x_1,x_2,x_4\right\}$. This means that we cannot complete the subgraph in Figure~\ref{fig:nosixcycle_case2_a} to a cubic graph; contradiction.


\begin{figure}
\centering
\begin{minipage}[c]{.47\textwidth}
\centering   
   \subfloat[]{\label{fig:nosixcycle_case2_a}\includegraphics[width=0.7\textwidth]{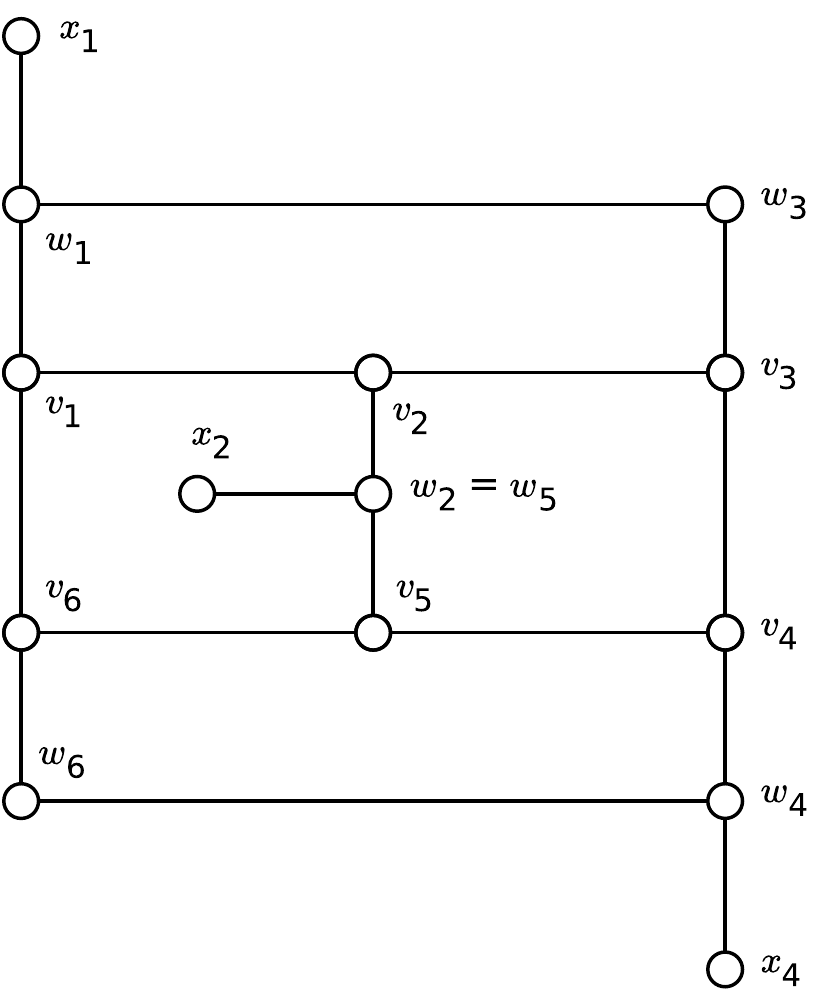}} 
\end{minipage}      
\begin{minipage}[c]{.47\textwidth}
\centering   
   \subfloat[]{\label{fig:nosixcycle_case2_b}\includegraphics[width=0.6\textwidth]{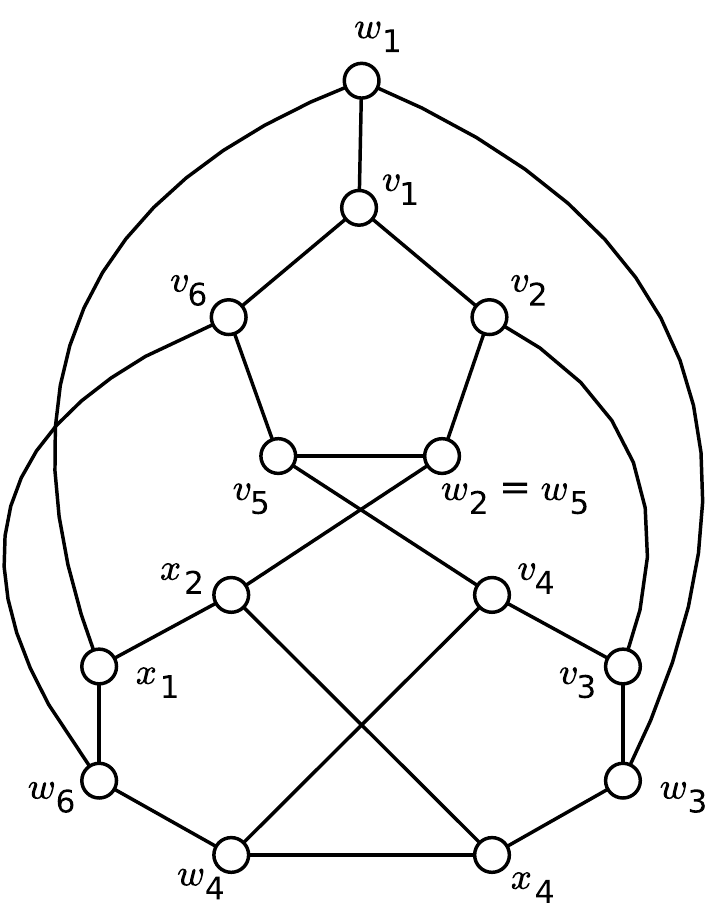}}   
\end{minipage} 
\caption{Illustration of Case $2$ of Lemma~\ref{le:nosixcycle}, when $w_2=w_5$ and $G=G_0$. If additionally $x_1x_2,x_4x_2,x_1w_6,x_4w_3$ are edges, then we obtain the forbidden graph $F_{14}^{(2)}$ (see Figure~\ref{fig:nosixcycle_case2_b}). Otherwise the subgraph depicted in Figure~\ref{fig:nosixcycle_case2_a} cannot be completed to a cubic graph with girth at least $5$ (on the same vertex set). \label{fig:nosixcycle_case2}}
\end{figure}

\textbf{Case $3$: Both parity classes have at least two sneaky edges.}\\
Three sneaky edges in one parity class create a triangle, so we may assume that both parity classes have exactly two sneaky edges. Without loss of generality, let $w_1w_3$ and $w_3w_5$ be the two sneaky edges for $C_1$. 
Consider a near independence packing $\mathcal{P}=(G_0, G_1, \dots, G_k)$ of $G$ with $G_0 = G[\left\{v_1,v_2,v_3,v_4,v_5,v_6,w_1,w_2,w_3,w_4,w_5,w_6, x_1, x_5\right\}]$, where $x_1 \notin \left\{ v_1,w_3 \right\}$ and $x_5 \notin \left\{v_5,w_3  \right\}$ are the third neighbors of respectively $w_1$ and $w_5$. Since $w_1w_5$ is not an edge, it follows that $S_0:= \left\{v_2,v_4,v_6,w_1,w_5 \right\}$ is an independent set. Note that $S_0$ has no neighbors in $V(G)-V(G_0)$. Thus, every independent set of $G-V(G_0)$ can be extended to an independent set of $G$ by adding $S_0$. Hence, $\gamma\geq 5$. Since $|V(G_0)|\leq 14$, Corollary~\ref{cor:removeForcubic} yields that $G=G_0$, $|V(G)|=14$ and $\alpha(G)=5$.

Since the parity class $C_2$ also has two sneaky edges, exactly two of the pairs $\left\{w_2,w_4\right\},\left\{w_4,w_6\right\}, \left\{w_2,w_6\right\}$ induce an edge. In particular, at least one of $w_2w_6,w_4w_6$ is an edge of $G$, and by symmetry of the structure identified so far, we may assume that $w_4w_6 \in E(G)$. 

\begin{figure}
\captionsetup[subfigure]{labelformat=empty}
\centering

   \subfloat[$F_{14}^{(1)}$]{\label{fig:nosixcycle_case3_F14_1}\includegraphics[width=0.45\textwidth]{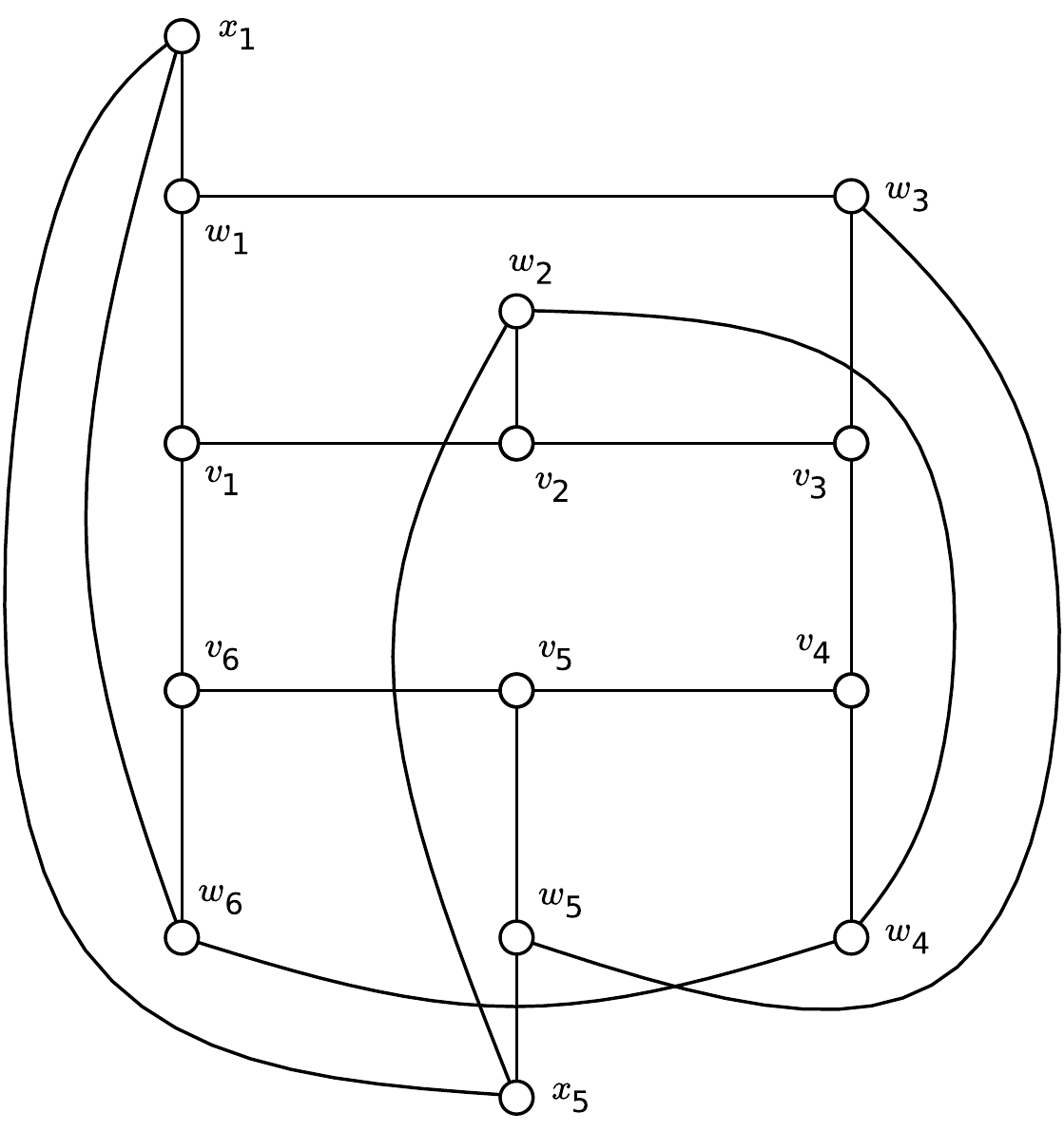}} \qquad
   \subfloat[$F_{14}^{(2)}$]{\label{fig:nosixcycle_case3_F14_2}\includegraphics[width=0.45\textwidth]{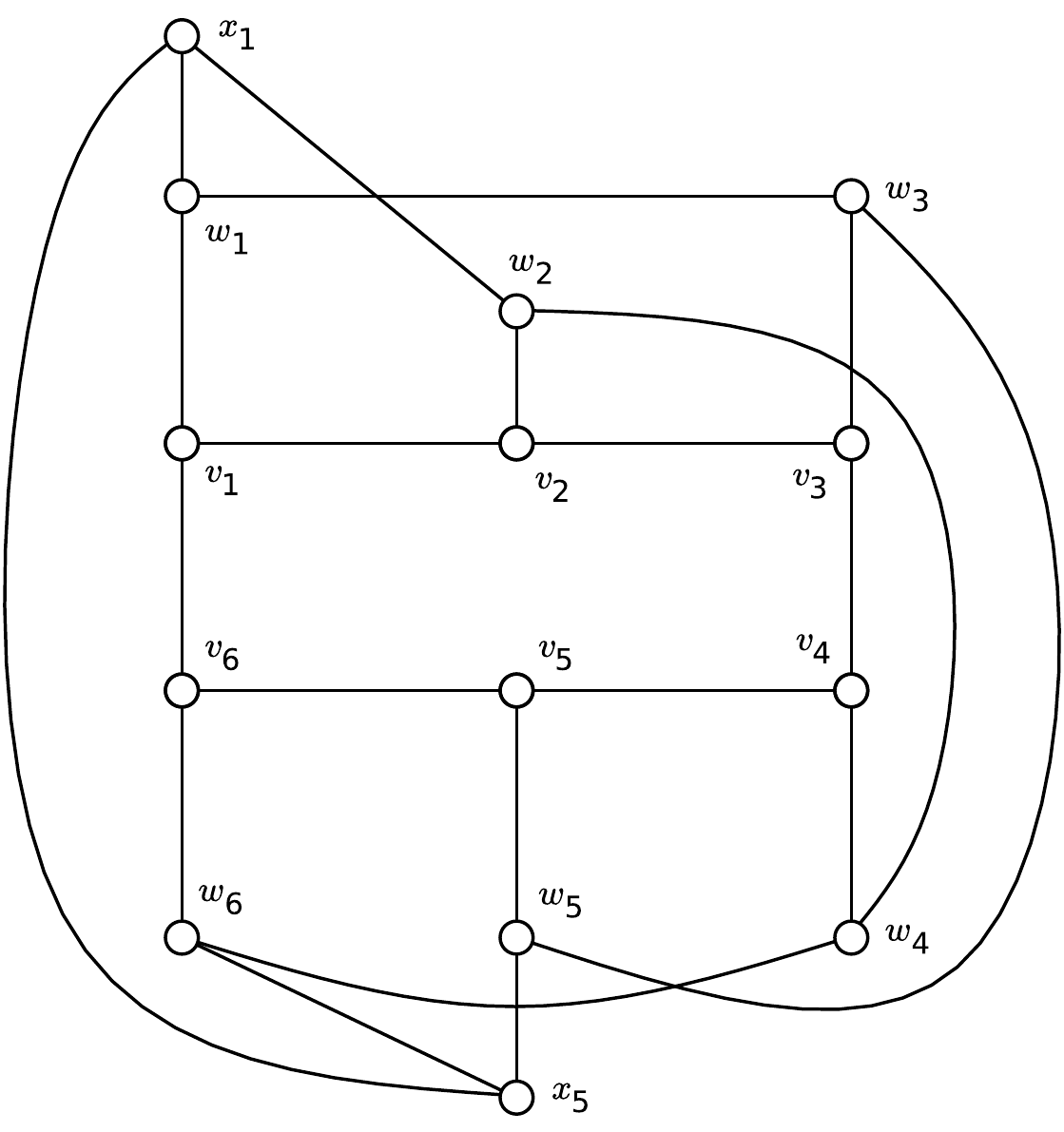}}

\caption{Illustration of the end of Case $3$ of Lemma~\ref{le:nosixcycle}, when both parity classes have two sneaky edges and $G=G_0$ is a graph on $14$ vertices and $\alpha(G)=5$. 
The two graphs are isomorphic to the forbidden graphs $F_{14}^{(1)}$ and $F_{14}^{(2)}$, respectively. \label{fig:nosixcycle_case3}}
\end{figure}

See Figure~\ref{fig:nosixcycle_case3} for an illustration of the following discussion.  
If $w_2w_6 \in E(G)$, then $\left\{w_1,w_2,w_4,w_5,v_3,v_6 \right\}$ is an independent set, contradicting that $\alpha(G)=5$. 
Thus $w_2w_6 \notin E(G)$, and hence $w_2w_4\in E(G)$. In that case, we must have $x_1x_5\in E(G)$, so that $G$ is either the forbidden graph $F_{14}^{(1)}$ (if $x_1w_6$ and $x_5w_2$ are edges) or the forbidden graph $F_{14}^{(2)}$ (if $x_1w_2$ and $x_5w_6$ are edges). Contradiction.
\end{proof}

\begin{lemma}\label{le:adjacentC5s}
If $C$ is a $5$-cycle of $G$, then for every edge $e$ of $C$ there is another $5$-cycle $C_e$ such that $E(C)\cap E(C_e)= \left\{e\right\}$.
\end{lemma}
\begin{proof}
Let $C:=v_1v_2v_3v_4v_5$ be a $5$-cycle.  For each $i\in\{1,\dots,5\}$, let $w_i \notin V(C)$ denote the third neighbor of $v_i$. Observe that $v_1,\ldots, v_5, w_1,\ldots,w_5$ must be pairwise distinct and $w_iw_{i+1} \notin E(G)$ for each $i$ (where indices are taken cyclically), since $G$ has no triangle, and no $4$-cycle or $6$-cycle either by Lemmas~\ref{le:nofourcycle} and~\ref{le:nosixcycle}. 

It suffices to prove the lemma for the edge $e=v_1v_2$. Arguing by contradiction, we assume from now on that $w_1$ and $w_2$ do not have a common neighbor. Let $G_0:= G[\left\{v_1,v_2,v_3,v_4,v_5,w_4 \right\}]$ and let $H$ denote the graph obtained from $G-V(G_0)$ by adding the edge $w_1w_2$. Note that $H$ is subcubic and triangle-free, by our assumption on $w_1$ and $w_2$.

Let $S$ be a maximum independent set of $H$. At least one of $w_1,w_2$, say $w_1$, is not in $S$. Then $S\cup \left\{ v_1,v_4 \right\}$ is an independent set of $G$. Thus $\alpha(G)\geq \alpha(H) + 2$. Also, $\lb(G_0)=\frac{1\cdot 9+ 4 \cdot 10 + 1 \cdot 11-2}{24}=2+ \frac{5}{12}$. Hence, 
\begin{equation}\label{eq:adjacentC5s_1}
\alpha(G) \geq \alpha(H) + 2 = \alpha(H) + \lb(G_0) - \frac{5}{12}.
\end{equation}

Let $G_1,G_2,\ldots,G_k$ be an independence packing of $H$. 
If the edge $w_1w_2$ is included in one of the graphs $G_1, \dots, G_k$, let us assume without loss of generality that it is in $G_k$. 

First, let us deal quickly with the case that $w_1w_2$ is {\em not} in $G_k$. 
Then $\mathcal{P}=(G_0, G_1, \dots, G_k)$ is a near independence packing of $G$. 
Since $\alpha(H)= \sum_{i=1}^k \alpha(G_i)$ and $\alpha(G) - \alpha(H) \geq 2$, we deduce that $\gamma \geq 2$. 
Using Lemma~\ref{le:contract} combined with $m_0=6$, we then obtain 
\[
\alpha(G) 
\geq \lb(G) + \gamma - \lb(G_0) + \left\lceil \frac{m_0 + 3\delta + 3\nu + 4\sigma}{2} \right\rceil  \cdot \tfrac{1}{12}
\geq \lb(G) - \tfrac{1}{12}, 
\]
implying that $G$ is not a counterexample. 
Thus, the edge $w_1w_2$ must be in $G_k$. 

Let $m:=|E(G)|- \sum_{i=0}^k |E(G_i)|$. 
Let $m_i$ ($i \in \{0, 1, \dots, k\}$) denote the numbers of edges of $G$ that have exactly one endpoint in $V(G_i)$. 

Since $m_0 = 6$, it follows that $m\geq 6-1=5$ (the $-1$ is due to the edge $w_1w_2$ which is in $G_k$ but not in $G$).  
Also, by $3$-connectivity of $G$ we know that $m_i \geq 3$ for each $i\in \{1, \dots, k\}$. 

Let us consider $G_k$. 
If $G_k$ is isomorphic to one of $F_{11}$, $F_{14}^{(1)}$, $F_{14}^{(2)}$, $F_{19}^{(1)}$, $F_{19}^{(2)}$, $F_{22}$, then $G_k$ has at most one degree-$2$ vertex, and in fact it must have one because otherwise $\{w_1,w_2\}$ would be a cutset of $G$, so $G_k$ is isomorphic to one of $F_{11}$, $F_{19}^{(1)}$, $F_{19}^{(2)}$. 
In particular, $m_k=3$ in this case. 
If $G_k$ is isomorphic to a bad graph then $G_k$ has four vertices of degree $2$. Since $w_1w_2 \in E(G_k)$, this implies that either $m_k=6$, or $m_k=4$ and there is one edge of $G$ with both endpoints in $V(G_k)$ but not in $G_k$. 

Similarly but not quite exactly as when considering independence packings, let us write $\delta, \sigma, \nu$ for the number of indices $i\in \{1, \dots, k\}$ such that $G_i$ is respectively dangerous, isomorphic to $K_1$ or $K_2$, and a graph on at least three vertices that is neither dangerous nor bad nor isomorphic to $F_{11}$, $F_{19}^{(1)}$ or $F_{19}^{(2)}$. 
Let also $I_F$ ($I_B$) be the indicator variable which is equal to $1$ if $G_k$ is isomorphic to one of $F_{11}$, $F_{19}^{(1)}$, $F_{19}^{(2)}$ (respectively, to a bad graph), and $0$ otherwise. 

Combining the previous observations, we obtain 
\begin{equation}\label{eq:adjacentC5s_2}
2m \geq 6-2 + 3 I_F + 6 I_B + 5 \delta + 3(\sigma + \nu).  
\end{equation}
The $-2$ accounts for the edge $w_1w_2$ in $G_k$ but not in $G$, and the $5\delta$ comes from the fact that every dangerous graph has five degree-$2$ vertices.

Applying Theorem~\ref{th:main} to each graph $G_i$ with $i\in \{1, \dots, k-1\}$, we obtain that 
\[
\alpha(G_i) \geq \left\{ 
\begin{array}{ll}
\lb(G_i) & \textrm{ if } G_i \textrm{ is dangerous }\\[.5ex]
\lb(G_i) + \frac{1}{6} & \textrm{ if } G_i \textrm{ is isomorphic to } K_1 \textrm{ or } K_2 \\[.5ex]
\lb(G_i) + \frac{1}{12} & \textrm{ otherwise. }
\end{array}
\right.
\]
In the third case, the lower bound comes from the fact that $G_i$ has at least three degree-$2$ vertices, contains no bad subgraph, and is not dangerous. 

As for the graph $G_k$, since $G_k$ is not a subgraph of $G$, we have less control on its structure. Applying Theorem~\ref{th:main} gives the following (weaker) lower bounds:
\[
\alpha(G_k) \geq \left\{ 
\begin{array}{ll}
\lb(G_k) - \frac{1}{12} & \textrm{ if } G_k \textrm{ is bad or isomorphic to one of } F_{11}, F_{19}^{(1)}, F_{19}^{(2)}  \\[.5ex]
\lb(G_k) + \frac{1}{6} & \textrm{ if } G_k \textrm{ is isomorphic to } K_1 \textrm{ or } K_2 \\[.5ex]
\lb(G_k) & \textrm{ otherwise. }
\end{array}
\right.
\]
Note that $G_k$ cannot in fact be isomorphic to $K_1$ since it contains the edge $w_1w_2$, but we will not use this fact. 
Let $I_N$ be the indicator variable which is equal to $1$ if $G_k$ falls in the third category and $G_k$ is not dangerous, and $0$ otherwise.  
Note that in the case $I_N=1$ we do not necessarily obtain a stronger lower bound of $\lb(G_k) + \frac{1}{12}$ for $G_k$ as $G_k$ could have a bad subgraph or less than three degree-$2$ vertices. 

Using \eqref{eq:adjacentC5s_1}, we obtain 
\begin{equation}\label{eq:adjacentC5s_4}
\alpha(G) \geq \lb(G_0) - \frac{5}{12} + \alpha(H)
= \lb(G_0) - \frac{5}{12} + \sum_{i=1}^{k} \alpha(G_i)
= \lb(G) + \frac{2(m-k)-10}{24} + \sum_{i=1}^{k} (\alpha(G_i)-\lb(G_i)).
\end{equation}

Combining the previous inequalities, we obtain

\begin{equation}\label{eq:adjacentC5s_5}
\alpha(G) \geq \lb(G) + \frac{1}{24} \left( -6 - I_F + 2 I_B + 3\delta + 5\sigma + 3\nu - 2I_N \right).
\end{equation}

If instead of using inequality (\ref{eq:adjacentC5s_2}) we use that $2m\geq 10$, then we find
\begin{equation}\label{eq:adjacentC5s_6}
\alpha(G) \geq \lb(G) + \frac{1}{24}  \left(-2\delta +2\sigma -4I_F -4I_B - 2I_N \right).
\end{equation}

From (\ref{eq:adjacentC5s_5}) and (\ref{eq:adjacentC5s_6}) we see that $\alpha(G)-\lb(G)\geq -\tfrac{1}{12}$, unless $\sigma=0$, $\delta + \nu\leq 1$, and $I_F+I_B=1$. Thus $k \in \{1,2\}$. 

Suppose first that $G_k$ is bad. By inequality (\ref{eq:adjacentC5s_5}), we may assume that $k=1$. 
If $G_k$ is not isomorphic to $B_8$ or $B_{16}^{(1)}$, then by Lemma~\ref{le:badproperties}.\ref{le:badproperties_6}, no edge of $G_k$ is contained in every $6$-cycle of $G_k$. 
Thus $G_k - w_1w_2$ contains a $6$-cycle, and so does $G$; contradiction. 
If $G_k$ is isomorphic to $B_8$ then $w_1w_2$ must be an edge of the $4$-cycle of $B_8$. But then in $G$ there is a $6$-cycle consisting of the path $w_1v_1v_2w_2$ and two other vertices of $G_k$. 
Contradiction. 
Finally, if $G_k$ is isomorphic to $B_{16}^{(1)}$, then observe that there is only one edge of $B_{16}^{(1)}$ contained in all $6$-cycles of $B_{16}^{(1)}$, and thus $w_1w_2$ must be that edge. 
There are eight ways to add the remaining edges in $G$ between $G_0$ and $G_k-w_1w_2$ while respecting $3$-regularity and triangle-freeness, see Figure~\ref{fig:obtainingF22}. This reveals that $G$ is either the forbidden graph $F_{22}$ (contradiction) or $G$ contains a $6$-cycle (contradiction). 

\begin{figure}
\centering
\includegraphics[width=0.5\textwidth]{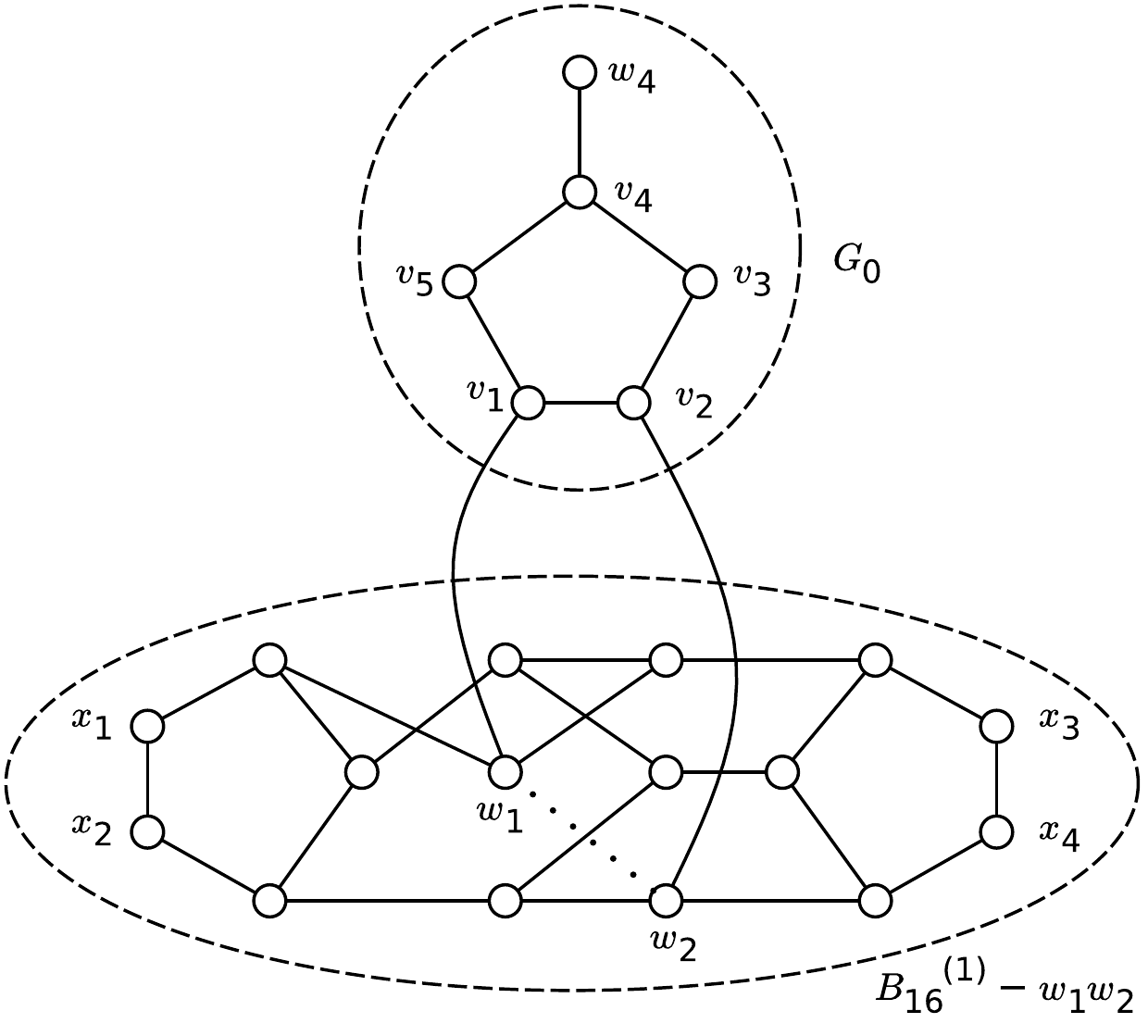}
\caption{Reduction to $F_{22}$, in Lemma~\ref{le:adjacentC5s}. There are $2^3$ ways of adding extra edges between $G_0$ and $G_k-w_1w_2$, such that the resulting graph is triangle-free and cubic. Indeed: $w_4$ must have one neighbor in $\left\{x_1,x_2\right\}$ and one neighbor in $\left\{x_3,x_4\right\}$, and $v_5$ and $v_3$ cannot be adjacent. One of the eight resulting graphs is $F_{22}$, obtained by adding the edges $w_4x_2, w_4x_3, v_5x_1$ and $v_3x_4$. The seven other graphs all contain $6$-cycles. 
} \label{fig:obtainingF22}
\end{figure}

It remains to consider the case where $G_k$ is isomorphic to one of $F_{11}, F_{19}^{(1)}, F_{19}^{(2)}$. 
Note that $k=1$ is not possible here, since $m_0=6$ but $m_1=3$. 
Thus $k=2$. 
Recall that $G_1$ is then either dangerous, or not dangerous with at least three degree-$2$ vertices. 
Let $y$ denote the unique degree-$2$ vertex of $G_k$. 
If $yv_3$ is an edge of $G$ then $\left\{v_3v_4,v_1v_5 \right\}$ is a $2$-edge cutset of $G$; contradiction. 
Thus $yv_3 \notin E(G)$. 
Symmetrically we know that $yv_5 \notin E(G)$. 
Below we distinguish two cases, namely $y w_4 \in E(G)$ and $y w_4\notin E(G)$.

\textbf{Case $1$: $yw_4 \notin E(G)$.}\\
In this case $\left\{y,v_3,v_5,w_4 \right\}$ is an independent set. Therefore there must be precisely \emph{five} edges in $G$ that have one endpoint in $V(G_1)$ (and the other endpoint in $\left\{y,v_3,v_5,w_4\right\}$). If $G_1$ is not dangerous then this implies that in inequality (\ref{eq:adjacentC5s_2}), we can replace the factor $3 \nu$ with $5 \nu$. This propagates to an extra additive term $\tfrac{1}{12}$ in inequality (\ref{eq:adjacentC5s_5}), yielding $\alpha(G)\geq \lb(G) -\frac{1}{12}$; contradiction. So we may assume that $G_1$ is dangerous. 
Let $x_4,x_4^*$ denote the neighbors of $w_4$ in $V(G_1)$ and let $y^*$ denote the neighbor of $y$ in $V(G_1)$. Observe that $\left\{w_3,w_5,x_4,x_4^*,y^*\right\}$ are the degree-$2$ vertices of $G_1$.

Because $x_4x_4^*$ is not an edge (otherwise $w_4x_4x_4^*$ is a triangle), it follows from Lemma~\ref{le:dangerousproperties}.\ref{le:dangerousproperties_T} that $G_1$ has an independent set $S_1$ of size $\lb(G_1)$ that avoids $\left\{w_3,w_5,y^*\right\}$. 
Let $S_2$ be a maximum independent set of $G_2-w_1w_2$. Since $G_2$ is critical, we know that $|S_2| = \alpha(G_2) + 1 \geq \lb(G_2) - \frac{1}{12} + 1$. 
Now, let 
\[
S := \left\{v_3,v_5\right\} \cup S_1 \cup S_2. 
\]
Note that $S$ is an independent set of $G$. 
Since $\lb(G) = \sum_{i=0}^2 \lb(G_i) - \frac{4}{12}$, and recalling that $\lb(G_0) = 2 + \frac{5}{12}$, we obtain that
\[
\alpha(G) \geq |S| 
\geq 2 + \lb(G_1) + \left(\lb(G_2) - \frac{1}{12} + 1\right)
= \sum_{i=0}^2 \lb(G_i) + \frac{6}{12}
= \lb(G) + \frac{10}{12},
\]
a contradiction.

\textbf{Case $2$: $yw_4 \in E(G)$.}\\
First note that $G-\left\{v_1,v_2,v_3,v_4,v_5,w_4\right\}$ is disconnected. Observe furthermore that $y\neq w_1$, for otherwise $\left\{w_1,w_2\right\}$ would be a $2$-cutset. Due to the edge $yw_4$ and the facts that $y \neq w_1$ and $G_2-w_1$ is connected, it follows that $G-\left\{v_1,v_2,v_3,v_4,v_5,w_1\right\}$ is connected. 

We started out the proof of this lemma from the perspective of an arbitrary edge $v_1v_2$ of the $5$-cycle and, arguing by contradiction, we assumed that $w_1$ and $w_2$ do not have a common neighbor.  
In case $w_4$ and $w_5$ do not have a common neighbor either, we can apply the same arguments instead to the edge $v_4v_5$, and deduce that the derived local graph structure from the perspectives of $v_1v_2$ respectively $v_4v_5$ are the same. That is, we end up in this last case as well when considering the edge $v_4v_5$. 
However, this implies that $G-\left\{v_1,v_2,v_3,v_4,v_5,w_1\right\}$ must be disconnected, a contradiction. 
Hence, $w_4$ and $w_5$ must have a common neighbor $x_4\neq y$. Symmetrically $w_4$ and $w_3$ must have $x_4$ as their common neighbor. But then $x_4w_3v_3v_4v_5w_5$ is a $6$-cycle; contradiction.
\end{proof}

\begin{lemma}\label{le:noC5s}
$G$ has no $5$-cycle.
\end{lemma}
\begin{proof}
Suppose $G$ has a $5$-cycle. 
By repeatedly applying Lemma~\ref{le:adjacentC5s}, we see that each edge of $G$ is incident to a $5$-cycle and (therefore) in fact contained in some $5$-cycle. Let $C:=v_1v_2v_3v_4v_5$ be a $5$-cycle. By Lemma~\ref{le:adjacentC5s}, there is a second $5$-cycle $v_1v_2w_2x_1w_1$ containing $v_1v_2$, with $w_1,w_2,x_1 \notin V(C)$. There cannot be a third $5$-cycle containing $v_1v_2$, for otherwise $G$ must have a $4$-cycle or $6$-cycle. We conclude that $v_1v_2$ (and therefore each edge of $G$) is contained in exactly two $5$-cycles.
As proved by Bondy and Locke~\cite[Theorem~8.2]{BL86}, the only connected triangle-free cubic graph having no $6$-cycle and admitting an edge-$2$-covering by $5$-cycles is the dodecahedron. The dodecahedron is not a counterexample to Theorem~\ref{th:main}. 
\end{proof}

In summary, we know that $G$ is $3$-connected and has girth at least $7$. In particular, by Lemma~\ref{le:dangerousproperties}.\ref{le:dangerousproperties_5cycle} and absence of $5$-cycles, no subgraph of $G$ is dangerous. Recall that $G$ also has no bad or forbidden subgraph.

We are now in a position to finish the proof of Theorem~\ref{th:main}.  
Because $G$ is cubic, $G$ contains an even cycle. Let $C:=v_1v_2\ldots v_{2t}$ be a shortest even cycle. 
We will derive a contradiction by constructing a shorter even cycle from $C$, using the properties of our minimum counterexample $G$. 

Note that the cycle $C$ has at most one chord, for otherwise there is a shorter even cycle. 
Let $w_i \notin V(C)$ denote the third neighbor of $v_i$, for $1\leq i \leq 2t$. If $w_i=w_j$ for two distinct $i,j$ then we call $w_i$ a \emph{sneaky vertex}. Observe that $i-j$ must be odd. Furthermore, there is at most one sneaky vertex, for otherwise there is a shorter even cycle (here we use that $G$ has girth at least $7$). 

Similarly, an edge of the form $w_iw_{j}$ is called a \emph{sneaky edge}. Observe that there are no edges of the form $w_iw_{i+2}$ (otherwise there is a $5$-cycle) or $w_iw_{i+2l+1}$ (otherwise there is a shorter even cycle). Therefore all sneaky edges are of the form $w_iw_{i+2l}$, for some integers $i$ and $l\geq 2$. In particular, this means that every sneaky edge is a `shortcut' for $C$, yielding an \emph{odd} cycle that is shorter than $C$. Later we will use the chord, the sneaky vertex, and sneaky edges, if any exists, to build contradictory shorter \emph{even} cycles.

We define two \emph{parity classes}, $Z_1:=V(C) \cup \left\{w_1,w_3,\ldots, w_{2t-1}\right\}$ and $Z_0:=V(C) \cup \left\{w_2,w_4,\ldots, w_{2t}\right\}$. If $w_iw_{i+2l}$ is a sneaky edge then we say it belongs to the parity class $Z_{i \text{ mod } 2}$. 
 
Let $V_1:=\left\{v_1,v_3,\ldots, v_{2t-1}\right\}$ and $V_0:=\left\{v_2,v_4,\ldots, v_{2t}\right\}$. If $C$ has no chord then both $V_0$ and $V_1$ are independent sets. If $C$ has a  chord then it must be of the form $v_{g}v_{g+2h}$, for some integers $g,h$ (otherwise there is a shorter even cycle). Without loss of generality we assume that $g$ is even, so that $V_1$ is an independent set. (Note that $V_0$ then is not an independent set.)

We consider a near independence packing $\mathcal{P}=(G_0, G_1, \dots, G_k)$ of $G$ with $G_0 = G[Z_1]$. 
Before pursuing further, let us first remark that in case $C$ has no chord, then the arguments below apply equally if we chose $G_0 = G[Z_0]$ instead, and indeed we will use this freedom of switching to $G_0 = G[Z_0]$ in Case~$2$ below. 

A maximum independent set of $G-V(G_0) = G-Z_1$ can be extended to an independent set of $G$ by adding $V_{1}$. 
Using the notations associated to our near independence packing $\mathcal{P}$, this shows that $\gamma \geq |V_1|=t$. 
The vertices in $Z_1$ are pairwise distinct (otherwise there is a shorter even cycle), so $|Z_1|=3t$. 
By Lemma~\ref{le:removeForcubic}, we obtain (using that $\delta=0$ because $G$ has no dangerous subgraph):
\begin{equation}\label{eq:shortestevencycle_1}
\alpha(G) \geq \lb(G) + t + \frac{2 -9 \cdot 3t}{24} + \frac{m_0}{24} = \lb(G) + \frac{2+m_0-3t}{24}.
\end{equation}
To obtain a contradiction, it thus suffices to show that $m_0 \geq 3t-4$.

If there is no chord, no sneaky vertex, and no sneaky edge, then each vertex in $\left\{ w_1, w_3,\ldots, w_{2t-1} \right\}$ has two neighbors in $V(G)- Z_1$ and each vertex in $V_0$ has one neighbor in $V(G)- Z_1$. The remaining vertices of $Z_1$ have no neighbors in $V(G)- Z_1$. This then immediately implies that $m_0=3t$.

Each chord, sneaky vertex or sneaky edge belonging to $Z_1$ decreases our estimate $3t$ on $m_0$ by (at most) $2$. It thus suffices to show that 
\begin{equation}\label{eq:sneakystuff}
 \mathbbm{1}_{\left\{\text{there is a chord} \right\}} + \mathbbm{1}_{\left\{\text{there is a sneaky vertex} \right\}} +  \#\left\{\text{sneaky edges belonging to } Z_1 \right\} \leq 2.
 \end{equation}

\textbf{Case $1$. $C$ has a chord.}\\
It suffices to show that there is no sneaky edge belonging to $Z_1$.  Because $G$ has girth at least $7$, we must have $t\geq 6$. Without loss of generality, we may assume that $v_2v_{2+2h}$ is the chord, for some integer $h$.  The chord `separates' $C$ into two shorter odd cycles, $C_1=v_2v_{2+2h}v_{2+2h+1}\ldots v_{2t}v_1$ and $C_2=v_2v_{2+2h}v_{2+2h-1}\ldots v_3$. 

Suppose that there is a sneaky edge $w_iw_j$ belonging to $Z_1$. If $v_i$ and $v_j$ are both in $C_1$ or both in $C_2$ (say both in $C_1$) then the path $v_iw_iw_jv_j$ is a shortcut for $C_1$, yielding a shorter even cycle. If $v_i$ is in $C_1$ and $v_j$ is in $C_2$, then either $v_2v_3\ldots v_iw_iw_jv_jv_{j-1} \ldots v_{2+2h}$ or $v_2v_1 \ldots v_j w_jw_iv_i v_{i+1}\ldots v_{2+2h}$ is a shorter even cycle. Therefore there is no sneaky edge belonging to $Z_1$.

\textbf{Case $2$. $C$ has no chord.}\\

Suppose $w_a w_b$ and $w_c w_d$ are two sneaky edges, not necessarily belonging to the same parity class. Recall that this implies that $a=b \text{ mod } 2$ and $c=d \text{ mod } 2$ (otherwise there is a shorter even cycle). Using the symmetries of the cycle $C$, we may assume without loss of generality that $1=a < b < d \leq 2t$ and  $a < c < d$. We will now argue that then $a<c<b<d$.  If this is not the case, then either $a < b= c < d$ or $a<b<c<d$.

First suppose that $a < b =c < d$. In other words, $w_aw_b$ and $w_cw_d$ share the vertex $w_b=w_c$. We have $b-a\geq 4$ and $d-c \geq 4$ (since $G$ has girth at least $7$), so that $d\geq 9$. But then $C_1:=v_1w_1w_bw_dv_d v_{d+1}\ldots v_{2t}$ is an even cycle shorter than $C=v_1v_2 \ldots v_d v_{d+1}\ldots v_{2t}$. Indeed, $|C|-|C_1|= | \left\{ v_2,v_3,\ldots, v_{d-1} \right\}| - | \left\{w_a,w_b,w_d  \right\}|= d-5 \geq 4$, and $d-5$ is even because $1=a=b=c=d \text{ mod } 2$.
Next, suppose that $a<b<c<d$. Then $C_2:=v_a w_a w_b v_b v_{b+1} \ldots v_c w_c w_d v_{d} \ldots v_{2t}$ is an even cycle shorter than $C$. Indeed, $|C|-|C_2|=|\left\{v_{a+1},v_{a+2}, \ldots, v_{b-1} \right\}| +  |\left\{v_{c+1},v_{c+2},\ldots, v_{d-1} \right\}| - |\left\{ w_a,w_b,w_c,w_d \right\}| = (b-a-1) + (d-c-1) - 4 \geq 3 + 3- 4=2$.
We have now proved that $a<c<b<d$; in other words, every two sneaky edges \emph{cross}.

We are done if at most one sneaky edge belongs to $Z_1$, so we may assume that $Z_1$ has at least two sneaky edges. Likewise, at least two sneaky edges belong to $Z_0$. Thus in total there are at least four sneaky edges $w_aw_b, w_cw_d, w_ew_f,w_g w_h$. 
By the crossing property we have derived above, we may assume without loss of generality that $1=a <c < e< g< b< d < f < h \leq 2t$. In particular $b$ must be odd, since $a$ is. 

Consider the cycles $C_3:=v_aw_aw_bv_bv_{b-1} \ldots v_{g} w_g w_h v_h v_{h+1} \ldots v_{2t}$ and $C_4:=v_av_{a+1} \dots v_c w_c w_d v_d v_{d-1} \ldots v_b w_b w_a$.
 Then $|C|-|C_3|=|\left\{v_{a+1},v_{a+2},\ldots, v_{g-1} \right\}| +  |\left\{v_{b+1},v_{b+2},\ldots, v_{h-1} \right\} | - |\left\{ w_a,w_b,w_g,w_h \right\} | = (g-a-1) + (h-b-1) - 4  = g+h-b-7$, while
$|C|-|C_4|= |\left\{v_{d+1}, \ldots, v_{2t}  \right\}| + |\left\{v_{c+1}, \ldots, v_{b-1}  \right\}|  -  |\left\{ w_a,w_b,w_c,w_d \right\} | = b-c-d+2t-5= b-g-h+2t-5+(h-d)+(g-c)\geq b-g-h+2t-5+ 2+ 2=  b-g-h+2t-1$.

Note that $|C|-|C_3|$ and $|C|-|C_4|$ are even because $b$ is odd and $g=h \text{ mod } 2$  and $c=d \text{ mod } 2$. Therefore $C_3$ and $C_4$ are even cycles.
Furthermore, the minimum of $|C|-|C_3|$ and $|C|-|C_4|$ is at least $-7+ \min(g+h-b,  2t+6 -( g+h-b)) \leq -7 +(t+3)$. This implies that there is a shorter even cycle unless $t\leq 4$. Because $G$ has girth at least $7$, we know that $t\geq 4$. Thus $t=4$ (i.e.\ $C$ is an $8$-cycle). 
Then, there is no sneaky vertex, because this would create a triangle, a $5$-cycle or a shorter even cycle. 
Since $Z_1$ clearly cannot contain more than two sneaky edges, inequality~\eqref{eq:sneakystuff} holds, as desired.

This concludes the proof of Theorem~\ref{th:main}. 
\qed

We end this section with a discussion of some aspects of the proof. 
These remarks can be freely skipped.

\subsubsection*{Forbidden graphs and how they emerge in the proof}

Here we briefly recap where the forbidden graphs arose as obstructions in the proof. 
The three noncubic graphs appear as exceptional cases in the proofs of Lemma~\ref{le:twoedgecutset} (about $2$-edge cutsets) and Lemma~\ref{le:nobadsubgraph} (no bad subgraph). This makes sense because all noncubic forbidden graphs have a bad subgraph (namely $B_8$ or $B_{16}^{(1)}$) as well as a $2$-edge cutset whose removal leaves two components, one of which is isomorphic to $K_1$ (cf.\ Figure~\ref{fig:forbidden_graphs}).
 
The three cubic forbidden graphs arise much later, when we already know that our minimum counterexample is cubic. More specifically, $F_{14}^{(1)}$ and $F_{14}^{(2)}$ occur in the proof of Lemma~\ref{le:nosixcycle} (no $6$-cycle) and $F_{22}$ emerges in the proof of Lemma~\ref{le:adjacentC5s}. This makes sense because $F_{14}^{(1)}$ and $F_{14}^{(2)}$ contain a $6$-cycle, whereas $F_{22}$ does not. As for $F_{22}$: In the proof of Lemma~\ref{le:adjacentC5s}, we remove a $5$-cycle and one of its neighbors from $G$, and then add a specific edge, yielding an auxiliary graph $H$ to which we apply induction. The graph $F_{22}$ then emerges in the special case that $H$ is isomorphic to $B_{16}^{(1)}$ (cf.\ Figure~\ref{fig:forbidden_graphs}).

\subsubsection*{Why is Lemma~\ref{le:nobadsubgraph} needed?}

Lemma~\ref{le:nobadsubgraph} shows that the minimum counterexample $G$ has no bad subgraph. 
Since its proof takes a few pages, it is natural to wonder whether we could simply avoid introducing this lemma.  
At first sight, such a strategy could be borrowed from Heckman and Thomas~\cite{HT06}. 
They first prove a technical theorem similar to our Theorem~\ref{th:main} (in the planar case) but only for graphs that do not have a bad subgraph. As a result they (almost) only need to deal with bad subgraphs when they derive $\alpha(G)\geq 3n/8$ from their technical theorem, only requiring a short and elegant induction argument there.

We have come to the conclusion that we cannot use such a strategy in our case. It \emph{is} possible to obtain a relatively short derivation of Theorem~\ref{th:goal} from such a hypothetical weakening of Theorem~\ref{th:main} (only stated for graphs that do not have a bad subgraph). 
However, the proof of such a weaker statement is the real problem: It appears we really need the stronger `induction hypothesis' in Theorem~\ref{th:main}. For example, we need this in Lemma~\ref{le:adjacentC5s}, where we remove some vertices from $G$ and add an edge, creating a new graph $H$. In the independence packing of that new graph $H$, it could be that the component $G_k$ containing the new edge has a bad subgraph. In such a situation $G_k$ would not satisfy the conditions of the theorem, so we would not know anything about how $\alpha(G_k)$ and $\lb(G_k)$ compare. 
When faced with a similar situation, Heckman and Thomas~\cite{HT06} were able to dodge such problems by considering a planar embedding of $G$, and deducing that $G_k$ in fact cannot contain any bad subgraph as this would create a $K_5$ minor. In our situation, where we do not know that $G$ is planar, it seems we cannot hope to use such a short argument. 

\section{Independence ratio $11/30$ for connected triangle-free subcubic graphs}\label{se:fraughnaughlocke}
In this section we prove Theorem~\ref{th:trianglefreeconnected_extended} (see below), which is a more precise version of Theorem~\ref{th:trianglefreeconnected}. We show that, up to a few explicit exceptions, all connected triangle-free subcubic graphs on $n$ vertices have independence number at least $\frac{11}{30}n$. Before stating the theorem, we will now describe the exceptional families $\mathcal{T}$ and $\mathcal{T}^{-}$.

Let $\mathcal{T}$ denote the family of graphs that can be obtained from an order $\geq 2$ tree whose internal nodes have degree $4$, in the following way: internal nodes (if there are any) are replaced with a copy of $B_8$ and leaves are replaced with a copy of $F_{11}$. Furthermore, we let $\mathcal{T}^{-}$ denote the family of graphs that can be obtained from some graph in $\mathcal{T}$ by deleting (the vertices of) one copy of $F_{11}$. 

For the purpose of proving Theorem~\ref{th:trianglefreeconnected_extended}, it is convenient to provide the following alternative description of $\mathcal{T}$ and $\mathcal{T}^{-}$.
Let $G_{41}$ denote the connected graph on $41$ vertices  that consists of a copy of $B_8$ and three copies of $F_{11}$ that each are joined to the $B_8$ by an edge. Observe that $G_{41}$ and $F_{11}$ share the property of having exactly one vertex of degree $2$, the other vertices being of degree $3$.
If a graph $G$ can be obtained from a graph $G'$ by replacing a subgraph $F_{11}$ with a copy of $G_{41}$ (mapping the degree-$2$ vertex of $F_{11}$ to the degree two vertex of $G_{41}$), then we say that $G$ is a \emph{$30$-augmentation} of $G'$. A graph $G$ is said to be \emph{reducible to} $G'$  if $G=G'$ or $G$ can be obtained from $G'$ by repeated applications of $30$-augmentations. 
Observe that $\mathcal{T}$ equals the family of graphs that are reducible to the graph consisting of two copies of $F_{11}$ joined by an edge, and $\mathcal{T}^{-}$ equals the family of graphs that are reducible to $F_{11}$. 

\begin{theorem}\label{th:trianglefreeconnected_extended}
Let $G$ be a connected triangle-free subcubic graph on $n$ vertices. Then
\[
\alpha(G)= \frac{11}{30}n - \frac{1}{30} \cdot \begin{cases}  4 & \mbox{ if } G \cong F_{14}^{(1)} \mbox{ or } G\cong F_{14}^{(2)} \\
2 & \mbox{ if } G\cong F_{22} \mbox{ or } G \in \mathcal{T}  \\
1 & \mbox{ if } G \in \mathcal{T}^{-}.
\end{cases}
\]
Furthermore, in the remaining cases
\[  \alpha(G) \geq \frac{11}{30}n. \]
\end{theorem}
\begin{proof}
First we show that the number $\kappa(G):=30 \cdot \alpha(G)-11 \cdot |V(G)|$ is invariant under taking $30$-augmentations. Indeed, suppose that $G$ is a $30$-augmentation of $G'$. Because both $F_{11}$ and $G_{41}$ have a maximum independent set avoiding their unique degree-$2$ vertex, we have $\alpha(G)= \alpha(G')+\alpha(G_{41})-\alpha(F_{11})$. Furthermore, $30 \cdot (\alpha(G_{41})- \alpha(F_{11}) )= 30 \cdot 11 = 11 \cdot (|V(G_{41})|- |V(F_{11})|)$. It follows that $\kappa(G)=\kappa(G')$, as desired.

Thus, we may assume from now on that $G$ does not have $G_{41}$ as a subgraph. In particular: to prove the first part of the theorem it suffices to verify it for the graphs $F_{14}^{(1)}, F_{14}^{(2)}, F_{22}, F_{11}$ as well as the graph consisting of two copies of $F_{11}$ joined by an edge, which can be done by straightforward computation. It remains to prove that $\alpha(G)\geq \frac{11}{30}n$ if $G$ is not one of these five graphs (and does not have $G_{41}$ as a subgraph).

Let $\mathcal{P}=(G_1,\ldots, G_k)$ be an independence packing of $G$. Let $I$ denote the set of indices $i \in \left\{1,\ldots, k\right\}$ such that $G_i$ is not a copy of $F_{11}$. Because $G$ is not $F_{11}$, nor two copies of $F_{11}$ joined by an edge, we know that $I$ is nonempty. For each $i \in I$, let $H_i$ denote the graph that is the disjoint union of $G_i$ and those components of $\mathcal{P}$ that are a copy of $F_{11}$ and have an edge to $G_i$ in $G$. Also, let $f_i$ denote the number of those copies of $F_{11}$. Observe that $V(G)$ is the disjoint union $\bigsqcup_{i \in I} V(H_i)$ and that
$\alpha(G)= \sum_{i=1}^{k} \alpha(G_i)=\sum_{i \in I} \alpha(H_i)$, by the definition of independence packing. It therefore suffices to show that $\alpha(H_i)\geq \frac{11}{30}|V(H_i)|$ for each $i \in I$. We will do so now. 

Let $i \in I$. 
If $G_i$ is isomorphic to a forbidden graph, then by the discussion above $G_i$ must be isomorphic to $F_{19}^{(1)}$ or $F_{19}^{(2)}$. Then $H_i$ is either $G_i$ or: $G_i$ and a copy of $F_{11}$ joined by an edge. In the former case $\alpha(H_i)=7= \frac{11}{30}|V(H_i)| + \frac{1}{30}$ and in the latter case $\alpha(H_i) = 4+7 = \frac{11}{30}|V(H_i)|$, as desired.

If $G_i$ is a bad graph then $\alpha(G_i)=\lb(G_i)-\frac{1}{12}=\frac{3}{8}|V(G_i)|$, so 
\begin{eqnarray*}
\alpha(H_i) &=& \alpha(G_i) + f_i \cdot \alpha(F_{11})\\
&=&  \frac{11  \cdot |V(G_i)|}{30}+ \frac{|V(G_i)|}{120} + f_i \cdot \frac{11 \cdot |V(F_{11})|-1}{30}\\
&=& \frac{11 \cdot |V(H_i)|}{30}+ \frac{|V(G_i)|-4 f_i}{120}.
\end{eqnarray*}
Since $G_i$ has only four vertices of degree $2$, we must have $f_i\leq 4$. If $G_i$ is not isomorphic to $B_8$ then $|V(G_i)| \geq 16$.  If $G_i \cong B_8$ then $f_i \leq 2$ (for otherwise $G_{41}$ would be a subgraph of $G$). It follows that always $|V(G_i)|-4f_i\geq 0$, so that $\alpha(H_i) \geq \frac{11}{30}|V(H_i)|$.

Finally, we discuss the case that $G_i$ is neither a forbidden nor a bad graph. By Theorem~\ref{th:main}, we know that $\alpha(G_i) \geq \lb(G_i)$.  Therefore
\begin{eqnarray*}
\alpha(H_i) &=& \alpha(G_i)+ f_i \cdot \alpha(F_{11})\\
&\geq&  \frac{9 \cdot |V(G_i)| + f_i - 2 }{24} + f_i \cdot \frac{11 \cdot |V(F_{11})|-1}{30}\\
&=&  \frac{9 \cdot |V(G_i)| + f_i - 2}{24} + 4 f_i.
\end{eqnarray*}
On the other hand, we have $\frac{11}{30} \cdot |V(H_i)| = \frac{11}{30} \cdot \left( |V(G_i)| + 11 f_i \right) =  \frac{ 11 |V(G_i)| +  f_i }{30} + 4  f_i$, so 

\[ \alpha(H_i) \geq  \frac{11}{30} \cdot |V(H_i)|  + \left\lceil\frac{9 \cdot |V(G_i)| + f_i - 2}{24} \right\rceil - \frac{11 \cdot |V(G_i)| + f_i}{30}. \]

If $G_i$ is not isomorphic to $K_1$ or $K_2$ then $G_i$ has minimum degree at least $2$, and hence $f_i \leq |V(G_i)|$. 
It thus suffices to verify that $\left\lceil \frac{9g+f-2}{24} \right\rceil - \frac{11g+f}{30} \geq 0$ for all integers $f,g $ such that either  $0 \leq f \leq g$  or $1 \leq g \leq 2 \leq f \leq 4$. 
For $f+g\leq 9$ there is only a small number of integer pairs to be checked, while for $f+g \geq 10$ it directly follows from the fact that $\frac{9g+f-2}{24} = \frac{11g+f}{30} + \frac{f+g-10}{120}$.


\end{proof}

\section{If triangles are allowed}
\label{sec:triangles}

In this section we generalize our main theorem by relaxing the triangle-freeness condition, in the spirit of~\cite{HHRS08}. We first prove a generalization of Theorem~\ref{th:main} (for connected critical graphs) by a reduction to the triangle-free case. We then apply this result to every component of an independence packing.

To describe the result, we first need to define a class of graphs that is closely related to bad graphs. From Lemma~\ref{le:badproperties}.\ref{le:badproperties_3} it follows that a bad graph contains precisely two $4$-vertex paths whose interior vertices have degree $2$; let us call them the \emph{nice} paths. A graph which is not isomorphic to the complete graph $K_4$ is called \emph{almost bad} if it can be obtained from a bad graph by contracting one or both of its nice paths to an edge. We remark that each such contraction creates a triangle, so almost bad graphs are not triangle-free. Furthermore, almost bad graphs are critical by Lemma~\ref{le:oddsubdivision}.

For a graph $G$, let $T(G)$ denote the maximum number of vertex-disjoint triangles in $G$. 
Note that if $G$ is subcubic, then two triangles are vertex-disjoint if and only if they are edge-disjoint. 
We define the following refinement of the measure $\lb(G)$: 
$$\lb_T(G):= \frac{6 \cdot|V(G)| -|E(G)| -2 \cdot T(G) -1}{12}.$$

With the new terminology in hand, we can now state our (partial) generalization of Theorem~\ref{th:main} for critical graphs. Note that it is sharp due to e.g.\ the triangle and dangerous graphs. 

\begin{theorem}\label{th:mainwithtriangles}
Let $G$ be a connected critical  subcubic graph which is not isomorphic to any of $K_4, F_{11}$, $F_{14}^{(1)}$, $F_{14}^{(2)}$, $F_{19}^{(1)}$, $F_{19}^{(2)}$, $F_{22}$. Then
\begin{itemize}
	\item $\alpha(G) = \lb_T(G) - \frac{1}{12}$ if $G$ is bad or almost bad\\[0.1ex]
	\item $\alpha(G) \geq \lb_T(G)$ otherwise.  
\end{itemize}
\end{theorem}
\begin{proof}
The bound for bad and almost bad graphs is immediate from Theorem~\ref{th:main} and Lemma~\ref{le:oddsubdivision}. Hence it suffices to show that $\alpha(G)\geq \lb_T(G)$ if $G$ is not bad and not almost bad. 
We apply induction on the number of triangles in $G$. The base case, no triangle, is certified by Theorem~\ref{th:main}. Let $t$ be a triangle in $G$. If $t$ is disjoint from every other triangle, then let $e$ denote an arbitrary edge of $t$. Otherwise, let $e$ denote the unique edge of $t$ that is shared with another triangle $t_2$. 
This edge is unique because $G$ is not isomorphic to $K_4$; for the same reason, every other triangle must be disjoint from $t$ and $t_2$. 
Let $G'$ be the graph obtained by subdividing the edge $e$ twice. By Lemma~\ref{le:oddsubdivision}, $G'$ is critical and $\alpha(G')=\alpha(G)+1$. Moreover: $T(G')=T(G)-1$, and $G'$ has two more vertices and two more edges compared to $G$. Hence $\lb_T(G')= \lb_T(G)+1$ and therefore $\alpha(G)-\lb_T(G) = \alpha(G')-\lb_T(G')$. By the definition of almost bad graphs, if $G'$ is bad or almost bad, then $G$ is almost bad. Thus the desired bound follows by induction.
\end{proof}

As a corollary, we obtain the following. Here $B(G)$ denotes the maximum number of vertex-disjoint subgraphs of $G$ that are bad or almost bad. 

\begin{corollary}\label{cor:connectedwithtriangles}
Let $G$ be a connected subcubic graph containing none of $K_4, F_{11}$, $F_{14}^{(1)}$, $F_{14}^{(2)}$, $F_{19}^{(1)}$, $F_{19}^{(2)}$, $F_{22}$ as a subgraph. Then
$$\alpha(G) \geq  \lb_T(G) - \frac{B(G)}{12}.$$
\end{corollary}
\begin{proof}
Consider an independence packing $\mathcal{P}=(G_1, \ldots, G_k)$ of $G$. Since $G$ is connected, at least $k-1$ edges of $G$ are not contained in any of $G_1,\ldots, G_k$. Thus $\lb_T(G)\leq \sum_{i=1}^{k} \lb_T(G_i)$. Moreover, by Theorem~\ref{th:mainwithtriangles} and the definition of independence packing: $\alpha(G)= \sum_{i=1}^{k} \alpha(G_i) \geq  \sum_{i=1}^{k} \lb_T(G_i) - \frac{B(G)}{12}$. 
\end{proof}

Corollary~\ref{cor:connectedwithtriangles} is sharp due to e.g.\ connected graphs that have an independence packing consisting entirely of triangles. 
However, we remark that Corollary~\ref{cor:connectedwithtriangles} is not quite best possible, in the sense that we did not take advantage of any integrality trick, as e.g.\ in the proof of Theorem~\ref{th:goal}. This is because such an integrality trick is not possible for every value of the tuple $(|V(G)|,|E(G)|, T(G))$.

\section{Conclusion}
\label{sec:conclusion}

Let $\chi_f(G)$ denote the {\em fractional chromatic number} of the graph $G$. 
Among the several equivalent definitions of this parameter, one (using duality) is as follows:
\[
\chi_f(G) := \sup_{{\mathbf w} \neq 0} \frac{\sum_{v \in V(G)} {\mathbf w}(v)}{\alpha(G, {\mathbf w})}, 
\]
where ${\mathbf w}$ is a weight function associating a nonnegative real weight ${\mathbf w}(v)$ to each vertex $v\in V(G)$, and $\alpha(G, {\mathbf w})$ denotes the maximum weight of an independent set in $G$. 
From this definition, it is clear that $\frac{n}{\alpha(G)} \leq \chi_f(G)$ when $G$ is an $n$-vertex graph. 
Accordingly, it is natural to ask which upper bounds on $\frac{n}{\alpha(G)}$ extend to upper bounds on $\chi_f(G)$. 
This is the case for Staton's $\frac{15}{4}$ upper bound for triangle-free subcubic graphs (Theorem~\ref{th:Staton}):

\begin{theorem}[Dvo\v{r}\'ak, Sereni, and Volec~\cite{DSV14}]
Let $G$ be a triangle-free subcubic graph. 
Then, $\chi_f(G) \leq \frac{14}{5}$.
\end{theorem}

Heckman and Thomas~\cite{HT06} conjectured that $\chi_f(G) \leq \frac{8}{3}$ for every triangle-free subcubic planar graph $G$, which would generalize their result (Theorem~\ref{th:planar}). 
We conjecture that Theorem~\ref{th:goal} can similarly be extended to fractional colorings: 

\begin{conjecture}
Let $G$ be a triangle-free subcubic graph containing none of $F_{11}$, $F_{14}^{(1)}$, $F_{14}^{(2)}$, $F_{19}^{(1)}$, $F_{19}^{(2)}$, $F_{22}$ as a subgraph. Then, $\chi_f(G) \leq \frac{8}{3}$.  
\end{conjecture}


\section*{Acknowledgments} 

The Stevin Supercomputer Infrastructure at Ghent University was used to compute upper bounds on $i(g)$, as well as for preliminary computations to collect evidence and to build intuition for an appropriate stronger induction hypothesis for our main theorem.

We are much grateful to the anonymous reviewers for their useful comments and suggestions, which helped us improve the paper. 

\bibliographystyle{amsplain}


\begin{aicauthors}
\begin{authorinfo}[wouter]
  Wouter Cames van Batenburg\\
  D\'epartement d'Informatique\\
  Universit\'e Libre de Bruxelles\\
  Brussels, Belgium\\
  wcamesva\imageat{}ulb\imagedot{}ac\imagedot{}be \\
  \url{http://homepages.ulb.ac.be/~wcamesva}
\end{authorinfo}
\begin{authorinfo}[jan]
  Jan Goedgebeur\\
  Department of Applied Mathematics, Computer Science and Statistics\\
  Ghent University\\
  Ghent, Belgium\\
  and\\
  Computer Science Department\\
  University of Mons\\
  Mons, Belgium\\  
  jan\imagedot{}goedgebeur\imageat{}ugent\imagedot{}be \\
  \url{https://caagt.ugent.be/jgoedgeb/}
\end{authorinfo}
\begin{authorinfo}[gwen]
  Gwena\"el Joret\\
  D\'epartement d'Informatique\\
  Universit\'e Libre de Bruxelles\\
  Brussels, Belgium\\
  gjoret\imageat{}ulb\imagedot{}ac\imagedot{}be \\
  \url{http://di.ulb.ac.be/algo/gjoret}
\end{authorinfo}
\end{aicauthors}

\end{document}